\UseRawInputEncoding
\documentclass[a4paper,oneside,11pt]{article}
\usepackage{amsmath}
\usepackage{amssymb}
\usepackage{titlesec}
\usepackage{graphicx}
\usepackage{fancyhdr}
\usepackage{latexsym}
\usepackage{mathrsfs}
\usepackage{booktabs}
\usepackage{mathrsfs}
\usepackage{mathtools}
\usepackage{amsthm}
\usepackage{wasysym}
\usepackage{cite}
\usepackage{color}
\usepackage{amsfonts}
\usepackage{pgf,pgfarrows,pgfnodes,pgfautomata,pgfheaps}
\usepackage{amsmath,amssymb}
\usepackage{graphics}
\usepackage{multimedia}
\usepackage{graphics}
\usepackage{cite}
\usepackage{latexsym}
\usepackage{amsmath}
\usepackage{amssymb}
\usepackage[dvips]{epsfig}
\usepackage{amscd}
\numberwithin{equation}{section}

\newtheorem{theorem}{Theorem}[section]
\newtheorem{lemma}{Lemma}[section]
\newtheorem{remark}{Remark}[section]
\newtheorem{proposition}{Proposition}[section]
\newtheorem{corollary}{Corollary}[section]

\newcommand{\ba}{\begin{aligned}}
\newcommand{\ea}{\end{aligned}}
\newcommand{\be}{\begin{equation}}
\newcommand{\ee}{\end{equation}}

\newcommand{\bnn}{\begin{eqnarray*}}
\newcommand{\enn}{\end{eqnarray*}}

\newcommand{\thatsall}{\hfill$\Box$}
\date{}

\title{Large-Time Behavior of the 2D Compressible Navier-Stokes System in Bounded Domains with Large Data and Vacuum}

 \author{Xinyu F{\small AN}$^{a}$, Jing L{\small I}$^{b, c}$, Xue W{\small ANG}$^{b}$ \thanks{Email addresses: fanxinyu17@mails.ucas.edu.cn (X. Y. Fan), ajingli@gmail.com (J. Li), xuewa@amss.ac.cn (X. Wang)}\\
{\normalsize a. Institute of Applied Physics and Computational Mathematics,}\\
{\normalsize Beijing 100088, P. R. China};\\
{\normalsize b.  School of Mathematical Sciences,}\\
{\normalsize  University of Chinese Academy of Sciences, Beijing 100049, P. R. China;} \\
{\normalsize c. Institute of Applied Mathematics, AMSS,}\\
{\normalsize \& Hua Loo-Keng Key Laboratory of Mathematics,}\\
{\normalsize  Chinese Academy of Sciences,    Beijing 100190, P. R. China}
}
\begin{document}
\maketitle
\begin{abstract}
The large time behavior of the unique strong solution to the barotropic compressible Navier-Stokes system is studied with large external forces and initial data, where the shear viscosity is a positive constant and the bulk one is proportional to a power of the density.    
  Some  uniform   estimates on the $L^p$-norm of the density are   established, and then deduce that the density   converges to its steady state   in $L^p$-spaces, which transforms the large external force into a small one in some sense. Moreover, to deal with the obstacles brought by boundary,  the conformal mapping and the pull back Green function are applied to give a point-wise representation of the effective viscous flux, and then make use of slip boundary conditions to cancel out the singularity.   
\end{abstract}

\section{Introduction and main results}\label{SEC1}
$\quad$ We study the barotropic compressible Navier-Stokes system with the external force in a bounded two-dimensional (2D) domain $\Omega$:
\begin{equation}\label{S101}
\begin{cases}
\rho_t+\mathrm{div}(\rho u)=0, \\
(\rho u)_t+\mathrm{div}(\rho u\otimes u)
-\mu\Delta u-\nabla\big((\mu+\lambda)\mathrm{div} u\big)
+\nabla P=\rho\nabla f,
\end{cases}
\end{equation}
where $t\geq 0$ is time, $x\in \Omega$ is the spatial coordinate; $\rho=\rho(x, t)$ and $u=\big(u_1(x, t), u_2(x, t)\big)$ represent the unknown density and velocity respectively, and $f$ is a given external force.
The pressure $P$ is given by
\begin{equation}\label{S102}
P=a\rho^\gamma,\  a>0,\  \gamma>1.
\end{equation}
While the shear viscosity $\mu$ and the
bulk viscosity $\lambda$ satisfy:
\begin{equation}\label{S103}
0<\mu=\mathrm{constant}, \ \lambda(\rho)=b\rho^\beta,
\end{equation}
where $b$ and $\beta$ are positive constants. Without lost of generality, let us set
$a=b=1.$

The system is subject to the given initial data
\begin{equation}\label{S104}
\rho(x, 0)=\rho_0(x), \ \rho u(x, 0)=m_0(x), \ x\in\Omega.
\end{equation}
Moreover, we mainly investigate two types of boundary conditions.

\textit{Case 1.} Suppose that $\Omega\subset\mathbb{R}^2$ is bounded $C^{2,1}$ domain, and we study the Navier-slip boundary conditions:
\begin{equation}
u\cdot n=0,\ \ \mathrm{rot} u=-Ku\cdot n^\bot\ \mathrm{on}\ \partial\Omega,  \label{S105}
\end{equation}
where $K$ is a non-negative smooth function on $\partial \Omega$, $n=(n_1,n_2)$ denotes the unit outer normal vector on  $\partial \Omega$, while $n^\bot$ is the unit tangential vector on  $\partial \Omega$ denoted by
\be\label{S106} n^\bot\triangleq (n_2,-n_1).\ee
In addition, we us extend $n$, $n^\bot$, and $K$ smoothly to $\overline\Omega$.  

\textit{Case 2.} We also consider the periodic domain $\Omega=(0,1)\times(0,1)$, and the solution $(\rho, u)$ are functions with period 1 in both directions.

There is a large number of literature about the strong solvability for the multi-dimensional compressible Navier-Stokes system with constant viscosity coefficients.  The history of this area may trace back to Nash \cite{Nash} and Serrin \cite{ser1},  who established the local existence and uniqueness of classical solutions respectively with the density away from vacuum. The first result of global classical solutions was due to Matsumura-Nishida \cite{1980The} for the initial data close to a non-vacuum equilibrium in $H^{3}.$ Hoff \cite{1995Global,hoff2005} then studied the problem with discontinuous initial data and introduced a new type of a priori estimates on the material derivative $\dot{u}$. The major breakthrough in the frame of weak solutions was due to Lions \cite{1998Mathematical}. Under the assumption that the initial energy is finite, he obtained the global existence of weak solutions with vacuum in 3D domains provided the adiabatic constant $\gamma\ge\frac{9}{5}$, which was further released to $\gamma>\frac{3}{2}$ by Feireisl et al\cite{feireisl2004dynamics}. Recently, Huang-Li-Xin \cite{hlx21} and Li-Xin \cite{lx01} established the global existence and uniqueness of classical solutions in 3D and 2D whole spaces, where the initial energy was small enough and the density was allowed to vanish and even had compact support initially. More recently, Cai-Li \cite{caili01} extended above results to general bounded domains with the velocity field satisfying the Navier-slip boundary conditions.

The above results were mainly concerned with the model with constant viscosity coefficients and  the small initial data. Positive results for the general model without restriction on the size of initial data and external force are rather fewer.

We first mention that, according to the assertions given by Liu-Xin-Yang \cite{LXY}, if we transform Boltzmann equations into Navier-Stokes equations by  the Chapman-Enskog expansions, the viscosity will depend on the temperature $\theta$. Especially, the  barotropic flow admits that $$\theta\sim\rho^{\gamma-1},$$ thus we arrive at the model with viscosity depending on the
density. Moreover, Lions \cite{1998Mathematical} indicated that such model was close to the shallow water system. Indeed, it appears naturally in the geophysical flow as well, see \cite{BD1, BD2, BD3, BD4, BD5} for example. Therefore, it is of great importance to study the compressible Navier-Stoke equations with the density-dependent viscosity. In fact,
the system \eqref{S101} is an important example of such type introduced by Vaigant-Kazhikov .

In particular, Vaigant-Kazhikov \cite{vaigant1995}   obtained the unique global strong solution without vacuum for \eqref{S101}--\eqref{S105} in periodic domains, 
under the assumption that $$\beta>3.$$  Later, Perepelitsa \cite{PP} showed the uniform upper bound of the density and the large time behavior under the  additional  restriction $$\gamma=\beta=3.$$  Jiu-Wang-Xin \cite{jwx} then generalized the result in \cite{vaigant1995} to the case containing vacuum.
Recently, for the problem \eqref{S101} in the periodic domains or the whole space with the density allowed to vanish, Huang-Li \cite{HL,hl21} (see \cite{jwx1} also) relaxed the crucial condition $\beta>3$ to  $$\beta>\frac{4}{3}.$$ As indicated by \cite{Fri}, this result seems to be the optimal result for general data up to now even for the periodic case. While for spherical symmetric case, Li-Xin-Zhang \cite{LXZ} studied the free boundary problem and obtained the unique global strong solution, under the restriction that $\beta>1$.

An important problem closely related with the above existence assertions is the large time behavior of the solution, especially for the system under the action of the potential external force $F=\nabla f$, whose steady state in no more constant.  Feireisl-Petzeltová \cite{FP} and Novotny-Straškraba \cite{NS} proved that the density of any global weak solutions converged to the steady state in $L^p$ sense as $t\rightarrow\infty$. Matsumura-Padula \cite{1992Stability} derived both the existence of the unique classical solution to the problem \eqref{S101} and the stability of the steady state. Huang-Li-Xin \cite{HLX} established the large time behavior of strong and classical solutions to the 2D Stokes approximation equations  without any restriction on the size of initial date and external force.

Based on the global existence and uniqueness of the strong solution to the problem \eqref{S101}--\eqref{S105} in general simply connected domains in our previous paper \cite{FLL}, the aim of the present paper is to study the large time behavior of this solution.

Before stating the main results, we explain the notations and conventions used throughout the paper.
For a positive integer $k$ and constant $1 \leq p \leq \infty$,  the standard $L^p$ spaces and Sobolev ones are denoted as follows:
\begin{equation*}\begin{cases}
L^p=L^p(\Omega), \,\, W^{k, p}=W^{k, p}(\Omega), \,\, H^k=W^{k, 2}(\Omega) ,\\
\|f\|_{L^p}=\|f\|_{L^p(\Omega)}, \,\,\|f\|_{W^{k,p}}=\|f\|_{W^{k,p}(\Omega)},\,\,\|f\|_{H^k}=\|f\|_{W^{k,2}(\Omega)},\\
\tilde{H^1}=\{u\in H^1(\Omega)|u\cdot n=0,\mathrm{rot}u=-Ku\cdot n^\bot\ \text{on} \,\partial\Omega\}.
\end{cases}\end{equation*}
The transpose of gradient, the material derivative and the integral average of $f$ are respectively given by
\begin{equation}\notag
\nabla^{\bot}\triangleq(\partial_2, -\partial_1),\ \ \  \frac{d}{dt}f=\dot{f}\triangleq\frac{\partial}{\partial t}f+u\cdot\nabla f, \ \ \
\bar{f}=\frac{1}{|\Omega|}\int_\Omega fdx.
\end{equation}
For convenience, we set $$\omega=\mathrm{rot}u=\partial_2u_1-\partial_1u_2.$$

Note that the steady state
$\rho_s$ of the system \eqref{S101} solves the following equation:
\begin{equation}\label{S107}
\begin{cases}
\nabla \rho_s^\gamma=\rho_s\nabla f\quad\
\mbox{ in }\Omega,\\
\int_\Omega\rho_sdx=\int_\Omega\rho_0dx.
\end{cases}
\end{equation}
As indicated by Matsumura-Padula \cite{1992Stability}, the system \eqref{S107} is uniquely solvable under the condition that
\begin{equation}\label{CC101}
\int_\Omega\rho_0dx> \int_\Omega\left(\frac{\gamma-1}{\gamma}\bigg( f-\inf_{\overline{\Omega}}f\bigg)\right)^{\frac{1}{\gamma-1}}dx.
\end{equation}
More discussions on the steady state can be found in  Section \ref{SS22}.

We will focus on the Navier-slip boundary conditions, and the main results can be stated as below. 
\begin{theorem}\label{T11}
Let $\Omega\subset\mathbb{R}^2$ be a simply connected  $C^{2,1}$ bounded domain and
\begin{equation}\label{S110}
\beta>\frac{4}{3},\ \gamma>1.
\end{equation}
For given $q>2$ and $f\in H^2$,  suppose that the initial data $(\rho_0,u_0)$ satisfies
\begin{equation*}
\begin{split}
\rho_0\in W^{1,q},~~u_0\in\tilde{H}^1,~~ m_0=\rho_0u_0,\\
\end{split}
\end{equation*}
then the initial boundary value problem \eqref{S101}--\eqref{S105} admits a unique global strong solution $(\rho,u)$ in $\Omega\times[0,\infty)$.

Moreover, if the initial density $\rho_0$ meets \eqref{CC101}, then
the global solution $(\rho, u)$ obtained above has following properties:

1) The density $\rho$ is uniformly bounded. In precise, there is a constant $C$ determined by $\beta,\ \gamma,\ \mu,\ \rho_0$, and $ u_0$ such that
\begin{equation}\label{S111}
\sup_{0\leq t\leq T}\|\rho(t)\|_{L^\infty}\leq C.
\end{equation}

2) The strong solution $(\rho, u)$ converges to the steady state $(\rho_s, 0)$ in exponential rate. Precisely, $\forall p\in[1,\infty)$, there exist constants $C$ and $\xi$ depending only on $p,\ \beta,\ \gamma,\ \mu,\ \rho_0$, and $ u_0$ such that, for any $t\in (1,\infty),$
\begin{equation}\label{S112}
\|\rho(\cdot, t)-\rho_s\|_{L^p}+\|\nabla u(\cdot, t)\|_{L^p}\leq Ce^{-\xi t}.
\end{equation}

\end{theorem}

In the same manners, the above conclusions are valid for periodic domain without the external force as well.
\begin{theorem}\label{T12}
Let $\Omega=[0,1]\times[0,1]$ be the periodic domain. Suppose \eqref{S110} is valid and the external force $\nabla f\equiv 0$. If the periodic initial data $(\rho_0,u_0)$ satisfies
$$\rho_0\in W^{1,q},~~u_0\in H^1,~~m_0=\rho_0u_0,$$
for some $q>2$, then the initial boundary value problem \eqref{S101}--\eqref{S104} admits a unique global periodic strong solution $(\rho, u)$ in $\Omega\times[0,\infty)$. Moreover, the large time behavior \eqref{S111} and \eqref{S112} are true as well.
\end{theorem}

\begin{remark}
Theorem \ref{T11} establishes the large time behavior of the strong solution obtained in \cite{FLL}. It indicates that the density $\rho$ is uniformly bounded and the solution $(\rho,u)$ converges to the steady state $(\rho_s,0)$ in $L^p$ sense with exponential rate.
\end{remark}

\begin{remark}
Theorem \ref{T11} requires no limitation on the size  of both initial data and external force. It is a result concerning the large time behavior of compressible Navier-Stokes equations under the large data.
\end{remark}

\begin{remark}
Our method of proving Theorem \ref{T11} relies heavily on the simply connectedness of the domain. For general multi-connected domain, related problems will be left for the future.
\end{remark}

\begin{remark}
Theorem \ref{T12} only investigates the periodic case without external force. The major advantage is that the total momentum $\overline{\rho u}$ and the steady state $(\rho_s,0)=(\bar{\rho},0)$ remain constant. The general case with large external force $\nabla f$ is left for future.
\end{remark}

\begin{remark}
Compared with Huang-Li \cite{HL}, the restriction of $\beta$ and $\gamma$ in Theorem \ref{T12} is further released from $\beta>\frac{3}{2},\ 1<\gamma<4\beta-3$ to $\beta>\frac{4}{3},\ \gamma>1$. As indicated by \cite{Fri}, $\beta>\frac{4}{3}$ seems to be the optimal index one may expect up to now even for the existence result on the periodic domain.
\end{remark}

We now make some comments on the analysis of the paper. In order to derive the large time behavior, we must make all estimates involved independent with time $t$.

As indicated by our previous paper \cite{FLL}, the upper bound of the density still lies in the central position of the whole argument, which is closely related with the point-wise control of the effective viscous flux $G$, with
$$G=(2\mu+\lambda)\mathrm{div}u-(P-P_s).$$
However, the key ingredient to handle $G$ is the extra integrability of velocity, say, we need to find some small constant $\nu>0,$ such that
\begin{equation}\label{S113}
\sup_{0\leq t\leq T}\int_\Omega\rho|u|^{2+\nu}dx\leq C.
\end{equation}
Such assertion relies heavily on the $L^p$ estimates of the density $\rho$.

1. Thus, the central technical issue is to obtain the uniform $L^p$ estimates of $\rho$, which improves the previous result due to Vaigant-Kazhikov \cite{vaigant1995}.
 Precisely, for $\theta(\rho)=2\mu\log\rho+\frac{1}{\beta}\rho^\beta$ and $F=(2\mu+\lambda)\mathrm{div}u-P$, let us rewrite  \eqref{S101}$_1$ as
\begin{equation*}
\frac{d}{dt}\theta(\rho)+\rho^\gamma=-F.
\end{equation*}
Roughly speaking, $\rho^\gamma$ can be viewed as a damping term, which cancels out the effect of external force and boundary, thereby guarantees the uniform $L^p$ estimates of the density, see Lemma \ref{L32}. It should be noted that Vaigant-Kazhikov \cite{vaigant1995} dropped out this key damping term without making full use of it.


Moreover, since the domain under consideration is bounded, we only need to consider the larger part of the density, thus we introduce a truncation function $(\rho-M)_+$ for some properly determined $M$ to overcome technical problems in our arguments. These rough estimates of the density are indispensable in further arguments, see Subsections \ref{SS32} and \ref{SEC4}.

2. The second problem arises from the large external force, and we apply the ideas in \cite{HLX} to handle it. Briefly speaking, by multiplying \eqref{S101}$_2$ by $\rho_s^{-1}\mathcal{B}(\rho-\rho_s)$, we can expand the pressure term $P$ into the Taylor series of order 2 (one more order compared with multiplying $\mathcal{B}(\rho-\rho_s)$), which transforms the large external force into the small one, once we show 
$$\lim_{t\rightarrow\infty}\|\rho(\cdot, t)-\rho_s\|_{L^p(\Omega)}=0,$$ 
by the same method as \cite{HLX}. These arguments lead to the global integrability of $\|\rho-\rho_s\|_{L^2}$ as well, which gives a crude description of the decay rate of the density $\rho$.
\begin{equation}\notag\label{S114}
\int_0^T\int_\Omega(\rho+1)^{\gamma-1}(\rho-\rho_s)^2dxdt\leq C.
\end{equation}
We mention that, $\gamma>1$ provides a crucial extra order in our calculation, see \eqref{S338}.

3. The last obstacle comes from the boundary term, which will be overcame by our previous arguments \cite[Lemmas 12 \& 13]{FLL}. By virtue of the pull back Green function $\widetilde{N}(x,y)$ (see Lemma \ref{L46}), we write out $G$ explicitly and carefully study the singularity of the representation to properly bound $\|G\|_{L^\infty}$, see \eqref{S440}. Note that compared to \cite{FLL}, to get uniform estimates, we require some detailed calculations about $\|\nabla u\|_{L^{p}}$ (see Proposition \ref{L43}), which improves our previous estimate about it in \cite[Lemma 11]{FLL}.

Moreover, to handle the boundary term in lower order estimates, especially, the trace of  $\nabla u$ on $\partial \Omega$, we will apply the ideas due to Cai-Li\cite{caili01}. Observe that, on the one hand, the slip boundary condition $u\cdot n\vert_{\partial \Omega}=0$ gives
\begin{equation*}
  u=(u\cdot n^\bot)n^\bot, \,\, (u\cdot\nabla)u\cdot n=-(u\cdot\nabla)n\cdot u,
\end{equation*}
where $n^\bot$ is given by \eqref{S106}. On the other hand, as shown by  Cai-Li\cite{caili01}, for smooth  functions $g$ and $h$, \eqref{CCC401} ensures that
\begin{equation*}
\begin{split}
\int_{\partial\Omega} n^\bot\cdot\bigg(\nabla g\cdot h \bigg)dS
=\int_{\Omega}\nabla^\bot g\cdot \nabla h\,dx\leq C\|g\|_{H^1}\|h\|_{H^1}.
\end{split}
\end{equation*}

4. For periodic domain $[0,1]\times[0,1]$, the outline of arguments remains the same as Navier-slip boundary condition. However, our method seems not valid directly for the general large external force. The problem is lack of the integrability for $\|u\|_{L^p}$ with respect to $t$, since
$$\|u\|_{L^p}\leq C(p)\|\nabla u\|_{L^2}$$
is \textbf{not} true for the periodic domain. According to Poincar\'{e}'s inequality, we must module out the integral average to ensure that
\begin{align}\label{C101}
\|u-\bar{u}\|_{L^p}\leq C(p)\|\nabla u\|_{L^2}.
\end{align}
For present case, the total momentum $\overline{\rho u}$ is a good alternative for $\bar{u}$, since according to the conservation of momentum, we have
\begin{align}\label{C102}
\frac{d}{dt}\int_\Omega\rho u\,dx=0,
\end{align}
which is obtained by integrating \eqref{S101}$_2$ on $\Omega$ and make sure $\overline{\rho u}$ is a constant.

We mention that \eqref{C102} is sufficient for the periodic case without external force, see Appendix I. However, for general large external force, \eqref{C102} transforms into
\begin{align}\label{C103}
\frac{d}{dt}\int_\Omega\rho u\,dx=\int(\rho-\rho_s)\nabla f\,dx.
\end{align}
It seems hard to extract useful information from \eqref{C103} on decay rate of $\overline{\rho u}$ directly, which is the major obstacle for the periodic case containing the large external force.

The rest of the paper is organized as follows: Section \ref{SEC2} recalls some elementary properties and lemmas; then in Section \ref{SEC3}, we derive the crucial uniform $L^p$ estimates of the density $\rho$ and give two applications of it; with all preparations done, we finish our main  Theorem \ref{T11} via two steps in Section \ref{S5}.

We only treat the Navier-slip boundary conditions in the main text. Necessary modifications for the periodic domain and the proof of Theorem \ref{T12} will be sketched in Appendix I.

\section{Preliminary}\label{SEC2}
\quad In Subsection \ref{SS21}, we give some basic lemmas which will be used frequently. Then in Subsection \ref{SS22}, we state the existence and uniqueness assertion for the system \eqref{S101}--\eqref{S105} due to \cite{FLL} and make a short discussion about the steady state.
\subsection{Technical issue}\label{SS21}

\quad We first show a modified version of the classical Poincar\'{e}-Sobolev inequality.
\begin{lemma}[\cite{tal1}]\label{L22}
 There exists a positive constant $C$ depending only on $\Omega$ such that every function $u\in H^1(\Omega)$ satisfies, for $2<p<\infty$,
\begin{equation}
\|u\|_{L^p}\leq Cp^{1/2}\|u\|^{2/p}_{L^2}\|u\|^{1-2/p}_{H^1}. \label{S202}
\end{equation}
In particular, $\|u\|_{H^1}$ can be replaced by $\|\nabla u\|_{L^2}$ provided
$$ u\cdot n|_{\partial\Omega}=0 \mbox{ or }\int_\Omega udx=0.$$
\end{lemma}
Next, we list some elliptic estimates of Cauchy-Riemann equations (see \cite{adn} for more details).
\begin{lemma}\label{L23}
Suppose that the pair $(F,\omega)$ satisfies
\begin{equation*}
\left\{ \begin{array}{l}
\nabla F+\nabla^\bot\omega= f \quad\quad\quad\,\, \textit{in $\Omega$},\\
\int_\Omega Fdx=0,\ \quad\omega=g \quad
\textit{on $\partial\Omega$}.
       \end{array} \right. \\
\end{equation*}
Then we declare that:

1) If $f\in L^p(\Omega)$ and $g=0$, then we can find some constant $C(p)$ depending only on $p$ and $\Omega$ such that
\begin{equation}\label{S203}
\|\nabla F\|_{L^p}+\|\nabla\omega\|_{L^p}
\leq C(p)\|f\|_{L^p}.
\end{equation}

2) If $u\in W^{1,p}(\Omega)$, $f=0$, and
$g=u\big|_{\partial\Omega}$, then we can find some constant $C(p)$ depending only on $p$ and $\Omega$ such that
\begin{equation}\label{S205}
\| F\|_{W^{1,p}}+\| \omega\|_{W^{1,p}}
\leq C(p)\|u\|_{W^{1,p}}.
\end{equation}

3) If $u\cdot v\in L^p(\Omega)$, $f=\mathrm{div}(u\otimes v)=\partial_i(u_i\cdot v_j)$ with $u\cdot n=0$ on $\partial\Omega$,  and $g=0$, then we can find some constant $C(p)$ depending only on $p$ and $\Omega$ such that
\begin{equation}\label{S204}
\| F\|_{L^p}+\| \omega\|_{L^p}
\leq C(p)\|u\cdot v\|_{L^p}.
\end{equation}
\end{lemma}
We mention that the proofs for \eqref{S203} and \eqref{S205} can be found in \cite[Chapter IV]{adn}. We will provide a proof for assertion 3) in Appendix II.

In order to obtain the proper estimates upon $\nabla u$, we require the div-rot-type control, which can be found in \cite{Aramaki2014Lp,
Mitrea2005Integral,Wahl1992Estimating}.
\begin{lemma}\label{L24}
Let $\Omega$ be a simply connected bounded domain in $\mathbb{R}^2$ with Lipschitz boundary $\partial\Omega$.
Then for $u\in W^{1, p}\ (1<p<\infty)$ and $u\cdot n=0$ on $\partial\Omega$, we can find some constant $C(p)$ determined by $p$ and $\Omega$ such that
\begin{equation}
\|\nabla u\|_{L^p}\leq C(p)(\|\mathrm{div}u\|_{L^p}+\|\mathrm{rot}u\|_{L^p}). \label{S206}
\end{equation}
\end{lemma}
In the spirit of Lemma \ref{L24}, we can extend \eqref{S206} to a more general form.
\begin{corollary}\label{L42}
Under the settings of Lemma \ref{L24}, there exist positive constants $\tilde \nu$ and $  C$  both depending only on $\Omega$ such that, for any $\nu\in (0,\tilde \nu)$,
\begin{equation}\label{nau}
\int_\Omega|u|^{\nu}|\nabla u|^2dx\leq C\int_\Omega|u|^{\nu}\bigg((\mathrm{div}u)^2+\omega^2\bigg)dx.
\end{equation}
\end{corollary}
\begin{proof} First, direct computation  shows
\begin{equation*}
2|u|^\nu |\nabla u|^2\le 4|\nabla(|u|^{\nu/2}u)|^2+\nu^2|u|^\nu|\nabla u|^2.
\end{equation*}
By setting $\nu<1$, it implies that
\begin{equation}\label{S409}
|u|^\nu|\nabla u|^2\le 4|\nabla(|u|^{\nu/2}u)|^2.
\end{equation}

Next, let us calculate that
\begin{equation*}\ba
(\mathrm{div}(| u| ^{\nu/2} u))^2&=\big(| u| ^{\nu/2} \mathrm{div} u+\nabla | u| ^{\nu/2}\cdot u\big)^2\\&
\leq C| u| ^\nu \big(\mathrm{div} u\big)^2+ C\nu^2| u| ^\nu| \nabla u| ^2 .\ea
\end{equation*}
Similarly, we have
\begin{equation}\label{S411}
\big(\mathrm{rot}(| u| ^{\nu/2} u)\big)^2
\leq C| u| ^\nu \big(\mathrm{rot} u\big)^2+C\nu^2| u| ^\nu| \nabla u| ^2.
\end{equation}

Finally, since $\left.| u| ^{\nu/2}u\cdot n\right| _{\partial\Omega}=0$, we apply \eqref{S206} to deduce
\begin{equation*}
\int_\Omega| \nabla (| u|^{\nu/2}u)| ^2dx\leq
C\int_\Omega\left(| \mathrm{div}(| u| ^{\nu/2}u)| ^2 +| \mathrm{rot}(| u| ^{\nu/2}u)| ^2\right)dx,
\end{equation*}
which together with  \eqref{S409}--\eqref{S411}  shows \eqref{nau} after choosing $\nu<1$ suitably small.
\end{proof}

For bounded domain $\Omega$, we introduce the following ``inversion" operator of $\mathrm{div}$, which can be found in \cite{caili01}.
\begin{lemma}\label{L25}
We define
\begin{equation*}
L^p_0(\Omega)=\left\{f\big|\|f\|_{L^p(\Omega)}<\infty,\int_\Omega fdx=0\right\}.
\end{equation*}
Then, for $1<p<\infty$, there is a bounded linear operator $\mathcal{B}$ given by
\begin{equation*}
\begin{split}
\mathcal{B}:L^{p}_0&\rightarrow W^{1,p}_0,\\
\quad f&\mapsto \mathcal{B}(f),\\
\end{split}
\end{equation*}
such that $u=\mathcal{B}(f)$ is a solution to the  equation below,
\begin{equation}\label{S207}
\begin{cases}
\mathrm{div}u=f&\ \mbox{ in }\Omega,\\
u=0&\ \mbox{ on }\partial\Omega.
\end{cases}
\end{equation}
Moreover, we have following properties:

i). For $1<p<\infty$, there is a constant $C(p)$ depending on $\Omega$ and $p$, such that
\begin{equation*}
\|\mathcal{B}(f)\|_{W^{1,p}}\leq C(p)\|f\|_{L^p}.
\end{equation*}

ii). In particular, when $f=\mathrm{div}g$ with $g\cdot n=0$ on $\partial\Omega$, $\mathcal{B}(f)$ is well defined and satisfies that
\begin{equation*}
\|\mathcal{B}(f)\|_{L^p}\leq C(p)\|g\|_{L^p},
\end{equation*}
and $u=\mathcal{B}(f)$ is a weak solution to \eqref{S207}.
\end{lemma}

We end this subsection by the next Zlotnik inequality. It plays an important role in deriving the uniform upper bound of the density $\rho$.
\begin{lemma}[\cite{zlt}]\label{L26}
Let the function $y\in W^{1, 1}(0, T)$ satisfies
\begin{equation*}
y'(t)=g(y)+h'(t)\ on\ [0, T], \ y(0)=y^0,
\end{equation*}
with $g\in C(\mathbb{R})$ and $h\in W^{1, 1}(0, T)$.  If $g(\infty)=-\infty$ and
\begin{equation*}
h(t_2)-h(t_1)\leq N_0+N_1(t_2-t_1),
\end{equation*}
for all $0\leq t_1 < t_2\leq T$ with some $N_0 \geq 0$ and $N_1 \geq 0$,  then
\begin{equation*}
y(t)\leq \max{\{y^0, \bar\zeta\}}+ N_0<\infty\ on\ [0, T],
\end{equation*}
where $\bar\zeta$ is a constant such that
\begin{equation*}
g(\zeta)\leq-N_1\ for\ \zeta\geq\bar\zeta.
\end{equation*}
\end{lemma}
\subsection{Existence results}\label{SS22}
$\quad$We first quote our previous result in \cite{FLL} about the global existence of the unique strong solution to the initial boundary problem \eqref{S101}--\eqref{S105}.
\begin{lemma} \label{L21}
Under the conditions in Theorem \ref{T11}, the initial boundary problem \eqref{S101}--\eqref{S105} has the unique global strong solution $(\rho,u)$ satisfying for any $0<T<\infty$,
\begin{equation*}
\left\{ \begin{array}{l}
\rho\in C([0,T];W^{1,q}),\rho_t\in L^\infty(0,T;L^2),
\\
u\in L^\infty(0,T;H^1)\cap L^{1+1/q}(0,T;W^{2,q}),\\
\sqrt{t}u\in L^\infty(0,T;H^2)\cap L^{2}(0,T;W^{2,q}),
\sqrt{t}u_t\in L^2(0,T;H^1),\\
\rho u\in C([0,T];L^2),\sqrt{\rho}u_t\in L^2((0,T)\times\Omega).
       \end{array} \right. \\
\end{equation*}
\end{lemma}

Note that we only considered the case of $f=0$ in \cite{FLL}. However, the same proof is valid for   $f\in H^2$ as well, so we omit the complete details here.

In addition, the standard energy estimate involving potential force holds as well.
\begin{proposition}\label{L31}
For the unique global solution $(\rho, u)$ given by Lemma \ref{L21}, there is a constant $C$ depending only on $\gamma$,  $\rho_0$, $u_0$, and $f$ such that
\begin{equation}\label{S301}
\sup_{0\leq t\leq T}\int_\Omega(\rho|u|^2+\rho^\gamma)dx+\int_0^T A_1^2(t)dt\leq C,
\end{equation}
where we define
\begin{equation}\label{SS401}
A_1^2(t)\triangleq  \int_\Omega\left((\mu+\lambda)(\mathrm{div}u)^2+|\nabla u|^2\right)dx.
\end{equation}

\end{proposition}
\begin{proof}

Note that in dimension 2, we have
$$\Delta u=\nabla \mathrm{div}u+\nabla^\bot\omega,$$
thus $\eqref{S101}_2$ can be rewritten as
\begin{equation}\label{S303}
\rho\dot{u}-\nabla\big((2\mu+\lambda) \mathrm{div}u\big)-\mu\nabla^\bot\omega+\nabla P-\rho\nabla f=0.
\end{equation}

Multiplying \eqref{S303} by $u$ and integrating over $\Omega$ show that
\begin{equation}\label{S304}
\begin{split}
&\frac{d}{dt}\int_\Omega\left(\frac12\rho|u|^2+ \frac{\rho^\gamma}{\gamma-1}-\rho f\right)dx
+\int_\Omega\left((2\mu+\lambda)(\mathrm{div}u)^2+\mu\omega^2\right)dx\\
&= -\int_{\partial\Omega}K(u\cdot n^\bot)^2dS\leq 0,
\end{split}
\end{equation}
where we have used \eqref{S105}, $\eqref{S101}_1$, and the following estimate:
\begin{equation*}
\begin{split}
-\int_\Omega \rho u\cdot\nabla f dx=\int_\Omega\mathrm{div}(\rho u)f dx=-\int_\Omega\frac{\partial}{\partial t}\rho\cdot f dx=-\frac{d}{dt}\int_\Omega\rho f dx
\end{split}
\end{equation*}
due to \eqref{S105} and $\eqref{S101}_1$.

Moreover, by virtue of $\eqref{S101}_1$, it holds that
\begin{equation}\label{S305}
\left|\int_\Omega\rho fdx\right|\leq \|f\|_{L^\infty}\int_\Omega\rho dx=\|f\|_{L^\infty}\int_\Omega\rho_0dx\leq C.
\end{equation}
Thus, integrating \eqref{S304} over $(0, T)$ together with \eqref{S305} and \eqref{S206} yields \eqref{S301}.
\end{proof}

Next, we discuss the steady state of the system \eqref{S101}. For potential external force $\nabla f$, we mention that the steady state of the system \eqref{S101} is given by $(\rho_s, 0)$ with $\rho_s$ satisfying:
\begin{equation}\label{CC201}
\begin{cases}
\nabla \rho_s^\gamma=\rho_s\nabla f,\\
\int_\Omega\rho_s\, dx=\int_\Omega\rho_0\, dx.&
\end{cases}
\end{equation}
The uniqueness of the steady state is given by Matsumura-Padula \cite{1992Stability}:
\begin{lemma}\label{L11}
Let $f\in H^2$, and we assume that
\begin{equation}\label{S109}
\int_\Omega\rho_sdx\geq \int_\Omega\left(\frac{\gamma-1}{\gamma}\bigg(f-\inf_{\overline{\Omega}}f\bigg)\right)^{\frac{1}{\gamma-1}}dx.
\end{equation}
Then the equation \eqref{S107} has a unique solution   $0<\rho_s\in H^2(\Omega)$.
\end{lemma}

We end this section by a direct corollary giving explicit formula of $\rho_s$.
\begin{corollary}\label{C21}
Under the condition of Lemma \ref{L11}, the solution $\rho_s$ is given by
\begin{equation}\label{CC102}
\rho_s=\bigg(\frac{\gamma-1}{\gamma}(f+C_0)\bigg)^\frac{1}{\gamma-1},
\end{equation}
for some constant $C_0$ determined by the condition \eqref{S109}.
\end{corollary}
\begin{proof}
The proof is quite direct. Note that Lemma \ref{L11} ensures that $\rho_s>0$, thus we multiply \eqref{CC201}$_1$ by $\rho_s^{-1}$ and deduce that
$$\nabla\bigg(\frac{\gamma}{\gamma-1}\rho_s^{\gamma-1}-f\bigg)=0.$$

Consequently, we argue that
\begin{equation}\label{CCC101}
\frac{\gamma}{\gamma-1}\rho_s^{\gamma-1}=f+C_0,
\end{equation}
for some $C_0$ determined by \eqref{S109} and thus gives \eqref{CC102}.
\end{proof}

\section{Rough bound of the density}\label{SEC3}
\quad This section is devoted to obtain the crucial uniform $L^p$ estimates of the density $\rho$. It provides a crude restriction on the size of $\rho$ and prevents it growing too fast as $t\rightarrow\infty$, see Lemma \ref{L32} and Corollary \ref{L33}. Then we state two applications of roughly bounded density. One concerns the accurate large time behavior of $\rho$, and another one gives further energy estimates.

In this section, we always assume $(\rho,u)$ is the unique global strong solution of initial boundary value problem \eqref{S101}--\eqref{S105} obtained in Lemma \ref{L21}. The constant $C$ may vary from line to line, but it is independent of $t$.
\subsection{Uniform $L^p$ estimate of density}
\quad We first deduce the uniform $L^p$-norm of density $\rho$.
Let us rewrite  $\eqref{S101}$ as
\begin{equation}\label{S306}
\frac{d}{dt}\theta(\rho)+P(\rho)=-(F-\bar{F})-\bar{F},
\end{equation}
where $\theta(\rho)=2\mu\log\rho+\rho^\beta/\beta$, $F=(2\mu+\lambda)\mathrm{div}u-P$
and $\bar{F}$ is the integral average of $F$. Note that $\int_\Omega(F-\bar{F})dx=0.$

According to the conservation of the momentum \eqref{S101}$_2$, we observe that $F-\bar{F}$ and $\omega=\partial_2u_1-\partial_1u_2$ solve the Cauchy-Riemann equations:
\begin{equation}\label{S307}
\begin{cases}
\nabla (F-\bar{F})+\nabla^\bot\omega=\frac{\partial}{\partial t}(\rho u)+\mathrm{div}(\rho u\otimes u)
-\rho\nabla f& \mbox{ in } \Omega, \\
\int_\Omega (F-\bar{F})dx=0,\ \omega=-Ku\cdot n^\bot& \mbox{ on } \partial\Omega.
\end{cases}
\end{equation}
By virtue of the uniqueness assertion of problem \eqref{S307} (see \cite{adn}), we decompose \eqref{S307} into four parts  $$F-\bar{F}=\sum_{i=1}^4F_i,\
 \omega=\sum_{i=1}^4\omega_i.$$

Precisely, $(F_1,\omega_1)$ deals with the parts of time derivatives, and it is given by
\begin{equation}\label{S309}
\begin{cases}
\nabla F_1+\nabla^\bot\omega_1=\frac{\partial}{\partial t}(\rho u)&\mbox{in }\Omega, \\
\int_\Omega F_1dx=0,\,\omega_1=0&\mbox{on }  \partial\Omega.
\end{cases}
\end{equation}
In particular, if we introduce the related system
\begin{equation}\label{S313}
\begin{cases}
\nabla\tilde{F}_1+\nabla^\bot\tilde{\omega}_1=\rho u
&\mbox{in }\Omega,\\
\int_\Omega \tilde{F}_1dx=0,\,\tilde{\omega}_1=0
&\mbox{on }\partial\Omega.
\end{cases}
\end{equation}
By taking derivative with respect to $t$ in \eqref{S313}, we declare that
$$F_1=\frac{\partial}{\partial t}\tilde{F}_1.$$

Next, $(F_2,\omega_2)$ corresponds to the convective term which lies in the central position of our arguments and satisfies
\begin{equation}\label{S310}
\begin{cases}
\nabla F_2+\nabla^\bot\omega_2=\mathrm{div}(\rho u\otimes u)
&\mbox{in }\Omega,\\
\int_\Omega F_2dx=0,\,\omega_2=0
&\mbox{on }\partial\Omega.
\end{cases}
\end{equation}
According to the elliptic estimates \eqref{S204}, we can roughly identify that
$$F_2\sim\rho u\otimes u.$$

Then, $(F_3,\omega_3)$ describes the effect of external force, which is solved by
\begin{equation}\label{S311}
\begin{cases}
\nabla F_3+\nabla^\bot\omega_3=-\rho\nabla f
&\mbox{in }\Omega,\\
\int_\Omega F_3dx=0,\,\omega_3=0
&\mbox{on }\partial\Omega.
\end{cases}
\end{equation}

Finally, $(F_4,\omega_4)$ is due to the boundary term. We have
\begin{equation}\label{S312}
\begin{cases}
\nabla F_4+\nabla^\bot\omega_4=0
&\mbox{in } \Omega,\\
\int_\Omega F_4dx=0,\,\omega_4=-Ku\cdot n^\bot
&\mbox{on }\partial\Omega.
\end{cases}
\end{equation}

With the help of \eqref{S309}--\eqref{S312}, we can transform \eqref{S306}  into the desired form
\begin{equation}\label{S314}
\frac{d}{dt}(\theta(\rho)+\tilde{F}_1)+P=u\cdot\nabla \tilde{F}_1-\sum_{i=2}^4F_i-\bar{F}.
\end{equation}

We comment that $\tilde{F}_1$ is a regular term which can be viewed as a controllable perturbation, see \eqref{CC306}. Thus $\theta(\rho)$ is the principal term, it can be identified as
$$\theta(\rho)\sim\rho^\beta,$$ for $\rho>M>1$, while $P(\rho)=\rho^\gamma$ is the damping term.
Based on such observation, we state our crucial lemma.
\begin{lemma}\label{L32}
Let $g_+\triangleq\max\{g,0\}$, then for any $1\leq\alpha<\infty$, there are constants $M$ and $C$ depending only on $\mu$, $\alpha$, $\beta$, $\gamma$, $\rho_0$, $u_0$, and $\Omega$, such that
\begin{equation}\label{S315}
\begin{split}
&\sup_{0\leq t\leq T}\int_\Omega\rho\Phi^\alpha dx\leq C,
\end{split}
\end{equation}
where $\Phi\triangleq (\theta(\rho)+\tilde{F_1}-M)_+$ is the principal term in \eqref{S314}.
\end{lemma}
\begin{proof}
Multiplying \eqref{S314} by $\alpha\rho\Phi^{\alpha-1}$ and integrating over $\Omega$ yields
\begin{equation}\label{S316}
\begin{split}
&\frac{d}{dt}\int_\Omega\rho\Phi^\alpha dx+\alpha\int_\Omega\rho P(\rho)\Phi^{\alpha-1}dx
=\alpha\int_\Omega\rho\big(u\cdot\nabla \tilde{F}_1-\sum_{i=2}^4F_i-\bar{F}\big)\Phi^{\alpha-1}dx.
\end{split}
\end{equation}

 Let us transform the second term into a better form. We first deduce that under the condition of $\theta(\rho)+\tilde{F}_1-M\geq 0$, there is a constant $C$ such that
\begin{equation}\label{CC301}
\bigg(\Phi+M\bigg)^\frac{\gamma}{\beta}=\left(\theta(\rho)+\tilde{F}_1\right)^\frac{\gamma}{\beta}\leq
C\left(P(\rho)+|\tilde{F}_1|^{\frac{\gamma}{\beta}}\right).
\end{equation}

In fact, when $\rho\leq 1$, by determining $M>10\beta,$ we declare that
$$\tilde{F}_1\geq M-2\mu\log\rho-\frac{1}{\beta}\rho^\beta\geq-2\mu\log\rho>0.$$
Observing that  $-\log\rho$, $(\theta(\rho)-2\mu\log\rho)$, and $\tilde{F}_1$ are positive, we therefore derive that
\begin{equation}\label{CC302}
\left(\theta(\rho)+\tilde{F}_1\right)^\frac{\gamma}{\beta}\leq \left(\theta(\rho)-2\mu\log\rho+\tilde{F}_1\right)^\frac{\gamma}{\beta}
\leq C\left(\big(\theta(\rho)
-2\mu\log\rho\big)^{\frac{\gamma}{\beta}}+|\tilde{F}_1|^{\frac{\gamma}{\beta}}\right).
\end{equation}

When $\rho>1$, since $\theta(\rho)$ is positive, we deduce that
\begin{equation}\label{CC303}
\left(\theta(\rho)+\tilde{F}_1\right)^\frac{\gamma}{\beta}\leq C\left(\theta(\rho)^\frac{\gamma}{\beta}+|\tilde{F}_1|^{\frac{\gamma}{\beta}}\right)\leq
C\left(\big(\theta(\rho)
-2\mu\log\rho\big)^{\frac{\gamma}{\beta}}+|\tilde{F}_1|^{\frac{\gamma}{\beta}}\right).
\end{equation}

Recalling that $P(\rho)=\theta(\rho)-2\mu\log\rho$, thus \eqref{CC302} together with \eqref{CC303} gives \eqref{CC301}. Substituting \eqref{CC301} into \eqref{S316} yields
\begin{equation}\label{CC304}
\frac{d}{dt}\int_\Omega\rho\Phi^\alpha dx+\int_\Omega\rho(\Phi+M)^{\frac{\gamma}{\beta}}\Phi^{\alpha-1}dx
\leq C\int_\Omega\rho H\cdot \Phi^{\alpha-1}dx,
\end{equation}
where $H$ is given by
$$H\triangleq |\tilde{F}_1|^{\frac{\gamma}{\beta}}+|u|\cdot|\nabla\tilde{F}_1|+\sum_{j=2}^4|F_j|+|\bar{F}|.$$

 Clearly, the second term on the left side of \eqref{CC304} serves as damping, which plays the key role in deriving global estimates. We argue that for any $\varepsilon>0$ there are constants $M>10\beta$ and $C$ depending on $\varepsilon$, $\alpha,$ $\beta,$ $\gamma,$ and $\mu$ such that
\begin{equation}\label{CC305}
\int_\Omega\rho H\cdot \Phi^{\alpha-1}dx\leq
C(M)A_1^2\left(\int_\Omega\rho\Phi^\alpha dx+1\right)+\varepsilon\int_\Omega\rho(\Phi+M)^{\frac{\gamma}{\beta}}\Phi^{\alpha-1}dx,
\end{equation}
where the energy $A_1$ is given by \eqref{SS401}. To make  idea clear, the proof of \eqref{CC305} is postponed to the next Proposition \ref{P31}.

We substitute \eqref{CC305} into \eqref{CC304} and deduce that
\begin{equation}\label{CC314}
\frac{d}{dt}\int_\Omega\rho\Phi^\alpha dx+
\int_\Omega\rho(\Phi+M)^\frac{\gamma}{\beta}\Phi^{\alpha-1} dx\leq
CA_1^2\left(\int_\Omega\rho\Phi^\alpha dx+1\right).
\end{equation}
The energy estimate \eqref{S301} together with Gronwall's inequality gives
\begin{equation*}
\begin{split}
&\sup_{0\leq t\leq T}\int_\Omega\rho\Phi^\alpha dx\leq C.
\end{split}
\end{equation*}
We therefore finish the proof of Lemma \ref{L32}.
\end{proof}

Now, we establish the key estimate \eqref{CC305}. The calculations involved are rather detailed, we must check each terms carefully.

\begin{proposition}\label{P31}
Under the same setting of Lemma \ref{L32}, for any $\varepsilon>0$ there exist constants $M>10\beta$ and $C$ depending on $\varepsilon$, $\alpha,$ $\beta,$ $\gamma,$ and $\mu$, such that
\begin{equation*}
\int_\Omega\rho H\cdot \Phi^{\alpha-1}dx\leq
CA_1^2 \left(\int_\Omega\rho\Phi^\alpha dx+1\right)+\varepsilon\int_\Omega\rho(\Phi+M)^{\frac{\gamma}{\beta}}\Phi^{\alpha-1}dx.
\end{equation*}
\end{proposition}

\begin{proof}
Note that
\begin{equation*}
\begin{split}
&\int_\Omega\rho H\cdot \Phi^{\alpha-1}dx=\int_\Omega\rho\left(|\tilde{F}_1|^{\frac{\gamma}{\beta}}+|u|\cdot|\nabla\tilde{F}_1|+\sum_{j=2}^4|F_j|+|\bar{F}|\right)\Phi^{\alpha-1}dx\triangleq\sum_{k=1}^6I_k.
\end{split}
\end{equation*}

We first mention that $\Phi$ provides a proper control of density $\rho$, in the sense that, for any $1\leq\alpha<\infty$ and $\beta>1$, we can find a constant $C$ independent of $M$, satisfying
\begin{equation}\label{CC306}
\int_\Omega\rho^{\alpha\beta+1}dx\leq C\left(\int_\Omega\rho\Phi^\alpha dx+M^\alpha\right).
\end{equation}

In fact, observe that Sobolev embedding theorem together with the elliptic estimate \eqref{S203} guarantees
\begin{equation}\label{CC307}
\|\tilde{F}_1\|_{L^{p}}
\leq C\|\rho u\|_{L^{\frac{2p}{p+1}}}
\leq C\|\rho^\frac{1}{2}\|_{L^p}\cdot\|\rho^\frac{1}{2}u\|_{L^2}\leq
C\|\rho\|_{L^\frac{p}{2}}^\frac{1}{2},
\end{equation}
which combined with the fact $\beta>4/3$ also leads to
\begin{equation}\label{CC309}
\begin{split}
\int_\Omega\rho|\tilde{F}_1|^\alpha dx
\leq\|\rho\|_{L^{\alpha\beta+1}}\|\tilde{F}_1\|_{L^{ (\alpha\beta+1)/\beta}}^\alpha
\leq C\|\rho\|_{L^{\alpha\beta+1}}^{\frac{\alpha}{2}+1}.
\end{split}
\end{equation}

Moreover, direct computations give 
$$\theta(\rho)\leq\Phi+|\tilde{F}_1|+M.$$
Consequently, in view of \eqref{CC309}, we declare that
\begin{equation}\label{CC308}
\begin{split}
\int_\Omega\rho^{\alpha\beta+1}dx&=
\int_{\{\rho>1\}}\rho^{\alpha\beta+1}dx
+\int_{\{\rho\leq 1\}}\rho^{\alpha\beta+1}dx\\
&\leq \int_{\{\rho>1\}}\rho \big(\theta(\rho)\big)^\alpha dx+1\\
&\leq C\left(\int_{\Omega}\rho(\Phi^\alpha+|\tilde{F}_1|^\alpha)dx+M^\alpha\right)\\
&\leq C\left(\int_\Omega\rho\Phi^\alpha dx+\|\rho\|_{L^{\alpha\beta+1}}^{\frac{\alpha}{2}+1}+M^\alpha\right).
\end{split}
\end{equation}

Note that $\frac{\alpha}{2}+1<\alpha\beta+1$, thus we apply Young's inequality to \eqref{CC308} and establish \eqref{CC306}.
Then, let us check each $I_k$ in details. Since we are concerning with the bounded domain, it is harmless to assume $\alpha\geq 3$ in our proof.
\\

\textit{Estimate of $I_1$.}
The term $I_1$ is a controllable perturbation due to $\tilde{F}_1$. There are two cases we must investigate.

First, if $\gamma\leq 2\beta$,
by virtue of \eqref{S203} and Sobolev embedding theorem, we declare that, for any $\varepsilon>0$,
\begin{equation}\label{CC315}
\|\tilde{F}_1\|_{L^\infty}^2
\leq C(\varepsilon)\|\rho u\|_{L^{2+\varepsilon}}^2
\leq C(\varepsilon)\|\nabla u\|_{L^2}^2\|\rho\|_{L^{2+2\varepsilon}}^2.
\end{equation}

When $\gamma\leq 2$, we check that
\begin{equation*}
\begin{split}
\|\rho\|_{L^{2+2\varepsilon}}^2
\leq\|\rho\|_{L^{\gamma}}^{2(1-\theta)}\|\rho\|_{L^{\alpha\beta+1}}^{2\theta}\leq C\|\rho\|_{L^{\alpha\beta+1}}^{2\theta},
\end{split}
\end{equation*}
where $\theta$ is given by
\begin{equation}\label{CC310}
\frac{1-\theta}{\gamma}+\frac{\theta}{\alpha\beta+1}=\frac{1}{2+2\varepsilon}.
\end{equation}
Under current setting, we have $\alpha\geq 3$, $\beta>1$, and $1<\gamma\leq 2$. By taking $\varepsilon<\gamma-1$, \eqref{CC310} ensures
$$\frac{2\theta}{\alpha\beta+1}<\frac{1}{\alpha}.$$
Consequently, if $\gamma\leq 2$, we apply \eqref{CC306} to derive
$$\|\rho\|_{L^{2+2\varepsilon}}^2
\leq C\|\rho\|_{L^{\alpha\beta+1}}^{2\theta}\leq C\left(\int_\Omega\rho\Phi^\alpha dx+M^\alpha\right)^{\frac{1}{\alpha}}.$$

While for $\gamma>2$, by choosing $2\varepsilon<\gamma-2$, we directly have
\begin{equation*}
\|\rho\|_{L^{2+2\varepsilon}}^2
\leq C\|\rho\|_{L^{\gamma}}^2
\leq C.
\end{equation*}

In conclusion, we declare that
\begin{equation*}
\|\tilde{F}_1\|_{L^\infty}^2
\leq C\|\nabla u\|_{L^2}^2\|\rho\|_{L^{2+2\varepsilon}}^2
\leq C(M)A_1^2\left(\int_\Omega\rho\Phi^\alpha dx+1\right)^{\frac{1}{\alpha}},
\end{equation*}
which leads to
\begin{equation}\label{CC312}
\begin{split}
I_1&=\int_{\Omega}\rho
|\tilde{F}_1|^{\frac{\gamma}{\beta}}\Phi^{\alpha-1}dx\\
&\leq C\|\tilde{F}_1\|_{L^\infty}^2\int_\Omega\rho
\Phi^{\alpha-1}dx+C\int_\Omega\rho
\Phi^{\alpha-1}dx\\
&\leq CA_1^2 \left(\int_\Omega\rho\Phi^\alpha dx+1\right)+\varepsilon\int_\Omega\rho(\Phi+M)^{\frac{\gamma}{\beta}}\Phi^{\alpha-1}dx,\\
\end{split}
\end{equation}
provided $M>(C/\varepsilon)^\frac{\beta}{\gamma}$.

Next, if $\gamma>2\beta$, we mention that \eqref{S301} ensures
$$\sup_{0\leq t\leq T}\|\rho u\|_{L^{2\gamma/(\gamma+1)}}\leq \sup_{0\leq t\leq T}(\|\sqrt\rho u\|_{L^2}\|\rho\|_{L^\gamma}^{1/2})\le C.$$
Thus, according to elliptic estimate \eqref{S203} and Sobolev embedding theorem, we declare that, for some $\varepsilon>0$,
\begin{equation*}
\|\tilde{F}_1\|_{L^\infty}^{\frac{\gamma}{\beta}}
\leq C(\varepsilon)\|\rho u\|_{L^{2\gamma/(\gamma+1)}}^{\frac{\gamma}{\beta}-2}\|\rho u\|_{L^{\gamma-\varepsilon}}^2
\leq C\|\rho u\|_{L^{2\gamma/(\gamma+1)}}^{\frac{\gamma}{\beta}-2}\|\rho\|_{L^{\gamma}}^2\|\nabla u\|_{L^2}^2,
\end{equation*}
where $\varepsilon$ is determined by
\begin{equation}\label{CC311}
\frac{\gamma+1}{2\gamma}\cdot\left(1-\frac{2\beta}{\gamma}\right)+\frac{1}{\gamma-\varepsilon}\cdot\frac{2\beta}{\gamma}<\frac{1}{2}.
\end{equation}

We mention that \eqref{CC311} is possible. Observe that, for the case of $\beta>1$ and $\gamma>2\beta>2$, we have
$$\frac{\gamma+1}{2\gamma}\cdot\left(1-\frac{2\beta}{\gamma}\right)+\frac{1}{\gamma}\cdot\frac{2\beta}{\gamma}<\frac{1}{2},$$
since it is equivalent to $1/\gamma+1/(2\beta)<1.$ Thus \eqref{CC311} is valid for  properly small $\varepsilon$.

Consequently, we argue that
$$\|\tilde{F}_1\|_{L^\infty}^{\frac{\gamma}{\beta}}\leq C\|\rho u\|_{L^{2\gamma/(\gamma+1)}}^{\frac{\gamma}{\beta}-2}\|\rho\|_{L^{\gamma}}^2\|\nabla u\|_{L^2}^2\leq CA_1^2,$$
which leads to
\begin{equation}\label{CC313}
\begin{split}
I_1&=\int_{\Omega}\rho
|\tilde{F}_1|^{\frac{\gamma}{\beta}}\Phi^{\alpha-1}dx\\
&\leq C\|\tilde{F}_1\|_{L^\infty}^{\frac{\gamma}{\beta}}
\left(\int_\Omega\rho\Phi^{\alpha-1}dx\right) \\
&\leq CA_1^2
\left(\int_\Omega\rho\Phi^{\alpha}dx+1\right).\\
\end{split}
\end{equation}

Combining \eqref{CC312} with \eqref{CC313}, we declare that
$$I_1\leq CA_1^2 \left(\int_\Omega\rho\Phi^\alpha dx+1\right)+\varepsilon\int_\Omega\rho(\Phi+M)^{\frac{\gamma}{\beta}}\Phi^{\alpha-1}dx.$$
\\

\textit{Estimates of $I_2$ and $I_3$.}
We comment that $I_2$ and $I_3$ are principal terms. Note that the order of them coincides, since in view of elliptic estimate \eqref{S203} and \eqref{S204}, we can formally identify
$$u\cdot\nabla\tilde{F}_1,\,\ F_2\sim \rho u\otimes u.$$

To deal with $I_2$, we make use of Sobolev embedding theorem, elliptic estimate \eqref{S203}, and the fact $\beta>4/3$ to infer that
\begin{equation*}
\|u\cdot\nabla\tilde{F}_1\|_{L^{(\alpha\beta+1)/\beta}}
\leq C\|u\|_{L^{7(\alpha\beta+1)/\beta}}\cdot\|\rho u\|_{L^{7(\alpha\beta+1)/6\beta}}
\leq C\|\nabla u\|_{L^2}^2\|\rho\|_{L^{\alpha\beta+1}},
\end{equation*}
which combined with \eqref{CC306} gives
$$\|u\cdot\nabla\tilde{F}_1\|_{L^{(\alpha\beta+1)/\beta}}
\leq CA_1^2\left(\int_\Omega\rho\Phi^\alpha dx +1\right)^{\frac{1}{\alpha\beta+1}}.$$
Therefore, we argue that
\begin{equation*}
\begin{split}
I_2&=\int_\Omega\rho|u|\cdot|\nabla\tilde{F}_1|
\Phi^{\alpha-1}dx\\
&\leq\|u\cdot\nabla\tilde{F}_1\|_{L^{(\alpha\beta+1)/\beta}}\cdot\|\rho\|_{L^{\alpha\beta+1}}^{\frac{1}{\alpha}}
\left(\int_\Omega\rho\Phi^{\alpha}dx\right)^{\frac{\alpha-1}{\alpha}}\\
&\leq C\|u\cdot\nabla\tilde{F}_1\|_{L^{(\alpha\beta+1)/\beta}}\left(\int_\Omega\rho\Phi^\alpha dx +1\right)^{1-\frac{\beta}{\alpha\beta+1}}\\
&\leq CA_1^2\left(\int_\Omega\rho\Phi^\alpha dx +1\right).
\end{split}
\end{equation*}

$I_3$ is similar, with the help of \eqref{S204}, Sobolev embedding theorem and the fact $\beta>4/3$, we infer that
\begin{equation*}
\|F_2\|_{L^{(\alpha\beta+1)/\beta}}
\leq C\|\rho|u|^2\|_{L^{(\alpha\beta+1)/\beta}}
\leq C\|\nabla u\|_{L^2}^2\|\rho\|_{L^{\alpha\beta+1}},
\end{equation*}
which combined with \eqref{CC306} gives
$$\|F_2\|_{L^{(\alpha\beta+1)/\beta}}\leq
CA_1^2\left(\int_\Omega\rho\Phi^\alpha dx +1\right)^{\frac{1}{\alpha\beta+1}}.$$
Therefore, we argue that
\begin{equation*}
\begin{split}
I_3&=\int_\Omega\rho|F_2|\cdot\Phi^{\alpha-1}dx\\
&\leq\|F_2\|_{L^{(\alpha\beta+1)/\beta}}\cdot\|\rho\|_{L^{\alpha\beta+1}}^{\frac{1}{\alpha}}
\left(\int_\Omega\rho\Phi^{\alpha}dx\right)^{\frac{\alpha-1}{\alpha}}\\
&\leq C\|F_2\|_{L^{(\alpha\beta+1)/\beta}}\left(\int_\Omega\rho\Phi^\alpha dx +1\right)^{1-\frac{\beta}{\alpha\beta+1}}\\
&\leq CA_1^2\left(\int_\Omega\rho\Phi^\alpha dx +1\right).
\end{split}
\end{equation*}
\\

\textit{Estimate of $I_4$.}
 By virtue of elliptic estimate \eqref{S203} and Sobolev embedding theorem, we have that  for some constant $C$ independent with $M$,
\begin{equation*}
\|F_3\|_{L^\infty}
\leq C\|\rho\nabla f\|_{L^{3}}
\leq C\|\rho\|_{L^{4}}.
\end{equation*}
By taking $\alpha'=3/\beta$ in \eqref{CC306}, we deduce that
$$\|F_3\|_{L^\infty}
\leq C\|\rho\|_{L^{\alpha'\beta+1}}
\leq C\left(\int_\Omega\rho\Phi^{\frac{3}{\beta}}dx+M^{\frac{3}{\beta}}\right)^{\frac{1}{4}}.$$

Consequently,
\begin{equation*}
\begin{split}
I_4&\leq C\|F_3\|_{L^\infty}\int_\Omega\rho\Phi^{\alpha-1}dx\\
&\leq C\left(\int_\Omega\rho\Phi^{\frac{3}{\beta}}dx\right)^{\frac{1}{4}}\int_\Omega\rho\Phi^{\alpha-1}dx+C M^{\frac{3}{4\beta}} \int_\Omega\rho\Phi^{\alpha-1}dx\\
&\leq C\int_\Omega\rho\Phi^{\frac{3}{4\beta}+\alpha-1}dx+C M^{\frac{3}{4\beta}} \int_\Omega\rho\Phi^{\alpha-1}dx\\
&\leq\varepsilon\int_\Omega(\Phi+M)^{\frac{\gamma}{\beta}}\Phi^{\alpha-1}dx,
\end{split}
\end{equation*}
by setting $M>(C\varepsilon)^\frac{\beta}{\gamma-3/4}$, where  we have used the fact $\gamma>1>3/4$.
\\

\textit{Estimate of $I_5$.}
Note that the elliptic estimate \eqref{S205} and Sobolev embedding theorem give for any $1\leq p<\infty$,
\begin{equation}\label{CC353}
\|F_4\|_{L^p}\leq C(p)\|\nabla u\|_{L^2}\leq CA_1.
\end{equation}
Thus, by taking $p=2\alpha\gamma/(\gamma+1)$ in \eqref{CC353}, we check that
\begin{equation*}
\begin{split}
I_5&=\int_\Omega\rho|F_4|\Phi^{\alpha-1}dx\\
&\leq C\int_\Omega \rho|F_4|^2\Phi^{\alpha-1}dx+\int_\Omega \rho\Phi^{\alpha-1}dx\\
&\leq C\|\rho\|_{L^\gamma}\|F_4\|_{L^p}^2(\int_\Omega\rho\Phi^{\alpha}dx)^{\frac{\alpha-1}{\alpha}}+\int\rho\Phi^{\alpha-1}dx\\
&\leq CA_1^2\left(\int_\Omega\rho\Phi^{\alpha}dx+1\right)+\varepsilon\int\rho(\Phi+M)^\frac{\gamma}{\beta}\Phi^{\alpha-1}dx,
\end{split}
\end{equation*}
by setting $M>(C\varepsilon)^\frac{\beta}{\gamma}$.
\\

\textit{Estimate of $I_6$.}
 Note that by virtue of \eqref{CC306}, we argue that
\begin{equation*}
\begin{split}
\bar{F}&=\int_\Omega\rho^\beta\mathrm{div}udx+\int_\Omega\rho^\gamma dx\\
&\leq\left(\int\rho^\beta(\mathrm{div}u)^2dx\right)^\frac{1}{2}\cdot\left(\int_\Omega\rho^\beta dx\right)^\frac{1}{2}+C,\\
&\leq C\left(\int\rho^\beta(\mathrm{div}u)^2dx\right)\cdot\left(\int_\Omega\rho^{\alpha\beta+1}dx\right)^{\frac{\beta}{\alpha\beta+1}} +C\\
&\leq CA_1^2\left(\int_\Omega\rho\Phi^\alpha dx+1\right)^{\frac{\beta}{\alpha\beta+1}}+C.
\end{split}
\end{equation*}

Consequently, we deduce that
\begin{equation}\label{S326}
\begin{split}
I_6&=|\bar{F}|\int_\Omega\rho\Phi^{\alpha-1}dx\\
&\leq C A_1^2
\left(\int_\Omega\rho\Phi^{\alpha}dx+1\right)^{\frac{\beta}{\alpha\beta+1}+\frac{\alpha-1}{\alpha}}+C\int_\Omega\rho\Phi^{\alpha-1}dx\\
&\leq CA_1^2\left(\int_\Omega\rho\Phi^{\alpha}dx+1\right)+\varepsilon\int\rho(\Phi+M)^\frac{\gamma}{\beta}\Phi^{\alpha-1}dx,\\
\end{split}
\end{equation}
by setting $M>(C\varepsilon)^\frac{\beta}{\gamma}$.

Combining all estimates about $I_k$, we set $M>\max\{10\beta,(C\varepsilon)^\frac{\beta}{\gamma-3/4}\}$ and arrive at
$$\int_\Omega\rho H\cdot \Phi^{\alpha-1}dx\leq
C(M)A_1^2 \left(\int_\Omega\rho\Phi^\alpha dx+1\right)+\varepsilon\int_\Omega\rho(\Phi+M)^{\frac{\gamma}{\beta}}\Phi^{\alpha-1}dx,$$
which finishes the proof of Proposition \ref{P31}.
\end{proof}

\begin{remark}
In Vaigant-Kazhikov \cite{vaigant1995}, the damping term
$$\int_\Omega\rho P(\rho)\Phi^{\alpha-1}dx,$$
is dropped out directly since it is positive.

However  in our arguments, we make full use of it to cancel out the perturbation term $I_1$, the external force term $I_4$, the boundary term $I_5$ and the lower order term $I_6$, which is important in deriving uniform estimates of density in Lemma \ref{L32}.
\end{remark}

Lemma \ref{L32} further leads to the uniform $L^p$ estimates of the density $\rho$.
\begin{corollary}\label{L33}
For any $3\leq\alpha<\infty$, there are constants $C$ and $M_1$ depending only on $\mu$, $\alpha$, $\beta$, $\gamma$, $\rho_0$, $u_0$, and $\Omega$, such that
\begin{equation}\label{S327}
\sup_{0\leq t\leq T}\|\rho\|_{L^{\alpha\beta+1}}+\int_0^T\int_\Omega(\rho-M_1)_+^{\alpha-1} dxdt\leq C.
\end{equation}
\end{corollary}
\begin{proof}
This corollary rearranges Proposition \ref{L32} to a better form.
First, by virtue of \eqref{CC306} and Lemma \ref{L32}, we declare that
\begin{equation}\label{S329}
\sup_{0\leq t\leq T}
\left(\int_\Omega\rho^{\alpha\beta+1}dx\right)\leq C\sup_{0\leq t\leq T}\left(\int_\Omega\rho\Phi^\alpha dx\right)+C\leq C.
\end{equation}
Substituting \eqref{S329} into \eqref{CC314}, we also declare that, for $\alpha\geq 3$
\begin{equation}\label{S331}
\int_0^T\int_\Omega\rho\Phi^{\alpha-1}dx dt\leq C.
\end{equation}

Now, observe that
\begin{equation*}
\int_\Omega\rho|\tilde{F}_1|^{\alpha-1} dx
\leq\|\rho\|_{L^2}\|\tilde{F}_1\|_{L^{2(\alpha-1)}}^{\alpha-1}
\leq C\|\tilde{F}_1\|_{L^{2(\alpha-1)}}^{\alpha-3}\cdot
\|\tilde{F}_1\|_{L^{\infty}}^{2}.
\end{equation*}
By virtue of \eqref{CC307} and \eqref{S329}, on the one hand, we have
\begin{equation*}
\|\tilde{F}_1\|_{L^{2(\alpha-1)}}\leq C\|\rho\|_{L^{\alpha-1}}^{\frac{1}{2}}\leq C.
\end{equation*}
On the other hand, \eqref{CC315} together with \eqref{S329} also guarantees
\begin{equation*}
\|\tilde{F}_1\|_{L^{\infty}}
\leq C\|\rho\|_{L^{3}}\|\nabla u\|_{L^2}\leq CA_1.
\end{equation*}
These results lead to
\begin{equation*}
\int_0^T\int_\Omega\rho|\tilde{F}_1|^{\alpha-1} dxdt
\leq C\int_0^TA_1^2dt\leq C,
\end{equation*}
which together with \eqref{S331} gives
\begin{equation}\label{S332}
\begin{split}
&\int_0^T\int_\Omega\rho(\theta(\rho)-M)_+^{\alpha-1}dx dt\leq C\int_0^T\int_\Omega\rho\left(\Phi^{\alpha-1}+|\tilde{F}_1|^{\alpha-1}\right)dx dt\leq C.
\end{split}
\end{equation}

 Next, let us transform \eqref{S332} into a neat form. Observe that $(\rho^\beta-\beta M)>0$ enforces $\rho\geq 1$ since $M>10\beta$. Consequently, we argue that
\begin{equation*}
(\theta(\rho)-M)_+=(2\mu\log\rho+\frac{1}{\beta}\rho^{\beta}-M)_+\geq C(\rho^\beta-\beta M)_+.
\end{equation*}
We set $M_1^\beta=\beta M$ to deduce
\begin{equation*}
(\theta(\rho)-M)_+\geq C(\rho^\beta-\beta M)_+\geq C(\rho-M_1)_+,
\end{equation*}
which together with \eqref{S332} gives
\begin{equation}\label{S333}
\int_0^T\int_\Omega(\rho-M_1)_+^{\alpha-1}dxdt
\leq C\int_0^T\int_\Omega\rho(\theta(\rho)-M )_+^{\alpha-1}dxdt\leq C.
\end{equation}

Now \eqref{S333} together with \eqref{S329} gives \eqref{S327}, we therefore finish the proof.
\end{proof}

\begin{remark}
Compared with the $L^p$ estimate of density $\rho$ due to Vaigant-Kazhikov \cite{vaigant1995}
$$\|\rho\|_{L^p}\leq C(T)p^{\frac{2}{\beta-1}},$$
Lemma \ref{L33} gives the uniform $L^p$ estimates of the density $\rho$,
$$\|\rho\|_{L^p}\leq C(p).$$

Moreover, with the help of the damping term, we also derive the global in time integrability of $\rho$,  
$$\int_0^T\int_\Omega(\rho-M_1)_+^pdxdt\leq C\ (p\geq 2),$$
which prevents $\rho$ growing too fast and improves the result of Vaigant-Kazhikov \cite{vaigant1995}.
\end{remark}

\subsection{Application I: Large time behavior of density}\label{SS32}

\quad Proposition \ref{L34} ensures the density is uniformly bounded in $L^p$ space, then we can study the large time behavior of the density more accurately. We first apply the argument in \cite{HLX} to deduce that $\rho$ converges to the steady state $\rho_s$ in strong sense.

\begin{proposition}\label{L34}
Under the hypothesis of Theorem \ref{T11}, suppose \eqref{CC101} is valid and $(\rho, u)$ is the unique strong solution given by Lemma \ref{L21}. If $\rho_s$ is the steady state given by \eqref{S107}, then for any $1\leq p<\infty$, we have
\begin{equation}\label{S334}
\lim_{t\rightarrow\infty}\|\rho-\rho_s\|_{L^{p}}=0.
\end{equation}
\end{proposition}
For the sake of completeness, we sketch the proof here.
\begin{proof}
According to Corollary \ref{L33}, once we establish \eqref{S334} for any $p=2$, Proposition \ref{L34} follows by interpolation. Thus, we will show that
$$\lim_{t\rightarrow\infty}\|\rho-\rho_s\|_{L^2}=0.$$

\textit{Step 1.}
Let us first show the convergence on the interval $t\in[\tau,\tau+1]$, $\tau\in[0,\infty)$. Precisely, for any $1\leq p<\infty$, we have
\begin{equation}\label{CC317}
\lim_{\tau\rightarrow\infty}\|\rho-\rho_s\|_{L^p(\Omega\times(\tau,\tau+1))}=0.
\end{equation}

In particular, suppose that $\{\tau_n\}_{n=1}^\infty$ is any positive subsequence tending to $\infty$. We introduce $\rho_n=\rho(x,t+\tau_n),\ u_n=u(x,t+\tau_n)$, $P_n=P(\rho_n)$ and restrict the problem on $\Omega\times[0,1]$. By virtue of Proposition \ref{L33} and Sobolev embedding theorem, for any $1\leq p<\infty$, we declare that
\begin{equation}\label{CC316}
\|\rho_n u_n\|^2_{L^p}\leq C\left(\sup_{0\leq t\leq 1}\|\rho_n\|^2_{L^{2p}}\right)\cdot\|u_n\|^2_{L^{2p}}\leq C\|\nabla u_n\|^2_{L^2}.
\end{equation}
In view of Proposition \ref{L31}, we also have
$$\int_0^\infty\|\nabla u\|_{L^2}^2dt\leq C,$$
thus the absolutely continuity of integral together with \eqref{CC316} leads to
\begin{equation}\label{S343}
\lim_{n\rightarrow\infty}\int_{0}^{1}\left(\|\rho_n u_n\|^2_{L^p}+\|\nabla u_n\|_{L^2}^2\right)dt
=0.
\end{equation}

According to Proposition \ref{L33}, up to some subsequence, we may assume
\begin{equation}\label{S344}
\begin{split}
\rho_n&\rightharpoonup\hat{\rho}\  \textit{in}\ L^4(\Omega\times(0,1)),\\
P_n&\rightharpoonup\hat{P}\ \textit{in}\ L^4(\Omega\times(0,1)),\\
P_n^2&\rightharpoonup\hat{P^2}\ \textit{in}\ L^2(\Omega\times(0,1)).
\end{split}
\end{equation}
Next we show that the convergence in \eqref{S344} is actually in strong sense by the same method as in \cite{HLX}.

Let us consider a sequence of vector functions $\{Q_n\}_{n=1}^\infty$ on $(0,1)\times\Omega\subset\mathbb{R}^3$:
\begin{equation*}
Q_n(t, x_1,x_2)=\bigg(P_n,0,0\bigg).
\end{equation*}

We check that $\mathrm{div}Q_n=\partial_t P_n$ in $\mathbb{R}^3$.
Thus by virtue of \eqref{S101}$_1$, we argue that
\begin{equation*}
\mathrm{div}Q_n=\partial_tP_n=-\mathrm{div}(P_nu_n)-(\gamma-1)P_n\mathrm{div}u_n,
\end{equation*}
which together with \eqref{S327} and \eqref{S343} gives
\begin{equation}\label{S346}
\mathrm{div}Q_n \rightarrow 0\ \textit{ in $W^{-1,3}(\Omega\times(0,1))$}.
\end{equation}

Similarly, we check that $\mathrm{rot}Q_n=\bigg(0,\nabla P_n\bigg)$ in $\mathbb{R}^3$. Note that \eqref{S101}$_2$ gives
$$\mathrm{rot}Q_n=-\partial_t(\rho_nu_n)-\mathrm{div}(\rho_n u_n\otimes u_n)+\nabla(\mu+\lambda)\mathrm{div}u_n+\mu\Delta u_n+\rho_n\nabla f,$$
which together with \eqref{S327} and \eqref{S343} also leads to
\begin{equation}\label{S347}
\mathrm{rot}Q_n\ \textit{is precompact in $W^{-1,\frac{3}{2}}(\Omega\times(0,1))$}.
\end{equation}

Combining \eqref{S344}--\eqref{S347}, we apply the div-curl lemma due to Murat \cite{Mu} and Tartar \cite{Tar} (see also Zhou \cite{Zho}) to deduce that
\begin{equation}\label{S348}
\hat{P}\cdot\hat{P}=(\hat{P})^2=\hat{P^2}.
\end{equation}

However, recalling that $P(x)=x^\gamma$ is strictly convex, hence \eqref{S348} together with \eqref{S344} and \eqref{S327} guarantees that, for any $1\leq p<\infty$,
\begin{equation}\label{CC318}
\lim_{n\rightarrow\infty}\left(\|\rho_n-\hat{\rho}\|_{L^{p}(\Omega\times(0,1))}+\|P_n-(\hat{\rho})^\gamma\|_{L^p(\Omega\times(0,1))}\right)=0.
\end{equation}

With the help of \eqref{CC318} and \eqref{S343}, we pass to the limit in \eqref{S101}$_1$ and deduce that in the sense of distribution,
$$\partial_t\hat{\rho}=0\ \mathrm{in}\ \Omega\times[0,1],$$
which means $\hat{\rho}=\hat{\rho}(x)$ is independent with $t$.

Similarly, by passing to the limit in \eqref{S101}$_2$, we apply \eqref{CC318} and \eqref{S343} to obtain the stationary system
\begin{equation*}
\begin{cases}
\nabla(\hat{\rho})^\gamma=\hat{\rho}\nabla f\ \mathrm{in}\ \Omega,\\
\int_\Omega\hat{\rho}dx=\int_\Omega\rho_0dx,
\end{cases}
\end{equation*}
where the last line is due to the conservation of mass.
Now the uniqueness assertion in Lemma \ref{L11} enforces  $\hat\rho=\rho_s$.

In conclusion, for any subsequence $\{\rho_n\}$, there is a further convergent subsequence (still denoted by $\{\rho_n\}$) converging to the unique limit $\rho_s$ in the sense of
$$\lim_{n\rightarrow\infty}\|\rho_n-\rho_s\|_{L^p}=0,$$
for any $1\leq p<\infty$, which is equivalent to \eqref{CC317} and finishes the first step.
\\

\textit{Step 2.} Then we establish the point-wise convergence, $$\lim_{t\rightarrow\infty}\|\rho-\rho_s\|_{L^2}=0.$$

Let us first derive that
\begin{equation}\label{CC321}
\rho\rightharpoonup\rho_s\ \mathrm{in}\ L^2(\Omega)\ \mathrm{as}\ t\rightarrow\infty.
\end{equation}
In fact, for any $t\in[0,\infty),\ s\in[0,1]$ and test function $\phi(x)\in C_0^\infty(\Omega)$, we integrate \eqref{S101}$_1$ over $[t,t+s]$  to deduce that
\begin{equation}\label{CC319}
\int_\Omega\rho\cdot\phi dx\, (t)=\int_\Omega\rho\cdot\phi dx\, (t+s)-\int_t^{t+s}\int_\Omega\rho u\cdot\nabla\phi dxd\tau.
\end{equation}

Integrating \eqref{CC319} over $[0,1]$ with respect to $s$, we arrive at
\begin{equation}\label{CC320}
\int_\Omega\rho\cdot\phi dx(t)=\int_t^{t+1}\left(\int_\Omega\rho\cdot\phi dx\right) d\tau-
\int_t^{t+1}(t+1-\tau)\cdot\left(\int_\Omega\rho u\cdot\nabla\phi dx\right)d\tau,
\end{equation}
where we have applied the following fact due to Fubini's theorem that
\begin{equation}\label{CC325}
\int_0^1\left(\int_t^{t+s}g(\tau)d\tau\right) ds=\int_t^{t+1}(t+1-\tau)\cdot g(\tau)d\tau.
\end{equation}

Observe that
\begin{equation*}\int_t^{t+1}(t+1-\tau)\cdot\left(\int_\Omega\rho u\cdot\nabla\phi dx\right)d\tau
\leq C\int_{t}^{t+1}\|\rho u\|_{L^1}^2dt.
\end{equation*}
Thus, letting $t\rightarrow \infty$ in \eqref{CC320}, we deduce from \eqref{CC317} and \eqref{S343} that
\begin{equation}\label{CC374}
\begin{split}
&\lim_{t\rightarrow\infty}\int_t^{t+1}\int_\Omega\rho\cdot\phi dxd\tau=\int_0^{1}\int_\Omega\rho_s\cdot\phi dxd\tau=\int_\Omega\rho_s\cdot\phi dx,\\
&\lim_{t\rightarrow\infty}\int_t^{t+1}(t+1-\tau)\cdot\left(\int_\Omega\rho u\cdot\nabla\phi dx\right)d\tau=0.
\end{split}
\end{equation}
Consequently, we derive that
$$\lim_{t\rightarrow\infty}\int_\Omega\rho\cdot\phi dx=\int_\Omega\rho_s\cdot\phi dx,$$
which implies \eqref{CC321} holds.

Similar process also leads to
\begin{equation}\label{CC322}
\lim_{t\rightarrow\infty}\int_\Omega\rho^2dx=\int_\Omega\rho_s^2dx.
\end{equation}

In fact, \eqref{S101}$_1$ gives us
\begin{equation}\label{CC323}
\partial_t\rho^2+\mathrm{div}(\rho^2 u)=-\rho^2\mathrm{div}u.
\end{equation}
For any $t\in[0,\infty),\ s\in[0,1]$, we integrate \eqref{CC323} over $[t,t+s]$ to deduce that
\begin{equation}\label{CC324}
\int_\Omega\rho^2 dx\,(t)=\int_\Omega\rho^2 dx\,(t+s)+\int_t^{t+s}\int_\Omega\rho^2\mathrm{div}u\ dxd\tau.
\end{equation}

With the help of \eqref{CC325}, we integrate \eqref{CC324} over $[0,1]$ with respect to $s$ and derive 
\begin{equation}\label{CC326}
\int_\Omega\rho^2 dx(t)=\int_t^{t+1}\int_\Omega\rho^2 dxd\tau+
\int_t^{t+1}(t+1-\tau)\cdot\left(\int_\Omega\rho^2\mathrm{div}u\  dx\right)d\tau.
\end{equation}
Letting $t\rightarrow\infty$ in \eqref{CC326}, by the same arguments as in \eqref{CC374}, we apply \eqref{CC317} and \eqref{S343} to argue that \eqref{CC322} holds.

Combining \eqref{CC321} with \eqref{CC322}, we deduce that
$$\lim_{t\rightarrow\infty}\|\rho-\rho_s\|_{L^2}=0,$$
and finish the proof of Proposition \ref{L34}.
\end{proof}

\begin{remark}
In light of \eqref{S334}, without lost of generality, for some $\delta>0$ determined later, we may assume
\begin{equation}\label{CC330}
\|\rho-\rho_s\|_{L^4}\leq\delta,
\end{equation}
in the rest of paper.
\end{remark}

Proposition \ref{L34} makes sure the density converging to $\rho_s$, however the rate of decay is still unknown. Next proposition provides a crude estimate of it in terms of global $L^2$ integrability of $(\rho-\rho_s)$. For later application, we introduce the energy term $A_2$,
\begin{equation}\label{CC343}
\begin{split}
A_2^2(t)&\triangleq\int_\Omega(\rho+1)^{\gamma-1}(\rho-\rho_s)^2dx,\\
\end{split}
\end{equation}
then,  the conclusion in term of $A_2$ can be stated as follows:
\begin{proposition}\label{L35}
Under the condition of Proposition \ref{L34}, there is a constant $C$ depending on $\mu$, $\beta$, $\gamma$, $\rho_0$, $u_0$, and $\Omega$, such that
\begin{equation}\label{S335}
\int_0^TA_2^2(t)\, dt\leq C.
\end{equation}
\end{proposition}
\begin{proof}
Let us rewrite  \eqref{S101}$_2$ as
\begin{equation}\label{CC327}
\nabla(\rho^\gamma-\rho_s^\gamma)+(1-\rho\cdot\rho_s^{-1})\nabla\rho_s^\gamma=R,
\end{equation}
where the remainder $R$ is given by
$$R=-\frac{\partial}{\partial t}(\rho u)-\mathrm{div}(\rho u\otimes u)+\mu\Delta u+\nabla\big((\mu+\lambda)\mathrm{div}u\big)=\sum_{i=1}^4R_i.$$
The left side of \eqref{CC327} is the principal term, we apply the idea in \cite{HLX} to handle it.

In view of Lemma \ref{L25}, $\mathcal{B}(\rho-\rho_s)$ is well defined.
Multiplying \eqref{CC327} by $\mathcal{B}(\rho-\rho_s)\cdot\rho_s^{-1}$ and integrating over $\Omega$ yield
\begin{equation}\label{CC328}
\int_\Omega\bigg(\rho_s^{-1}\nabla(\rho^\gamma-\rho_s^\gamma)+(\rho_s^{-1}-\rho\cdot\rho_s^{-2})\nabla\rho_s^\gamma\bigg)\cdot\mathcal{B}(\rho-\rho_s)dx=\int_\Omega R\cdot \mathcal{B}(\rho-\rho_s)\rho_s^{-1}dx.
\end{equation}
\\

\textit{Estimates of the principal term.}
We check the left side of \eqref{CC328}.
In fact, integration by parts reduces it to
\begin{equation}\label{CC332}
\begin{split}
\int_\Omega\rho_s^{-1}(\rho^\gamma-\rho_s^\gamma)(\rho-\rho_s)dx+\int_\Omega\nabla\rho_s^{-1}\cdot
H(\rho,\rho_s)\cdot\mathcal{B}(\rho-\rho_s)dx,
\end{split}
\end{equation}
where $H(\rho,\rho_s)=\rho^\gamma-\rho_s^{\gamma}-\gamma\rho_s^{\gamma-1}
(\rho-\rho_s).$

The second term in \eqref{CC332} requires detailed analysis. Note that Taylor's expansion gives us
$$H(\rho,\rho_s)=(\gamma^2-\gamma)\bigg(\int_{Q}
\sigma\big(\xi\sigma\rho+(1-\xi\sigma)\rho_s\big)^{\gamma-2}d\xi\, d\sigma\bigg)(\rho-\rho_s)^2,$$
where $Q=[0,1]\times[0,1]$. Let us check the term in the largest bracket.

When $\gamma<2$,
we mention that $\inf\rho_s>0$ and $\gamma>1$ ensure 
\begin{equation}\label{S338}
\begin{split}
&\int_{Q}\sigma
(\xi\sigma\rho+(1-\xi\sigma)\rho_s)^{\gamma-2}\,d\xi d\sigma\\
&\leq C\int_{Q}
(1-\xi\sigma)^{\gamma-2}d\xi d\sigma\\
&=C\int_0^1\sigma^{-1}\left(\int_0^\sigma(1-\xi)^{\gamma-2}d\xi\right)d\sigma\\
&=C\int_0^1 \frac{1-(1-\sigma)^{\gamma-1}}{(\gamma-1)\sigma}d\sigma\leq C.
\end{split}
\end{equation}
Observe that $\frac{1-(1-\sigma)^{\gamma-1}}{(\gamma-1)\sigma}\rightarrow 1$ as $\sigma\rightarrow 0$.

When $\gamma\geq 2$,
we check directly that
\begin{equation}\label{CC329}
\int_Q
(\xi\sigma\rho+(1-\xi\sigma)\rho_s)^{\gamma-2}d\xi d\sigma\leq C(\rho+1)^{\gamma-2}\leq C(\rho+1)^{\gamma-1}.
\end{equation}

Combining \eqref{S338} with \eqref{CC329}, we declare that
$$0\leq H(\rho,\rho_s)\leq C(\rho+1)^{\gamma-1}(\rho-\rho_s)^2.$$
Moreover, Sobolev embedding theorem, Lemma \ref{L25}, and \eqref{CC330} guarantee
$$\|\mathcal{B}(\rho-\rho_s)\|_{L^\infty}\leq C\|\nabla \mathcal{B}(\rho-\rho_s)\|_{L^4}\leq C\|\rho-\rho_s\|_{L^4}\leq \delta.$$
Consequently, the second term in \eqref{CC332} satisfies
\begin{equation}\label{CC331}
\begin{split}
&\int_\Omega\nabla\rho_s^{-1}\cdot
H(\rho,\rho_s)\cdot\mathcal{B}(\rho-\rho_s)dx\leq C\big\|\mathcal{B}(\rho-\rho_s)\big\|_{L^\infty}A_2^2\leq \delta A_2^2.
\end{split}
\end{equation}

In contrast, the first term in \eqref{CC332} is directly given by
\begin{equation}\label{CC333}
\int_\Omega\rho_s^{-1}(\rho^\gamma-\rho_s^\gamma)(\rho-\rho_s) dx\geq CA_2^2.
\end{equation}
Substituting \eqref{CC331} and \eqref{CC333} into \eqref{CC328}, we arrive at
\begin{equation}\label{CC342}
A_2^2\leq C\int_\Omega R\cdot\mathcal{B}(\rho-\rho_s)\rho_s^{-1}dx.
\end{equation}
\\

\textit{Estimates of the remaining terms.}
The method is rather routine, we check each terms in details. Note that
$$\int_\Omega R\cdot\mathcal{B}(\rho-\rho_s)\rho_s^{-1}dx=\sum_{i=1}^4\int_\Omega R_i\cdot\mathcal{B}(\rho-\rho_s)\rho_s^{-1}dx=\sum_{k=1}^4J_k.$$

First, we introduce
$$D(t)=\int_\Omega \rho u\cdot\mathcal{B}(\rho_s-\rho)\cdot\rho_s^{-1}dx.$$
Observe that \eqref{S327} and Lemma \ref{L25} ensure
\begin{equation}\label{S337}
\left|D(t)\right|
\leq C\|\rho^{\frac{1}{2}}u\|_{L^2}\cdot\|\rho^{\frac{1}{2}}\|_{L^4}\cdot\|\mathcal{B}(\rho-\rho_s)\|_{L^4}\leq C.
\end{equation}
Hence, by virtue of \eqref{S101}$_1$, \eqref{S327}, and Lemma \ref{L25}, we declare
\begin{equation}\label{CC339}
\begin{split}
J_1
&=D(t)'+\int_\Omega\rho u\cdot\mathcal{B}(\mathrm{div}(\rho u))\cdot\rho_s^{-1}dx\\
&\leq D(t)'+C\|\rho u\|_{L^2}^2\leq D(t)'+CA_1^2,
\\
\end{split}
\end{equation}
where $A_1$ is given by \eqref{SS401}.

Similarly, in view of Lemma \ref{L25}, we declare that, for any $1<p<\infty$,
\begin{equation}\label{CC334}
\|\nabla\big(\mathcal{B}(\rho-\rho_s)\cdot\rho_s^{-1}\big)\|_{L^p}\leq C(p)\|\rho-\rho_s\|_{L^p}.
\end{equation}
Thus, we apply \eqref{S327}, Lemma \ref{L25} and set $p=4,2$ in \eqref{CC334} to give
\begin{equation}\label{CC340}
\begin{split}
J_2
&=\int_\Omega\rho u\otimes u\cdot\nabla\big(\mathcal{B}(\rho-\rho_s)\rho_s^{-1}\big)dx\leq C\|\rho\|_{L^4}^2\cdot\|u\|_{L^4}^2\leq CA_1^2,\\
J_3&=\mu\int_\Omega\nabla u\cdot\nabla\big(\mathcal{B}(\rho_s-\rho)\rho_s^{-1}\big)dx\leq C\|\nabla u\|_{L^2}\|\rho-\rho_s\|_{L^2}\leq CA_1^2+\varepsilon A_2^2.
\end{split}
\end{equation}

The estimate of $J_4$ requires more efforts. Note that
\begin{equation}\label{CC338}
\begin{split}
J_4&=\int_\Omega(\mu+\lambda)\cdot\mathrm{div}u\cdot\mathrm{div}\big(\mathcal{B}(\rho_s-\rho)\cdot\rho_s^{-1}\big)dx.\\
&\leq C\int_\Omega\bigg((\rho-M_1)_+^\beta+M_1^\beta\bigg)\cdot|\nabla u|\cdot|\nabla\big(\mathcal{B}(\rho_s-\rho)\cdot\rho_s^{-1}\big)|dx,
\end{split}
\end{equation}
where $M_1$ is given by Corollary \ref{L33}.

In particular, by choosing $p=\gamma+1>2$ in \eqref{CC334} and $q=2\beta(\gamma+1)/(\gamma-1)$, we calculate that
\begin{equation*}
\begin{split}
&\int_\Omega(\rho-M_1)_+^\beta\cdot|\nabla u|\cdot|\nabla\big(\mathcal{B}(\rho_s-\rho)\cdot\rho_s^{-1}\big)|dx\\
&\leq C\int_\Omega(\rho-M_1)_+^{q}dx+C\int_\Omega|\nabla u|^2dx+\varepsilon\int_\Omega|\rho-\rho_s|^{\gamma+1}dx\\
&\leq C\int_\Omega(\rho-M_1)_+^{q}dx+
CA_1^2+\varepsilon A_2^2,\\
\end{split}
\end{equation*}
which combined with \eqref{CC338} ensures
\begin{equation}\label{CC341}
J_4\leq C\int_\Omega(\rho-M_1)_+^{q}dx+C
A_1^2+\varepsilon A_2^2.
\end{equation}

In conclusion, putting \eqref{CC339}, \eqref{CC340}, and \eqref{CC341} into \eqref{CC342}, we chose $\varepsilon<\frac{C}{100}$ to argue that
\begin{equation}\label{S336}
\begin{split}
A_2^2
&\leq C\left(D(t)'+A_1^2+\int_\Omega(\rho-M_1)_+^{q}dx\right).
\end{split}
\end{equation}

Integrating \eqref{S336} with respect to $t$ and using  Proposition \ref{L31}, Corollary \ref{L33}, and \eqref{S337} yield
\begin{equation*}
\int_0^TA_2^2(t)\, dt\leq C,
\end{equation*}
which gives \eqref{S335}.
\end{proof}
\subsection{Application II: Further energy estimates}\label{SEC4}
\quad To get  more accurate descriptions about the large time behavior of the system, some further energy estimates are required. Let us first introduce some  related energy terms.

For $A_1$ given by \eqref{SS401} and $A_2$ given by \eqref{CC343}, we set
\begin{equation*}
\begin{split}
A(t)&\triangleq A_1(t)+A_2(t).
\end{split}
\end{equation*}
Note that \eqref{S301} together with \eqref{S335} guarantees, for some $C$ independent with $T$,
\begin{equation}\label{CC344}
\sup_{0\leq t\leq T}A_2(t)+\int_0^TA^2(t)\, dt\leq C.
\end{equation}

In addition, we also consider
\begin{equation*}
\begin{split}
B^2(t)&\triangleq\int_\Omega\rho|\dot{u}|^2(x,t)\, dx,\\
R_T&\triangleq 1+\|\rho\|_{L^\infty(\Omega\times[0,T])}.
\end{split}
\end{equation*}

We start by a proposition concerning the extra integrability of the momentum $\rho u$. Compared to \cite[Lemma 3.7]{HL}, we establish a uniform version here.
\begin{proposition}\label{L41}
Under the same assumption of Theorem \ref{T11}, there exists a constant $C$ depending on $\Omega$, $\mu$, $\beta$, $\gamma$, $\|\rho_0\|_{L^{\infty}}$, and $\|u_0\|_{H^1}$ such that
\begin{equation}\label{S405}
\sup_{0\leq t\leq T}\int_\Omega\rho | u| ^{2+\nu}dx\leq C ,
\end{equation}
where we define
\begin{equation}\label{S406}
\nu\triangleq R_T^{-\frac{\beta}{2}}\nu_{0},
\end{equation}
for some suitably small generic constant $\nu_0\in (0,1)$ depending only on $\mu$ and $\Omega.$
\end{proposition}
\begin{proof}
Following the proof of \cite[Lemma 3.7]{HL}, multiplying \eqref{S101}$_2$ by $(2+\nu)|u|^\nu u$  and integrating the resulting equality over $\Omega$, we arrive at
\begin{equation*}
\begin{split}
&\frac{d}{dt}\int_\Omega\rho |u|^{2+\nu}dx+ \int_\Omega|u|^\nu\bigg((1+\lambda)(\mathrm{div}u)^2+\omega^2\bigg)dx\leq C(K_1+K_2),\\
\end{split}
\end{equation*}
where the right side is given by
\begin{equation*}
\begin{split}
K_1&=\int_\Omega\nu\big(\lambda|\mathrm{div}u|+|\nabla u|\big)|u|^{\nu}|\nabla u|\, dx,\\
K_2&=\int_\Omega\bigg(\rho^{\gamma-1}+|\nabla f|+1\bigg)\cdot|\rho-\rho_s|\cdot\bigg(|u|^{\nu+1}+|u|^\nu|\nabla u|\bigg)\, dx.
\end{split}
\end{equation*}

First, note that \eqref{S406} ensures $\nu\lambda\leq \nu_0(\rho^\frac{\beta}{2}+1)$, hence we argue that
\begin{equation}\label{CC345}
\begin{split}
K_1&\leq\nu_0\int_\Omega|u| ^\nu\bigg(\lambda(\mathrm{div}u)^2 +|\nabla u|^2\bigg)dx\\
&\leq C\nu_0\int_\Omega|u|^\nu\bigg((1+\lambda)(\mathrm{div}u)^2+\omega^2\bigg)dx,
\end{split}
\end{equation}
where the last line is due to Corollary \ref{L42}.

Next, observe that if $0\leq \psi\in L^{2q}(\Omega)$, with $q=2/(1+\nu)$, we have
\begin{equation}\label{CC346}
\begin{split}
&\int_\Omega\psi\cdot|\rho-\rho_s|\cdot|u|^{\nu+1}dx\\
&\leq C\int_\Omega\left(|\rho-\rho_s|^{\frac{2}{1-\nu}}+\psi^{q}|u|^2\right)dx\\
&\leq C\big(1+\|\psi\|_{L^{2q}}^q\big)\cdot\int_\Omega\left(|\rho-\rho_s|^{\frac{2}{1-\nu}}+|\nabla u|^2\right)dx\\
&\leq CA^2,
\end{split}
\end{equation}
where in the last line, by choosing $\nu<(\gamma-1)/4$, we ensure that
$$|\rho-\rho_s|^{\frac{2}{1-\nu}}\leq(\rho+\rho_s)^{\frac{2\nu}{1-\nu}}|\rho-\rho_s|^2\leq C(\rho+1)^{\gamma-1}(\rho-\rho_s)^2.$$

Similar arguments also give
\begin{equation}\label{CC347}
\begin{split}
&\int_\Omega\psi\cdot|\rho-\rho_s|\cdot|u|^{\nu}|\nabla u|dx\leq CA^2.
\end{split}
\end{equation}

Now setting $\psi=(\rho^{\gamma-1}+|\nabla f|+1)$ in \eqref{CC346} and \eqref{CC347}, we arrive at
\begin{equation}\label{CC377}
K_2\leq CA^2.
\end{equation}

Combining \eqref{CC345} with \eqref{CC377}, by determining $\nu_0\leq (10C)^{-1}$, we argue that
\begin{equation*}
\begin{split}
&\frac{d}{dt}\int_\Omega\rho |u|^{2+\nu}dx+ \int_\Omega|u|^\nu\bigg((1+\lambda)(\mathrm{div}u)^2+\omega^2\bigg)dx\leq C A^2,\\
\end{split}
\end{equation*}
which together \eqref{CC344} yields \eqref{S405} and finishes the  proof of Proposition \ref{L41}.
\end{proof}

\begin{remark}
Note that the additional integrability guaranteed by \eqref{S405} is the key ingredient to get point-wise control of $G$, see Lemma \ref{L48}. Compared with \cite[Lemma 9]{FLL}, we improve it to the uniform version Lemma \ref{L41}.

We mention that the assumption $\gamma>1$ plays an important role in our calculation, because the extra order $(\gamma-1)$ makes it possible to apply H\"{o}lder's inequality to deal with the index $1+\nu$ in \eqref{CC346}.
\end{remark}

Next, we focus on the energy term $B(t)$. It lies in the central position of further arguments, since $B(t)$ controls $\|\nabla u\|_{L^p}$ especially for $p>2$.
\begin{proposition}\label{L43} Suppose $\kappa\triangleq 1-2/p\in(0,1)$, and $\varepsilon>0$, then there exists some positive constant $C$ depending only on $p$, $\varepsilon$, $\mu$, and $\Omega$ such that, for $t\in[0,T]$,
\begin{equation}\label{S412}
\|\nabla u\|_{L^{p} }\leq CR_T^{\kappa/2+\varepsilon}\cdot(1+A)^{1-\kappa}(1+A+B)^{\kappa}.
\end{equation}
In particular, when $\kappa\leq\varepsilon$ and $\gamma<2\beta$, we have
\begin{equation}\label{S413}
\|\nabla u\|_{L^{p}}\leq C R_T^{2\varepsilon}\cdot A^{1-\kappa}(1+A+B)^{\kappa}.
\end{equation}
\end{proposition}
\begin{proof}
We define the effective viscous flux by
$G\triangleq(2\mu+\lambda)\mathrm{div}u-(P-P_s)$
 with $P_s=\rho_s^\gamma$. According to \eqref{S327}, we check that
\begin{equation}\label{CC351}
\begin{split}
\left|\int_\Omega G dx\right|&\leq C\int_\Omega\bigg(\rho^\beta|\mathrm{div}u| +(\rho+1)^{\gamma-1}|\rho-\rho_s|\bigg)dx\leq CA.\\
\end{split}
\end{equation}
Similarly, with the help of \eqref{S327}, direct computations also give
\begin{equation}\label{CC348}
\begin{split}
\int_\Omega\omega^2 dx\leq A^2,\ \int_\Omega G^2dx \leq CR_T^{\beta+\gamma} A^2,\ \int_\Omega(2\mu+\lambda)^{-1}G^2dx\leq C(1+A)^2.
\end{split}
\end{equation}

We turn to the higher order estimates of $G$ and $\omega$. Note that according to boundary condition \eqref{S105}, \eqref{S101}$_2$ implies $\omega$ solves the Dirichlet problem:
\begin{equation}\label{S414}
\begin{cases}
  \mu\Delta\omega=\nabla^\bot \bigg(\rho\dot{u}-(\rho-\rho_s)\nabla f\bigg)& \mbox{ in } \Omega, \\
   \omega=-Ku\cdot n^\bot& \mbox{ on } \partial\Omega.  \end{cases}
\end{equation}

Then, for $1<p<2$, standard $L^p$ estimate of elliptic equations \eqref{S414} (see\cite{gilbarg2015elliptic}) implies that
\begin{equation}\label{CC372}
\begin{split}
\|\nabla G\|_{L^p}+\|\nabla\omega\|_{L^p}&\leq C\bigg(\|\rho\dot{u}\|_{L^p}+\|(\rho-\rho_s)\nabla f\|_{L^p}+\|\nabla u\|_{L^p}\bigg)\\
&\leq CR_T^{1/2}(A+B).\\
\end{split}
\end{equation}

However, when $p=2$, by choosing $\varepsilon<(\gamma-1)/(\gamma+1)$, we have
\begin{equation*}
\|(\rho-\rho_s)\nabla f\|_{L^2}\leq C\|f\|_{H^2}\cdot\|\rho-\rho_s\|_{L^{2/(1-\varepsilon)}}\leq CA^{1-\varepsilon}.
\end{equation*}
Consequently, the $L^2$ estimate is given by
\begin{equation}\label{CC368}
\begin{split}
\|\nabla G\|_{L^2}+\|\nabla\omega\|_{L^2}&\leq C\bigg(\|\rho\dot{u}\|_{L^2}+\|(\rho-\rho_s)\nabla f\|_{L^2}+\|\nabla u\|_{L^2}\bigg)\\
&\leq CR_T^{1/2}(A^{1-\varepsilon}+A+B).\\
\end{split}
\end{equation}

Combining \eqref{CC351} with \eqref{CC348}, \eqref{CC372}, \eqref{CC368}, and Poincar\'{e} inequality, for $1<p<2$ we declare that
\begin{equation}\label{S415}
\begin{split}
&\|G\|_{W^{1,p}}+\|\omega\|_{W^{1,p}}\leq CR_T^{1/2}(A+B),\\
&\|G\|_{H^1}+\|\omega\|_{H^1}
\leq CR_T^{1/2}(1+A+B).\\
\end{split}
\end{equation}

While for $2<p<\infty$, in view of \eqref{CC348}, \eqref{S415}, and Lemma \ref{L22}, we also declare that
\begin{equation}\label{CC349}
\|\omega\|_{L^p}\leq C \|\omega\|_{L^2}^{1-\kappa}\|\omega\|_{H^1}^\kappa\leq CR_T^{\kappa/2}\cdot A^{1-\kappa}(1+A+B)^{\kappa}.
\end{equation}
\\

\textit{Proof of \eqref{S412}.}
For general setting, Corollary \ref{L33} ensures that, for any $1\leq p<\infty,$ we have
$$\|(2\mu+\lambda)^{-1}(P-P_s)\|_{L^p}\leq C(p).$$
Moreover, with the help of \eqref{CC348}, \eqref{S415}, and Lemma \ref{L22}, we deduce that
\begin{equation}\label{CC350}
\begin{split}
\big\|(2\mu+\lambda)^{-1}G\big\|_{L^{p} }
&\leq C \bigg(\big\|(2\mu+\lambda)^{-1}G\big\|_{L^{2}}\bigg)^{1-\kappa-\varepsilon}\cdot\bigg(\|G\|_{L^{2(\kappa+\varepsilon)/\varepsilon}}\bigg)^{\kappa+\varepsilon}\\
&\leq C\big\|(2\mu+\lambda)^{-1}G\big\|_{L^{2}}^{1-\kappa-\varepsilon}\cdot\bigg(\|G\|_{L^{2}}^{\varepsilon} \|G\|_{H^1}^{\kappa}\bigg)\\
&\leq CR_T^{\kappa/2+\varepsilon}(1+A)^{1-\kappa}(1+A+B)^{\kappa},
\end{split}
\end{equation}

Consequently, in view of \eqref{S206}, \eqref{CC349}, and \eqref{CC350}, we deduce that
\begin{equation*}
\begin{split}
\|\nabla u\|_{L^{p} }&\leq C\bigg(\|\mathrm{div}u\|_{L^{p} }+\|\omega\|_{L^{p} }\bigg) \\
&\leq C \bigg(\|(2\mu+\lambda)^{-1}G\|_{L^{p} }+\|\omega\|_{L^{p} }+1\bigg) \\
&\leq CR_T^{\kappa/2+\varepsilon}(1+A)^{1-\kappa}(1+A+B)^{\kappa},
\end{split}
\end{equation*}
which is \eqref{S412}.
\\

\textit{Proof of \eqref{S413}.}
When $\kappa<\varepsilon$, and $\gamma<2\beta$, the estimate about $(2\mu+\lambda)^{-1}(P-P_s)$ can be improved. Let us check that
\begin{equation}\label{S421}
\begin{split}
\big\|(2\mu+\lambda)^{-1}(P-P_s)\big\|_{L^{p}}^p\leq C\int_\Omega\big(\rho+1\big)^{p(\gamma-\beta)-2}
\big(\rho-\rho_s\big)^2dx\leq CA^2.
\end{split}
\end{equation}
Note that, by choosing $\kappa<\varepsilon<(\gamma+1)^{-1}$ and $\gamma<2\beta$, we have $p(\gamma-\beta)-2<(\gamma-1).$

In particular, $\kappa<\varepsilon$ also ensures  \eqref{CC351} can be improved to
\begin{equation}\label{CC352}
\begin{split}
\big\|(2\mu+\lambda)^{-1}G\big\|_{L^{p}}
&\leq C\big\|(2\mu+\lambda)^{-1}G\big\|_{L^{2}}^{1-\kappa-\varepsilon}\cdot\bigg(\|G\|_{L^{2}}^{\varepsilon} \|G\|_{H^1}^{\kappa}\bigg)\\
&\leq CR_T^{2\varepsilon}\cdot A^{1-\kappa}(1+A+B)^{\kappa},
\end{split}
\end{equation}
since by setting $p=2$ in \eqref{S421}, we have
$$\|(2\mu+\lambda)^{-1}G\big\|_{L^{2}}\leq C\bigg(\|\mathrm{div}u\|_{L^2}+\big\|(2\mu+\lambda)^{-1}(P-P_s)\big\|_{L^{2}}\bigg)\leq CA.$$

Thus, by making use of \eqref{S206}, \eqref{CC349}, \eqref{S421}, and \eqref{CC352}, we derive that
\begin{equation*}
\begin{split}
\|\nabla u\|_{L^{p} }&\leq C\bigg(\|\mathrm{div}u\|_{L^{p} }+\|\omega\|_{L^{p} }\bigg) \\
&\leq C \bigg(\left\|(2\mu+\lambda)^{-1}\cdot(G+P-P_s)\right\|_{L^{p}}+\|\omega\|_{L^{p}}\bigg) \\
&\leq CR_T^{2\varepsilon}\cdot A^{1-\kappa}(1+A+B)^{\kappa}.
\end{split}
\end{equation*}
which is \eqref{S413}. We therefore finish the proof.
\end{proof}

We end this section by an energy estimate providing proper control of $B$.
\begin{proposition}\label{L44}
Under the same assumptions of Theorem \ref{T11}. For any $\varepsilon\in(0,1)$,  there is a constant $C $ depending only on $\varepsilon$, $\mu$, $\gamma$,  $\rho_0$, $u_0$, and $\Omega$ such that
\begin{equation}\label{S416}
\sup_{0\leq t\leq T}\mathrm{log}(C+ A^2) +\int^T_0 \frac{B^2}{1+A^2}dt
\leq C R_T^{1+\varepsilon }.
\end{equation}
\end{proposition}

\begin{proof}
We mention that Corollary \ref{L33} simplify the arguments given by \cite{FLL}. For convenient, we set $\mu=1$ in calculations. Let us first rewrite \eqref{S101}$_2$ as
\begin{equation}\label{S419}
\rho\dot{u}-\nabla G-\nabla^\bot\omega=(\rho-\rho_s)\nabla f.
\end{equation}
Multiplying \eqref{S419} by $2\dot{u}$ and integrating over $\Omega$  by parts give
\begin{equation}\label{CC354}
B^2+\int_\Omega\bigg(G\cdot\mathrm{div}\dot{u}+\omega\cdot\mathrm{rot}\dot{u}\bigg)dx=\int_\Omega(\rho-\rho_s)\dot{u}\cdot\nabla f\, dx+\int_{\partial\Omega}\dot{u}\cdot \bigg(G\cdot n+\omega\cdot n^\bot\bigg)dS.
\end{equation}
\\

\textit{The estimates on right side of \eqref{CC354}.}
We handle the external force term via
\begin{equation*}
\begin{split}
&\int_\Omega(\rho-\rho_s)\dot{u}\cdot\nabla f dx\\
&=\left(\int_\Omega(\rho-\rho_s)u\cdot\nabla fdx\right)'
+\int_\Omega\bigg(\mathrm{div}(\rho u)u+(\rho-\rho_s)(u\cdot\nabla)u\bigg)\cdot\nabla f dx\\
&\leq \left(\int_\Omega(\rho-\rho_s)u\cdot\nabla fdx\right)'+CA^2,
\end{split}
\end{equation*}
where in the last line, we have applied the fact
\begin{equation}\label{CC355}
\begin{split}
&\int_\Omega\bigg(\mathrm{div}(\rho u)u+(\rho-\rho_s)(u\cdot\nabla)u\bigg)\cdot\nabla f dx\\
&=-\int_\Omega\bigg(\rho u\otimes u\cdot\nabla^2 f+\rho_s u\cdot\nabla u\cdot\nabla f\bigg) dx\\
&\leq C\bigg(1+\|\rho\|_{L^4}\bigg)\|f\|_{H^2}\cdot\|\nabla u\|_{L^2}^2\leq CA^2.
\end{split}
\end{equation}

Moreover, according to \eqref{S327}, we check that
\begin{equation}\label{CC365}
\left|\int_\Omega(\rho-\rho_s)u\cdot\nabla f\, dx\right|
\leq \|f\|_{H^2}\|\rho-\rho_s\|_{L^2}\|u\|_{L^4}\leq C+\varepsilon A^2.
\end{equation}

Next, we need some calculations dealing boundary terms. To make process clear, we introduce some notations. Let
\begin{equation}\label{CC364}
\begin{split}
&\langle u,v\rangle\triangleq u\cdot v,\ \ \nabla_u v\triangleq u\cdot\nabla v,
\end{split}
\end{equation}
which denote the Euclidean inner product (of $u$ and $v$) and the directional (in $u$) derivative (of $v$) respectively.

Recalling that $n$ denotes the unit outer normal vector of $\partial\Omega$, and $n^\bot$ is a unit tangential vector of $\partial\Omega$. According to $\langle n,n\rangle=\langle n^\bot,n^\bot\rangle=1$,
we argue that the vector $\nabla_{n^\bot}n^\bot$ is along the normal direction, because
\begin{equation}\label{CC359}
2 \langle \nabla_{n^\bot}n^\bot, n^\bot\rangle=\nabla_{n^\bot}\langle n^\bot, n^\bot\rangle=0.
\end{equation}

We denote the projection of $u$ on $n^\bot$ by $\hat{u}=\langle u, n^\bot\rangle$ , which is a scaler. Note that \eqref{S105} gives $\langle u,n\rangle=0$, which means $u$ along the tangent direction. Thus we have
\begin{equation}\label{CC366}
u=\langle u, n^\bot\rangle n^\bot=\hat{u}\cdot n^\bot \ \ \mathrm{and}  \ \  \dot{u}=u_t+\nabla_uu.\end{equation}

Let us check that $\langle u,n\rangle=0$ provides that
\begin{equation}\label{CC356}
\begin{split}
\langle\nabla_uu,n\rangle
=\nabla_u\langle u,n\rangle-\langle u,\nabla_un\rangle=-\langle u,\nabla_un\rangle.\\
\end{split}
\end{equation}
Similarly we also have
\begin{equation}\label{CC357}
\begin{split}
\langle\nabla_uu,n^\bot\rangle=\nabla_u\hat{u}-\langle u,\nabla_un^\bot\rangle=\nabla_u\hat{u}= \nabla_{n^\bot}(\hat{u}^2/2),
\end{split}
\end{equation}
since $\nabla_un^\bot=\hat{u}\cdot(\nabla_{n^\bot}n^\bot)$ is along normal direction and $\nabla_u\hat{u}=\hat{u}\, (n^\bot\cdot\nabla)\hat{u}$ due to \eqref{CC359} and \eqref{CC366}.

Combining \eqref{CC356} with \eqref{CC357}, we argue that
\begin{equation}\label{CC358}
\begin{split}
&\langle\dot{u},n\rangle
=\partial_t\langle u,n\rangle+\langle\nabla_uu,n\rangle=-\langle u,\nabla_un\rangle,\\
&\langle\dot{u},n^\bot\rangle
=\partial_t\langle u,n^\bot\rangle+\langle\nabla_uu,n^\bot\rangle=\partial_t\hat{u}+\nabla_{n^\bot}(\hat{u}^2/2).
\end{split}
\end{equation}

Moreover, we will apply the following estimate of boundary term repeatedly, where the ideas are due to \cite{caili01}. According to \eqref{CC366}, we have
\begin{equation*}
\begin{split}
\int_{\partial\Omega} u\cdot\bigg(\nabla v\cdot w\bigg)dS
&=\int_{\partial\Omega}n^\bot\cdot\bigg(\nabla v\cdot w\cdot\hat{u}\bigg)dS\\
&=\int_{\partial\Omega}n \cdot\bigg(\nabla^\bot v\cdot w\cdot\hat{u}\bigg)dS\\
&=\int_\Omega\mathrm{div}\bigg(\nabla^\bot v\cdot w\cdot\hat{u}\bigg)dx.\\
\end{split}
\end{equation*}
Note that $\mathrm{div}\nabla^\bot=0$, hence we declare that
\begin{equation}\label{CCC401}
\bigg|\int_{\partial\Omega} u\cdot\bigg(\nabla v\cdot w\bigg)dS\bigg|\leq C\int_\Omega|\nabla v|(|w|+|\nabla w|)(|u|+|\nabla u|)dx.
\end{equation}

After these calculations, we turn to the boundary term of \eqref{CC354} which consists of two parts. In terminology of \eqref{CC364}, we have
$$
\int_{\partial\Omega}\dot{u}\cdot\bigg(G\cdot n+\omega\cdot n^\bot\bigg)dS=
\int_{\partial\Omega} G\langle\dot{u}, n\rangle\, dS+\int_{\partial\Omega}\omega\langle\dot{u},n^\bot\rangle dS.$$

In view of \eqref{CC358}$_1$ and \eqref{S415}, we argue that
\begin{equation}\label{CC361}
\begin{split}
\int_{\partial\Omega} G\langle\dot{u}, n\rangle\, dS&=\int_{\partial\Omega} -G\langle u, u\cdot\nabla n \rangle\, dS\\
&\leq C\|G\|_{H^1}\|\nabla u\|_{L^2}^2\\
&\leq C\bigg(R_T^{\frac{1}{2}}(1+A+B)\bigg)\cdot A^2\\
&\leq CR_T (1+A^2)A^2+\varepsilon B^2.
\end{split}
\end{equation}

Meanwhile, in view of \eqref{CC358}$_2$ and boundary condition \eqref{S105}, we declare that
\begin{equation}\label{CC363}
\begin{split}
\int_{\partial\Omega}\omega\langle\dot{u},n^\bot\rangle dS
&=-\int_{\partial\Omega}K\hat{u}\cdot\bigg(\partial_t\hat{u}+\nabla_{n^\bot}(\hat{u}^2/2)\bigg)dS\\
&=-\left(\int_{\partial\Omega}K\hat{u}^2/2\,dS\right)'-\int_{\partial\Omega}n^\bot\cdot \bigg(\nabla(\hat{u}^3/3)\cdot K\bigg) dS\\
&\leq -\left(\int_{\partial\Omega}K\hat{u}^2/2\,dS\right)'+CA^3,
\end{split}
\end{equation}
where in the last line, we have applied \eqref{CCC401} to declare that
\begin{equation*}
\begin{split}
\int_{\partial\Omega}n^\bot\cdot\bigg(\nabla(\hat{u}^3/3)\cdot K\bigg)
dS
&\leq C\int_\Omega|u|^2\cdot|\nabla u|dx\leq CA^3.
\end{split}
\end{equation*}

Combining \eqref{CC355}, \eqref{CC361}, and  \eqref{CC363} yields
\begin{equation}\label{CC362}
\mbox{The right side of \eqref{CC354}}\leq -E_1'+CR_T (1+A^2)A^2+\varepsilon B^2.
\end{equation}
We mention that by virtue of \eqref{CC365},
\begin{equation}\label{CC424}
E_1(t)=\int_{\partial\Omega}K\hat{u}^2/2\, dS-\int_\Omega(\rho-\rho_s)u\cdot\nabla f dx,
\end{equation}
which satisfies $-C-\varepsilon A^2\leq E_1\leq C(1+A^2).$
\\

\textit{The estimates on left side of \eqref{CC354}.}
Direct computations yield
\begin{equation*}
\begin{split}
\mathrm{div}\dot{u}&=\frac{d}{dt}\bigg(\frac{G+P-P_s}{2+\lambda}\bigg)+(\partial_1u\cdot\nabla) u_1+(\partial_2u\cdot\nabla) u_2,\\
\mathrm{rot}\dot{u}&=\frac{d}{dt}(\mathrm{rot}u)-(\partial_1u\cdot\nabla) u_2+(\partial_2u\cdot\nabla )u_1.
\end{split}
\end{equation*}
Therefore, the left side of  \eqref{CC354} is transformed into
\begin{equation}\label{CC373}
\int_\Omega\bigg(G\cdot\mathrm{div}\dot{u}+\omega\cdot\mathrm{rot}\dot{u}\bigg)dx+B^2\geq\frac{d}{dt}\int_\Omega\bigg(\frac{G^2}{2+\lambda}+\omega^2\bigg)\, dx +B^2-R.
\end{equation}

The remaining term $R$ is given by
\begin{equation*}
\begin{split}
&R=\int_\Omega\bigg(|\omega|+|G|\bigg)\cdot|\nabla u|^2dx+\int_\Omega \textbf{F}(\rho)\cdot G\cdot\bigg(|u|+|\nabla u|\bigg)dx\triangleq R_1+R_2.\\
\end{split}
\end{equation*}
with some positive function $\textbf{F}(x)\leq C(1+x^\gamma)$.

The estimate for $R_1$ is subtle. We still set $\kappa=1-2/p>1-\varepsilon$, then check that
\begin{equation}\label{CC370}
\begin{split}
R_1
&= \int\bigg(|\omega|+|G|\bigg)\cdot|\nabla u|^2 dx\leq\bigg(\|\omega\|_{L^p}+ \|G\|_{L^{p}}\bigg)\|\nabla u\|_{L^{2p/(p-1)}}^{2}.\\
\end{split}
\end{equation}

Note that in view of \eqref{CC348} and \eqref{CC368}, we have
\begin{equation}\label{CC369}
\begin{split}
\|\omega\|_{L^p}+\|G\|_{L^{p}}
&\leq C\bigg(\|\omega\|_{L^2}+\|G\|_{L^{2}}\bigg)^{1-\kappa}\bigg(\|\omega\|_{H^1}+\|G\|_{H^{1}}\bigg)^\kappa\\
&\leq CR_T^{\kappa/2+\varepsilon}\cdot
A^{1-\kappa}(A^{1-\varepsilon}+A+B)^\kappa.\end{split}
\end{equation}

Meanwhile, with the help of \eqref{S412}, for $\theta=(1-\kappa)/2\kappa$, we deduce that
\begin{equation}\label{CC371}
\begin{split}
\big\|\nabla u\big\|_{L^{2p/(p-1)}}&\leq\big\|\nabla u\big\|_{L^2}^{1-\theta}\cdot\big\|\nabla u\big\|_{L^p}^{\theta}\\
&\leq CR_T^{2\varepsilon}\cdot A^{1-\theta} \bigg( (1+A)^{(1-\kappa)}(1+A+B)^{\kappa}\bigg)^\theta,\\
\end{split}
\end{equation}
as $\kappa\theta=1-\kappa<\varepsilon$.

Now combining \eqref{CC370} with \eqref{CC369} and \eqref{CC371}, after some detailed calculations, we derive that
\begin{equation}\label{CC375}
\begin{split}
R_1
&\leq CR_T^{\kappa/2+\varepsilon}  A^{3-2\theta-\kappa}\big(A^{1-\varepsilon}+A+B\big)^\kappa\bigg( (1+A)^{(1-\kappa)}(1+A+B)^{\kappa}\bigg)^{2\theta}\\
&\leq CR_T^{1/2+\varepsilon}(A+A^2)\cdot\bigg(A+A^2+B\bigg) \\
&\leq CR_T^{1+\varepsilon}(1+A^2)A^2+\varepsilon B^2.
\end{split}
\end{equation}

Next we deal with $R_2$. Because of $\textbf{F}(x)\leq C(1+x^\gamma)$, Corollary \ref{L33} ensures $$\textbf{F}(\rho)\in L^\infty(0,T;L^p(\Omega)),$$
for any $1\leq p<\infty$ and $0\leq T<\infty$.

Thus, with the help of \eqref{S415}, we argue that
\begin{equation}\label{CC376}
\begin{split}
R_2&=\int_\Omega \textbf{F}(\rho)\cdot G\cdot\bigg(|u|+|\nabla u|\bigg)dx\\
&\leq C\,\|\textbf{F}(\rho)\|_{L^4}\,\|G\|_{W^{1,4/3}} \,\|\nabla u\|_{L^2}\\
&\leq CR_T^{1/2}\cdot A \bigg(A+B\bigg)\\
&\leq CR_TA^2+\varepsilon B^2.
\end{split}
\end{equation}

Substituting \eqref{CC375} and \eqref{CC376} into \eqref{CC373} provides
\begin{equation}\label{CC367}
\begin{split}
\mbox{The left side of \eqref{CC354}}
\geq E_2'+B^2-CR_T^{1+\varepsilon}(1+A^2)A^2.
\end{split}
\end{equation}
Note that by virtue of \eqref{S206} and \eqref{S327}, we deduce that
\begin{equation}\label{CC425}
E_2(t)=\int_\Omega\bigg(\frac{G^2}{2+\lambda}+\omega^2\bigg)dx,
\end{equation}
which satisfies $C(A^2-1)\leq E_2\leq C(A^2+1)$.

Combining \eqref{CC362} with \eqref{CC367} and setting $E\triangleq E_1+E_2$ give us
\begin{equation}\label{S422}
E'+B^2\leq CR_T^{1+\varepsilon}(1+A^2)A^2.
\end{equation}
Moreover, \eqref{CC424} and \eqref{CC425} ensure that we can find some constants $C$ and $N$ such that, for any $t\in[0,T]$,
\begin{equation*}
\frac{1}{N}\bigg(1+A^2(t)\bigg)\leq\bigg(C+E(t)\bigg)\leq N\bigg(1+A^2(t)\bigg).
\end{equation*}

Multiplying \eqref{S422} by $1/(C+E)$ and integrating with respect to $t$ yield
\begin{equation*}
\sup_{0\leq t\leq T}\log(C+A^2)+\int_0^T\frac{B^2}{1+A^2}dt\leq CR_T^{1+\varepsilon},
\end{equation*}
which gives \eqref{S416} and finishes the proof.
\end{proof}

\section{Proof of Theorem \ref{T11}}\label{S5}
\quad With all preparation done, we turn to the proof of main Theorem \ref{T11}. We derive the uniform upper bound of $\rho$ in Section \ref{SS41} and obtain the exponential decay of the system in Section \ref{SS42}.
\subsection{Uniform upper bound of density}\label{SS41}
\quad  In Section \ref{SEC3}, we establish the uniform $L^p$ estimate of $\rho$ for any $1\leq p<\infty$. However, we can not pass $p\rightarrow\infty$ directly, because the estimate is not uniform with respect to $p$.

Recalling that \eqref{S306} gives
\begin{equation*}
\frac{d}{dt}\theta(\rho)+(P-P_s) =-G,
\end{equation*}
where $G=(2\mu+\lambda)\mathrm{div}u-(P-P_s)$. Thus,
the key issue to derive $L^\infty$ estimates of $\rho$ is to obtain the point-wise control of $G$.

In view of \eqref{S101}$_2$ and \eqref{S105}, for $t\in[0,T]$, $G$ solves Neumann problem :
\begin{equation}\label{S429}
\begin{cases}
  \Delta G=\mathrm{div} H& \mbox{ in } \Omega, \\
   n\cdot\nabla G=n\cdot H-\mu \big(n^\bot\cdot\nabla\big)\omega & \mbox{ on } \partial\Omega,  \end{cases}
\end{equation}
where $H=\rho\dot{u}-(\rho-\rho_s)\nabla f$.
We adopt the Green function to get the point-wise representation of $G$. The method is based on our previous paper \cite{FLL}

Note that the Green function $N(x,y)$ for Neumann problem (see \cite{2016Representation}) on the unit disc $\mathbb{D}$ is given by
$$N(x,y)=-\frac{1}{2\pi}\bigg(\log|x-y|+\log\left||x|y-\frac{x}{|x|}\right|\bigg).$$
Let $\varphi=(\varphi_1, \varphi_2):\overline{\Omega}\rightarrow\overline{\mathbb{D}}$ be a conformal mapping guaranteed by Riemann mapping theorem (see \cite{2007Complex}). We define the pull back Green's function $\widetilde{N}(x,y)$ of $\Omega$ by:
\begin{equation*}
\begin{split}
\widetilde{N}(x,\, y)&=N\big(\varphi(x),\varphi(y)\big) \ \ \mathrm{for}\ x,y\in\Omega.
\end{split}
\end{equation*}

The pull back Green's function $\widetilde{N}$ serves as a fundamental solution as well, owing to the fact that conformal mapping preserves harmonicity. In particular, the point-wise representation of $G$ is given by next lemma, whose complete proof can be found in \cite[Lemma 13 \& Proposition 4]{FLL}.
\begin{lemma}\label{L46}
Suppose that $G\in C\big([0,T];C^1(\bar\Omega)\cap C^2(\Omega)\big)$ solves the problem \eqref{S429}. For any $x\in\Omega$, we have
\begin{equation} \label{CC401}
\begin{split}
G(x,t)
=&-\int_\Omega \nabla_y\widetilde{N}(x,\, y) \cdot H (y,t)\, dy
+\int_{\partial\Omega}\bigg(\frac{\partial \widetilde{N}}{\partial n}\cdot G+\mu\widetilde{N} \big(n^\bot\cdot\nabla\big)\omega \bigg)dS\\
=&-\frac{d}{dt}I(x,t)+J(x,t)+K(x,t).
\end{split}
\end{equation}
Especially, the time derivative term is given by
$$I(x,t)\triangleq\int_\Omega\nabla_y\widetilde{N}\big(x,y\big)\cdot \rho u(y,t)\, dy,$$
the principal term is defined by
$$J(x,t)\triangleq\int_\Omega\left(\partial_{x_iy_j}\widetilde{N}(x,y)\cdot u_i(x)+\partial_{y_iy_j}\widetilde{N}(x,y)\cdot u_i(y)\right)\rho u_j(y)dy,$$
while the remaining term satisfies
\begin{equation*}
\begin{split}
K(x,t)&\triangleq\int_\Omega\nabla_y\widetilde{N}\cdot(\rho_s-\rho)\nabla f\,dy+\int_{\partial\Omega}\bigg(\frac{\partial \widetilde{N}}{\partial n}\cdot G+\mu\widetilde{N} \big(n^\bot\cdot\nabla\big)\omega \bigg)dS.
\end{split}
\end{equation*}
\end{lemma}

The key step to get uniform upper bound of  $\rho$ is to declare the following proposition.
\begin{proposition}\label{L47}
Under the conditions of Theorem \ref{T11}, for any $\varepsilon>0$ and $0\leq t_1\leq t_2<\infty$, there is a constant $C$ depending only on $\varepsilon$, $\beta$, $\mu$, $\rho_0$, $u_0$, and $\Omega$, such that

\noindent when $\gamma<2\beta$, we have
\begin{equation}\label{S432}
\int_{t_1}^{t_2}-G(x(t),t)\, dt\leq CR_T^{1+\varepsilon}(t_2-t_1)+CR_T^{1+\frac{\beta}{4}+2\varepsilon},
\end{equation}
while for $\gamma\geq 2\beta$, we have
\begin{equation}\label{CC412}
\int_{t_1}^{t_2}-G(x(t),t)\, dt\leq CR_T^{1+\frac{\beta}{4}+2\varepsilon}\cdot\bigg(t_2-t_1+1\bigg),
\end{equation}
where $x(t)$ is the flow line determined by $x(t)'=u(x(t),t)$.
\end{proposition}
\begin{proof}
Set $I(t)=I(x(t),t)$,
then \eqref{CC401} implies
\begin{equation*}
\begin{split}
-G\big(x(t),t\big)&=I'(t)-J\big(x(t),t\big)-K\big(x(t),t\big)\\
&\leq I'(t)+\|J(\cdot, t)\|_{L^\infty(\Omega)}+\|K(\cdot, t)\|_{L^\infty(\Omega)}.
\end{split}
\end{equation*}
We check each term in details.
\\

\textit{Estimates for $I(t)$.} With the help Proposition \ref{L41}, for $\nu=R_T^{-\frac{\beta}{2}}\nu_0$, we choose $q=(2+\nu)/(1+\nu)<2$ and declare that
\begin{equation}\label{CC435}
\begin{split}
\left|\int_\Omega\nabla_{y}\tilde{N}\cdot(\rho u)\, dy\right|
&\leq C\left(\int_\Omega|x-y|^{-q}dy\right)^{1/q}\left(\int_\Omega \rho^{2+\nu}|u|^{2+\nu}dy\right)^{\frac{1}{2+\nu}}\\
&\leq C\cdot\nu^{-1/q}R_T^{1/q}\left(\int_\Omega \rho |u|^{2+\nu}dy\right)^{\frac{1}{2+\nu}}\leq CR_T^{1+\frac{\beta}{4}+\varepsilon}.
\end{split}
\end{equation}

Consequently, we check that
\begin{equation}\label{CC402}
\int_{t_1}^{t_2}I'(t)dt=I(t_2)-I(t_1)\leq CR_T^{1+\frac{\beta}{4}+\varepsilon}.
\end{equation}
\\

\textit{Estimates for $\|K\|_{L^\infty}$.} Recalling that
$$K=\int_\Omega\nabla_y\widetilde{N}\cdot(\rho_s-\rho)\nabla f\,dy+\int_{\partial\Omega}\bigg(\mu\widetilde{N} \big(n^\bot\cdot\nabla\big)\omega+\frac{\partial \widetilde{N}}{\partial n}\cdot G\bigg)dS. $$

For the first term, by taking advantage of Proposition \ref{L33}, we deduce that
\begin{equation}\label{CC403}
\begin{split}
\int_\Omega\big|\nabla_y\widetilde{N}\cdot(\rho_s-\rho)\nabla f\big|dy
\leq C\|f\|_{H^2}\cdot\|\nabla_y\tilde{N}\|_{L^{4/3}}
\|\rho-\rho_s\|_{L^8}
\leq C.
\end{split}
\end{equation}

For boundary term, in view of \eqref{S105}, we  apply \eqref{CCC401} and set $p=4$ in Proposition \ref{L43} to declare that
\begin{equation}\label{CC405}
\begin{split}
\int_{\partial\Omega}\widetilde{N} \big(n^\bot\cdot\nabla\big)\omega\, dS
&=\int_{\partial\Omega}\widetilde{N}\big(n^\bot\cdot\nabla\big)(u\cdot n^\bot)\, dS\\
&\leq C \int_{\Omega}(|\widetilde{N}|+|\nabla\widetilde{N})|\cdot(|\nabla u|+|u|)\, dy,\\
&\leq C\bigg(\|\widetilde{N}\|_{L^{4/3}}+\|\nabla\widetilde{N}\|_{L^{4/3}}\bigg)\|\nabla u\|_{L^4}\\
&\leq CR_T^{1/2+\varepsilon}\bigg(1+A+B\bigg).
\end{split}
\end{equation}

Moreover, \cite[Lemma 12]{FLL} ensures that $\partial_n\widetilde{N}$ is bounded uniformly on $\partial\Omega$, thus in view of \eqref{S415}, we also have
\begin{equation}\label{CC406}
\int_{\partial\Omega}\left|\partial_n\widetilde{N}\cdot G\right|dS\leq C\int_{\partial\Omega}\big|G\big|\, dS\leq C\big\|G\big\|_{H^1}\leq CR_T^{1/2}\bigg(1+A+B\bigg).
\end{equation}

Combining \eqref{CC344} with \eqref{CC403}--\eqref{CC406}, and \eqref{S416}, we arrive at
\begin{equation}\label{CC409}
\begin{split}
\int_{t_1}^{t_2}\|K(\cdot, t)\|_{L^\infty(\Omega)}dt
&\leq C\int_{t_1}^{t_2}R_T^{1/2+\varepsilon}\bigg(1+A+B\bigg)dt\\
&\leq C\int_{t_1}^{t_2}\bigg(R_T^{1+\varepsilon}(1+A^2)+\frac{B^2}{1+A^2}\bigg)dt\\
&\leq CR_T^{1+\varepsilon}(t_2-t_1+1).
\end{split}
\end{equation}

\textit{Estimates for $\|J\|_{L^\infty}$.} According to \cite[Proposition 4]{FLL}, we have
\begin{equation*}
\begin{split}
\|J(\cdot, t)\|_{L^\infty(\Omega)}&\leq C\left(\sup_{x\in\overline{\Omega}}\int_\Omega\frac{\rho|u|^2(y)}{|x-y|}\, dy
+\sup_{x\in\overline{\Omega}}\int_\Omega\frac{|u(x)-u(y)|}{|x-y|^2} \cdot \rho|u|(y)\, dy\right),
\end{split}
\end{equation*}
where $x'\triangleq\varphi^{-1}\bigg(\varphi(x)/|\varphi(x)|\bigg)$.

Note that Proposition \ref{L33} and Sobolev embedding theorem ensure that
\begin{equation}\label{CC408}
\begin{split}
&\int_{t_1}^{t_2}\sup_{x\in\overline{\Omega}}\bigg(\int_\Omega\frac{\rho|u|^2(y)}{|x-y|}dy\bigg)dt\\
&\leq C\int_{t_1}^{t_2}\left(\int_\Omega |y|^{-3/4}dy\right)^{\frac{1}{4}} \|\rho\|_{L^8}\cdot\|u\|_{L^{16}}^2dt\\
&\leq\int_0^T\|\nabla u\|_{L^2}^2dt\leq C.
\end{split}
\end{equation}

In addition, next Lemma \ref{L47} guarantees that, when $\gamma<2\beta$ we have
\begin{equation}\label{CC410}
\int_{t_1}^{t_2}\sup_{x\in\overline{\Omega}}\left(\int_\Omega\frac{|u(x)-u(y)|}{|x-y|^2}\rho|u|(y)dy\right)dt\leq CR_T^{1+\varepsilon}(t_2-t_1)+C
R_T^{1+\frac{\beta}{4}+\varepsilon},
\end{equation}
while for $\gamma\geq 2\beta$, we have
\begin{equation}\label{CC413}
\int_{t_1}^{t_2}\sup_{x\in\overline{\Omega}}\left(\int_\Omega\frac{|u(x)-u(y)|}{|x-y|^2}\rho|u|(y)dy\right)dt\leq CR_T^{1+\frac{\beta}{4}+2\varepsilon}\bigg(t_2-t_1+1\bigg).
\end{equation}

Combining \eqref{CC408} and \eqref{CC410} provides that, when $\gamma<2\beta$
$$\int_{t_1}^{t_2}\|J(\cdot,t)\|_{L^\infty(\Omega)}\,dt\leq CR_T^{1+\varepsilon}(t_2-t_1)+CR_T^{1+\frac{\beta}{4}+\varepsilon},$$
which together with \eqref{CC402} and \eqref{CC409} yields \eqref{S432}.

Similarly, combining \eqref{CC402}, \eqref{CC409}, \eqref{CC408}, and \eqref{CC413} also leads to \eqref{CC412}. We therefore finish the proof of Proposition \ref{L47}.
\end{proof}

Now we derive the crucial estimates \eqref{CC410} and \eqref{CC413} via detailed discussions.
\begin{lemma}\label{L48}
Under the same setting as Proposition \ref{L47}. There is a constant $C$ depending only on $\varepsilon$, $\beta$, $\mu$, $\rho_0$, $u_0$, and $\Omega$, for any $0\leq t_1\leq t_2<\infty$, such that

\noindent when $\gamma<2\beta$, we have
\begin{equation*}
\int_{t_1}^{t_2}\sup_{x\in\overline{\Omega}}\left(\int_\Omega\frac{|u(x)-u(y)|}{|x-y|^2}\rho|u|(y)dy\right)dt\leq CR_T^{1+\varepsilon}(t_2-t_1)+C
R_T^{1+\frac{\beta}{4}+2\varepsilon},
\end{equation*}
while for $\gamma\geq 2\beta$, we have
\begin{equation*}
\int_{t_1}^{t_2}\sup_{x\in\overline{\Omega}}\left(\int_\Omega\frac{|u(x)-u(y)|}{|x-y|^2}\rho|u|(y)dy\right)dt\leq CR_T^{1+\frac{\beta}{4}+2\varepsilon}\bigg(t_2-t_1+1\bigg).
\end{equation*}
\end{lemma}
\begin{proof}
We still set $\kappa=1-2/p\leq\varepsilon$.
 According to Morry's inequality (see \cite[Chapter 5, Theorem 5]{2010Partial}), there is a constant $C$ depending only on $p$ and $\Omega$ such that, for any $u\in W^{1,p}(\Omega)$ and $x,\, y\in \overline{\Omega}$, we have
\begin{equation*}
\big|u(x)-u(y)\big| \leq C\|\nabla u\|_{L^p}\cdot|x-y| ^{\kappa},
\end{equation*}
which leads to
\begin{equation}\label{S434}
\begin{split}
\int_\Omega\frac{|u(x)-u(y)|}{|x-y|^2}\rho|u|(y)dy
\leq C\|\nabla u\|_{L^p}\cdot\int_\Omega |x-y|^{-2+\kappa} \rho|u|(y)dy.
\end{split}
\end{equation}

Then, for $\delta>0 $ and $q<\kappa/2$, which will be determined later, on the one  hand,  we use  \eqref{S202}  to get
\begin{equation}\label{S435}
\begin{split}
&\int_{| x-y| <2\delta}| x-y| ^{\kappa-2}\rho| u| (y)\, dy\\
&\leq CR_T
\left(\int_{|y| <2\delta} |y| ^{(\kappa-2)/(1-q)} dy\right)^{1-q}\cdot
\| u\|_{L^{1/q}}\\
&\leq CR_T
\delta^{\kappa-2q}\cdot \bigg(q^{-1/2}\| u\|_{H^1}\bigg)\\
&\leq C q^{-1/2}R_T \cdot\bigg(A_1
\delta^{\kappa-2q}\bigg).
\end{split}
\end{equation}
On the other hand, for $\nu=R_T^{-\frac{\beta}{2}}\nu_{0}$ as in \eqref{S406}, we use Proposition \ref{L41} to derive
\begin{equation*}
\begin{split}
&\int_{|x-y| >\delta}|x-y| ^{\kappa-2}\cdot\rho| u| (y)\, dy\\
&\leq
C\left(\int_{|y| >\delta} |y| ^{\frac{2+\nu}{1+\nu}\cdot(2-\kappa)} dy\right)^{\frac{1+\nu}{2+\nu}}
\left(\int_\Omega\rho^{2+\nu}|u|^{2+\nu}dx\right)^{\frac{1}{2+\nu}}\\
&\leq CR_T \cdot
\delta^{\kappa-2/(2+\nu)}.
\end{split}
\end{equation*}

Now, we choose $\delta >0$ such that
\begin{equation*}
 \delta^{\kappa-2/(2+\nu)} =
A_1^{1-\kappa},
\end{equation*}
which in particular implies
\begin{equation}
\label{S438}
A_1\delta^{\kappa-2q} =A_1^{1-\kappa},
\end{equation}
provided we set
\begin{equation}\label{S439}
2q=\frac{\nu}{2+\nu}\cdot\frac{\kappa}{1-\kappa} \in \left(0, \kappa\right).
\end{equation}

Then, it follows from \eqref{S435}--\eqref{S438} that
\begin{equation}\label{S440}
\begin{split}
&\int_\Omega |x-y|^{\kappa-2} \rho|u|(y)\,dy\\
&\leq \left(\int_{|x-y|<2\delta}+\int_{|x-y|>\delta}\right)|x-y|^{\kappa-2}\rho|u|(y)\,dy\\
&\leq C \bigg(q^{-1/2}R_T\bigg)\cdot A_1^{1-\kappa}\leq CR_T^{1+\beta/4}\cdot A_1^{1-\kappa} ,
\end{split}
\end{equation}
where in the last line we have used $q^{-1/2}\leq C(p)\nu^{-1/2}\leq C(p)R_T^{\beta/4} $ due to \eqref{S439}. Now there are two cases we must distinguish.
\\

\textit{Case 1. $\gamma<2\beta$.}
By virtue of \eqref{S413}, for $\kappa<\varepsilon$ and $\gamma<2\beta$, we have
\begin{equation*}
\begin{split}
\|\nabla u\|_{L^p}
\leq C R_T^{\varepsilon}\cdot A^{1-\kappa}(1+A+B)^{\kappa},
\end{split}
\end{equation*}
which together with \eqref{S434} and \eqref{S440} gives
\begin{equation*}
\begin{split}
&\int_\Omega\frac{|u(x)-u(y)|}{|x-y|^2}\cdot\rho|u|(y)\, dy\\
&\leq CR_T^{1+\frac{\beta}{4}+\varepsilon}A^{2-2\kappa}(1+A+B)^\kappa\\
&\leq C\bigg(1+R_T^{1+\frac{\beta}{4}+2\varepsilon}A^2+\frac{B^2}{1+A^2}\bigg),
\end{split}
\end{equation*}
provided $\kappa<\varepsilon/(4+\beta).$

Integrating the above inequality with respect to $t$, we apply \eqref{S416} and arrive at
\begin{equation*}
\begin{split}
&\int_{t_1}^{t_2}\sup_{x\in\Omega}\left(\int_\Omega\frac{|u(x)-u(y)|}{|x-y|^2}\cdot\rho|u|(y)\, dy\right) dt\\
&\leq C\int_{t_1}^{t_2}\bigg(1+R_T^{1+\frac{\beta}{4}+2\varepsilon}A^2+\frac{B^2}{1+A^2}\bigg)\, dt\\
&\leq C\bigg(t_2-t_1
+R_T^{1+\frac{\beta}{4}+2\varepsilon}\bigg),
\end{split}
\end{equation*}
which finishes the case $\gamma<2\beta$.
\\

\textit{Case 2. $\gamma\geq 2\beta$.}
The arguments are similar. Observing that $\kappa<\varepsilon$, we apply \eqref{S412} and deduce that
$$\|\nabla u\|_{L^{p}}\leq CR_T^{\varepsilon}\cdot (1+A)^{1-\kappa}(1+A+B)^{\kappa},$$
which together with \eqref{S434} and \eqref{S440} gives
\begin{equation*}
\begin{split}
&\int_\Omega\frac{|u(x)-u(y)|}{|x-y|^2}\cdot\rho|u|(y)\, dy\\
&\leq CR_T^{1+\frac{\beta}{4}+\varepsilon}(A^{1-\kappa}+A^{2-2\kappa})(1+A+B)^\kappa\\
&\leq C\bigg(R_T^{1+\frac{\beta}{4}+2\varepsilon}(1+A^2)+\frac{B^2}{1+A^2}\bigg).
\end{split}
\end{equation*}
provided $\kappa<\varepsilon/(4+\beta).$

Integrating the above inequality with respect to $t$, we apply \eqref{S416} and arrive at
\begin{equation*}
\begin{split}
&\int_{t_1}^{t_2}\sup_{x\in\overline{\Omega}}\left(\int_\Omega\frac{|u(x)-u(y)|}{|x-y|^2}\cdot\rho|u|(y)\, dy\right) dt\\
&\leq C\int_{t_1}^{t_2}\bigg(R_T^{1+\frac{\beta}{4}+2\varepsilon}(1+A^2)+\frac{B^2}{1+A^2}\bigg)\, dt\\
&\leq CR_T^{1+\frac{\beta}{4}+2\varepsilon}\bigg(t_2-t_1
+1\bigg),
\end{split}
\end{equation*}
which finishes the case $\gamma>2\beta$.
We therefore complete the proof of Lemma \ref{L48}, thus Proposition \ref{L47} is valid as well.
\end{proof}

With all preparations done, we can derive the uniform upper bound of $\rho$.
\\

{\it{ Proof of \eqref{S111}.}}
Recalling that \eqref{S306} gives
\begin{equation}\label{S524}
\frac{d}{dt}\theta(\rho)+(P-P_s) =-G,
\end{equation}
where $G=(2\mu+\lambda)\mathrm{div}u-(P-P_s)$. Then, there are two cases we must investigate.
\\

\textit{Case 1. $\gamma<2\beta$.} Note that \eqref{S432} ensures
$$\int_{t_1}^{t_2}-G(x(t),t)\, dt\leq CR_T^{1+\varepsilon}(t_2-t_1)+CR_T^{1+\frac{\beta}{4}+2\varepsilon}.$$
Thus we integrate \eqref{S524} with respect to $t$, and apply the Zlotnik's inequality in Lemma \ref{L26} together with \eqref{S432} to infer that
\begin{equation}\label{S533}
R_T^\beta\leq CR_T^{\max\{1+\frac{\beta}{4}+2\varepsilon,(1+\varepsilon)\frac{\beta}{\gamma}\}}.
\end{equation}

Recalling that $\beta>\frac{4}{3}$ and $\gamma>1$, we set $\varepsilon<\min\{(3\beta-4)/8,\, \gamma-1\}$,  then \eqref{S533} in particular implies
\begin{equation*}
\sup_{0\leq t\leq T}\|\rho\|_{L^\infty}\leq C.
\end{equation*}
\\

\textit{Case 2. $\gamma\geq 2\beta$.} Note that \eqref{CC412} ensures
$$\int_{t_1}^{t_2}-G(x(t),t)\, dt\leq CR_T^{1+\frac{\beta}{4}+2\varepsilon}\bigg(t_2-t_1+1\bigg),
$$
Then we still integrate \eqref{S524} with respect to $t$, and make use of the Zlotnik's inequality in Lemma \ref{L26} together with \eqref{CC412} to give 
\begin{equation}\label{CC414}
R_T^\beta\leq CR_T^{\max\{1+\frac{\beta}{4}+2\varepsilon,\,(1+\frac{\beta}{4}+2\varepsilon)\frac{\beta}{\gamma}\}}.
\end{equation}
However, $\gamma>2\beta>2$ enforces $(1+\frac{\beta}{4}+2\varepsilon)\cdot\beta/\gamma\leq(1+\frac{\beta}{4}+2\varepsilon)$. Thus, \eqref{CC414} reduces to
\begin{equation}\label{CC415}
R_T^\beta\leq CR_T^{1+\frac{\beta}{4}
+2\varepsilon}.
\end{equation}

Recalling that $\beta>4/3$, we set $\varepsilon<(3\beta-4)/8$, and \eqref{CC415} makes sure
\begin{equation*}
\sup_{0\leq t\leq T}\|\rho\|_{L^\infty}\leq C.
\end{equation*}

We therefore finish the proof of \eqref{S111}.
\thatsall

\subsection{Exponential decay}\label{SS42}
\quad We apply the uniform upper bound of $\rho$ \eqref{S111} to derive the exponential decay of the system \eqref{S112} in several steps. The first Proposition deals with the density.

\begin{proposition}\label{P42}
Suppose \eqref{CC330} is valid, then under the conditions of Theorem \ref{T11}, $\forall T\in[0,\infty)$ and $\forall p\in[1,\infty)$, there are constants $C$ and $\xi_1$ depending only on $p$, $\beta$, $\gamma$, $\mu$, $\rho_0$, $u_0$, and $\Omega$, such that
\begin{equation}\label{CC416}
\sup_{0\leq t\leq T}e^{\xi_1 t}\bigg(\|\rho-\rho_s\|_{L^p}+\|\nabla u\|_{L^2}\bigg) \leq C.
\end{equation}
\end{proposition}
\begin{proof}
In view of \eqref{S111}, once we show \eqref{CC416} is true for $p=2$, then by interpolation, \eqref{CC416} is valid for other $1\leq p<\infty$ as well. Thus we will derive
$$\sup_{0\leq t\leq T}e^{\xi_1 t}\bigg(\|\rho-\rho_s\|_{L^2}+\|\nabla u\|_{L^2}\bigg) \leq C.$$

\textit{Step 1.} We establish the exponential decay of $\|\rho-\rho_s\|_{L^2}$.
By the same arguments as in Proposition \ref{L35}, we rewrite \eqref{S101}$_2$ as
\begin{equation}\label{CC417}
\nabla(\rho^\gamma-\rho_s^\gamma)+(1-\rho\cdot\rho_s^{-1})\nabla\rho_s^\gamma=R,
\end{equation}
where the remainder $R$ is given by
$$R=-\frac{\partial}{\partial t}(\rho u)-\mathrm{div}(\rho u\otimes u)+\mu\Delta u+\nabla\big((\mu+\lambda)\mathrm{div}u\big)=\sum_{i=1}^4R_i.$$

Note that $\mathcal{B}(\rho-\rho_s)$ is well defined, where $\mathcal{B}$ is given by Lemma \ref{L24}. Thus multiplying \eqref{CC417} by $\rho_s^{-1}\mathcal{B}(\rho-\rho_s)$ and integrating over $\Omega$, we take advantage of the
arguments through \eqref{CC331}--\eqref{CC341} to argue that, for $\varepsilon<(2C)^{-1}$,
\begin{equation}\label{CC434}
A_2^2\leq D'+C\bigg(A_1^2+J_4\bigg).
\end{equation}
Recalling that
\begin{equation*}
\begin{split}
A_1^2&=\int_\Omega\bigg((\mu+\lambda)(\mathrm{div}u)^2+|\nabla u|^2\bigg)dx,\\
A_2^2&=\int_\Omega(\rho+1)^{\gamma-1}(\rho-\rho_s)^2dx,\\
D&=\int_\Omega\rho u\cdot\mathcal{B}(\rho_s-\rho)\cdot\rho_s^{-1}dx,\\
J_4&=\int_\Omega(\mu+\lambda)\mathrm{div}u\cdot\mathrm{div}\bigg(\mathcal{B}(\rho_s-\rho)\cdot\rho_s^{-1}\bigg)dx.
\end{split}
\end{equation*}

Observe that by virtue of \eqref{S111} and Lemma \ref{L24}, we have
\begin{equation}\label{CC421}
\big|D(t)\big|\leq C\bigg(\|\rho^{\frac{1}{2}}u\|_{L^2}^2+\|\rho-\rho_s\|_{L^2}^2\bigg).
\end{equation}

Moreover, different from \eqref{CC341}, note that \eqref{S111} together with Lemma \ref{L24} directly implies that
$$J_4\leq CA_1^2+ \varepsilon A_2^2.$$
which substituted into \eqref{CC434} yields
\begin{equation}\label{CC418}
A_2^2\leq D'+CA_1^2.
\end{equation}

Then, let us go back to Proposition \ref{L31}. Note that \eqref{S304} provides
\begin{equation}\label{CC419}
\begin{split}
\textbf{E}(t)'+A_1^2\leq 0,
\end{split}
\end{equation}
where the energy term $\textbf{E}(t)$ can be written as
$$\textbf{E}(t)=\int_\Omega\bigg(\frac{1}{2}\rho|u|^2+\frac{1}{\gamma-1}(\rho^\gamma-\rho_s^\gamma)-(\rho-\rho_s)f\bigg)dx.$$

Observe that Corollary \ref{C21} and \eqref{CCC101} ensure that, for some constant $C_0$,
$$\frac{\gamma}{\gamma-1}\rho_s^{\gamma-1}=f+C_0.$$
Consequently, the conservation of mass gives
\begin{equation*}
\begin{split}
\int_\Omega f\cdot(\rho-\rho_s)dx=\int_\Omega(f+C_0)\cdot (\rho-\rho_s)dx=\int_\Omega\frac{\gamma}{\gamma-1}\rho_s^{\gamma-1}\cdot(\rho-\rho_s)dx.
\end{split}
\end{equation*}

Thus, we transform $\textbf{E}'(t)$ into
\begin{equation*}
\begin{split}
\textbf{E}'(t)&=\frac{d}{dt}\int_\Omega\bigg(\frac{1}{2}\rho|u|^2+\frac{1}{\gamma-1}\big(\rho^\gamma-\rho_s^\gamma-\gamma\rho_s^{\gamma-1}(\rho-\rho_s)\big)\bigg)dx\\
&=\frac{d}{dt}\int_\Omega\bigg(\frac{1}{2}\rho|u|^2+\frac{1}{\gamma-1}H(\rho,\rho_s)\bigg)dx.
\end{split}
\end{equation*}

Moreover, according to the discussion in \eqref{S338} and \eqref{CC329}, when \eqref{S111} is valid, we declare that there is a constant $C$ such that
\begin{equation}\label{CC420}
C^{-1}(\rho-\rho_s)^2\leq H(\rho,\, \rho_s)\leq C(\rho-\rho_s)^2.
\end{equation}

Now, multiplying \eqref{CC418} by $\delta$ and adding the resulting inequality to \eqref{CC419} give us
\begin{equation}\label{CC422}
\textbf{E}_0'+CA^2\leq 0,
\end{equation}
where  $\textbf{E}_0(t)=\textbf{E}(t)-\delta D(t)$. By setting $\delta<(4C)^{-1}$, \eqref{CC421} together with \eqref{CC420} yields
$$0\leq \textbf{E}_0\leq CA^2.$$

Thus, we apply Gronwall's inequality to \eqref{CC422} and deduce that, for $0\leq T<\infty$, there are constants $C$ and $\xi_1$,
\begin{equation}\label{CC427}
\sup_{0\leq t\leq T}\bigg(e^{4\xi_1 t}\|\rho-\rho_s\|_{L^2}^2\bigg)+\int_0^Te^{2\xi_1 t}A^2dt \leq C.
\end{equation}

\textit{Step 2.} We derive the exponential decay of $\|\nabla u\|_{L^2}$.
Let us go back to Proposition \ref{L44} and apply \eqref{S422} together with \eqref{S111} and \eqref{S416} to give 
\begin{equation}\label{CC423}
E'(t)+B^2\leq CA^2,
\end{equation}
where $E(t)=E_1(t)+E_2(t)$ is given by \eqref{CC424} and \eqref{CC425}.

According to \eqref{S111}, we also have $E(t)\leq A^2(t)$. Thus, multiplying \eqref{CC423} by $e^{2\xi_1 t}$ and integrating with respect to $t$, we apply \eqref{CC427} to deduce that, for $0\leq T<\infty$,
\begin{equation}\label{CC429}
\sup_{0\leq t\leq T}\bigg(e^{2\xi_1 t} E(t)\bigg)+\int_0^Te^{2\xi_1 t}B^2\, dt\leq C,
\end{equation}
which together with Lemma \ref{L24} gives
\begin{equation}\label{CC426}
\sup_{0\leq t\leq T}\bigg(e^{2\xi_1 t}\|\nabla u\|_{L^2}^2\bigg)\leq \sup_{0\leq t\leq T}Ce^{2\xi_1 t}\bigg(E(t)+\|\rho-\rho_s\|_{L^2}^2\bigg)\leq C.
\end{equation}

Combining \eqref{CC427} with \eqref{CC426}, we arrive at \eqref{CC416} for $p=2$. The general cases follow from interpolation arguments, hence we  finish the proof.
\end{proof}

The final step is to declare $\|\nabla u\|_{L^p}$ is uniformly bounded for any $1\leq p<\infty$. Once it is done, \eqref{S112} follows from  \eqref{CC416} and the interpolation arguments.

We need next uniform estimates which improve \cite[Lemma 15]{FLL}. The calculations are rather standard at present stage, so we merely sketch it. The complete details can be found in \cite[Lemma 15]{FLL}.
\begin{proposition}\label{L52} Under the conditions of \eqref{S111}, there is a positive constant $C $ depending only on $\Omega$, $\mu$, $\beta$, $\gamma$,  $\|\rho_0\|_{L^\infty}$ and $\|u_0\|_{H^1}$ such that
\begin{equation}\label{S505}
\sup_{0\leq t\leq T}\sigma B^2(t) + \int^T_0 \sigma\| \nabla\dot{u}\| _{L^2}^2dt \leq C,
\end{equation}
where $\sigma(t)\triangleq\min\{1, t\}$.
\end{proposition}

\begin{proof}
We first give a Poincar\'{e} type estimate under the condition $u\cdot n|_{\partial\Omega}=0$, whose proof can be found in \cite{caili01}.
For any $1\leq p<\infty$, there exists a positive constant $C$ depending on $p$ and $\Omega$ such that
\begin{equation}\label{S501}
\begin{split}
\|\dot{u}\|_{L^p}&\leq C(\|\nabla\dot{u}\|_{L^2}+\| \nabla u\|^2_{L^2}),\\
\|\nabla\dot{u}\|_{L^2}&\leq C(\|\mathrm{div}\dot{u}\|_{L^2} + \|\mathrm{curl}\dot{u}\|_{L^2}+
\|\nabla u\|_{L^4}^2).
\end{split}
\end{equation}
In particular, we set $p=4$ in Proposition \ref{L43} which together with \eqref{S111} gives
\begin{equation}\label{S514}
\|\nabla u\|_{L^4}\leq C\bigg(A+B\bigg).
\end{equation}

Operating $\dot{u}_j[\partial/\partial t + \mathrm{div}(u\cdot)]$ to $\eqref{S101}^j_2$,
summing with respect to $j$,  and  integrating the resulting equation over $\Omega$, we derive that
\begin{equation}\label{S507}
\begin{split}
\big(B^2\big)'
=&\tilde J_1+\tilde J_2+\tilde J_3,
\end{split}
\end{equation}
where the right side of \eqref{S507} is given by
\begin{equation*}
\begin{split}
\tilde{J}_1&=\int_\Omega\bigg(\dot{u}\cdot\nabla G_t +\dot{u}_j\cdot\mathrm{div}\big(\partial_jG u\big)\bigg)dx,\\
\tilde{J}_2&=\mu\int_\Omega\bigg(\dot{u}\cdot\nabla^\bot\omega_t +\dot{u}_j\partial_k\big((\nabla^\bot\omega)_ju_k\big)\bigg)dx\\
\tilde{J}_3&=\int_\Omega\bigg(
\dot{u}\cdot\partial_t\big((\rho-\rho_s)\nabla f\big)
+\dot{u}_j\cdot\mathrm{div}\big(u(\rho-\rho_s)\partial_j f\big)\bigg)dx.
\end{split}
\end{equation*}

The calculations involving $\tilde{J}_1$ and $\tilde{J}_2$ are quite routine, so we just sketch the process, while the complete details can be found in \cite[Lemma 15]{FLL} and \cite{caili01}. Let us deduce that  
\begin{equation}\label{CC432}
\begin{split}
\|\mathrm{div}\dot{u}\|_{L^2}^2+\tilde{J}_1&\leq \int_{\partial\Omega}\frac{d}{dt}G\cdot (\dot{u}\cdot n)dS+C\bigg(B^4+B^2+A^2\bigg)+\delta\|\nabla\dot{u}\|_{L^2}^2\\
&\leq\frac{d}{dt}\int_{\partial\Omega}CG(u\cdot\nabla n\cdot u)dS+C\bigg(B^4+B^2+A^2\bigg)+2\delta\|\nabla\dot{u}\|_{L^2}^2,
\end{split}
\end{equation}
where in the last line, we have applied \eqref{CC416}, \eqref{S501}, \eqref{S514}, and \eqref{CCC401} dealing the boundary term. We mention that \eqref{CC361} together with \eqref{S111} gives
\begin{equation}\label{CC428}
\int_{\partial\Omega}CG\bigg(u\cdot\nabla n\cdot u\bigg)\,dS\leq CA^2+ \frac{B^2}{4}.
\end{equation}

Similar arguments applied to $\tilde{J}_2$ also gives
\begin{equation}\label{CC433}
 \|\mathrm{rot}\dot{u}\|_{L^2}^2+\tilde{J}_2\leq C\bigg(B^4+B^2+A^2\bigg)+2\delta\|\nabla\dot{u}\|_{L^2}^2.
\end{equation}

To treat the new material $\tilde J_3$, we make use of \eqref{S101}$_1$, \eqref{S111}, and the boundary condition \eqref{S105} to argue that
\begin{equation}\label{S517}
\begin{split}
\tilde J_3
&=-\int_\Omega\mathrm{div}(\rho_s u)\cdot\dot{u}\cdot\nabla fdx+\int_\Omega(\rho-\rho_s)\dot{u}_j\cdot u_i\cdot\partial_{ij}f dx\\
&\leq C \bigg(\|\dot{u}\|_{L^4}+\|\nabla\dot{u}\|_{L^2}\bigg)\|u\|_{H^1}\\
&\leq \frac{1}{2}\|\nabla\dot{u}\|_{L^2}^2+C \|\nabla u\|_{L^2}^2,
\end{split}
\end{equation}
where the last line is due to \eqref{CC426} and  \eqref{S501}.

Substitute \eqref{CC432}, \eqref{CC433}, and \eqref{S517} into \eqref{S507} and setting $\delta<(4C)^{-1}$, we apply \eqref{S501} to infer that
\begin{equation}\label{S518}
\begin{split}
&\big(B^2\big)'
\leq -\frac{d}{dt}\int_{\partial\Omega}CG\bigg(u\cdot\nabla n\cdot u\bigg)\,dS
+C\bigg(B^4+B^2+A^2\bigg).
\end{split}
\end{equation}

Multiplying \eqref{S518} by $\sigma$, one gets \eqref{S505} after using \eqref{CC416}, \eqref{CC429}, \eqref{CC428}, and Gronwall's inequality.
\end{proof}

Now we establish \eqref{S112} and finish the main Theorem \ref{T11}.

{\it{ Proof of \eqref{S112}.}}
According to \eqref{S111}, \eqref{S412}, \eqref{CC416}, and \eqref{S505}, for $1\leq T<\infty$, we argue that
\begin{equation}\label{CC430}
\sup_{1\leq t\leq T}\|\nabla u\|_{L^p}\leq C\bigg(1+A+B\bigg)\leq C.
\end{equation}
By interpolation arguments, \eqref{CC430} together with \eqref{CC416} ensures for any $1\leq p<\infty$, there is a constant $\xi$ depending on $p$, such that
$$\sup_{0\leq t\leq T}\bigg(e^{\xi t}\|\nabla u\|_{L^p}\bigg)\leq C,$$
which together with \eqref{CC416} implies \eqref{S112} and finishes the proof.
\thatsall

\section{Appendix I: Periodic domain}
\quad We prove Theorem \ref{T12} for periodic domain $\Omega=[0,1]\times[0,1]$ without the external force $\nabla f=0$. The outline is similar with the case of the Navier-slip boundary conditions, thus we only focus on the different parts.

Observe that in periodic domain, we only have \eqref{S202},
$$\|u\|_{L^p}\leq Cp^{1/2}\|u\|_{L^2}^{2/p}\|u\|_{H^1}^{1-2/p},$$
where $\|u\|_{H^1}$ can \textbf{not} be replaced by $\|\nabla u\|_{L^2}$. Consequently, we lack the global integrability on $\|u\|_{L^p}$ which turns out to be vital for uniform estimates.

Instead of it, since $\nabla f=0$, we integrate \eqref{S101}$_2$ over $\Omega\times[0,t]$ directly and get
\begin{equation}\label{C501}
\int\rho u\,dx=\int\rho_0u_0\,dx=C.
\end{equation}
Thus $\overline{\rho u}$ remains constant. Moreover,
we assume $\overline{\rho}=1$ and apply Poincar\'{e} inequality to deduce that
\begin{equation}\label{C502}
\begin{split}
\int|u-\overline{\rho u}|^p\,dx&\leq C\int|u-\overline{u}|^p\,dx+C\,\big|\overline{\rho u}-\overline{u}\big|^p\\
&=C\int|u-\overline{u}|^p\,dx+C\,\bigg|\int(\rho-\overline{\rho})\cdot(u-\overline{u})\,dx\bigg|^p\\
&\leq C\,\big(\|\rho\|_{L^\gamma}^p+1\big)\cdot\big(\int|\nabla u|^2\,dx\big)^{2/p},
\end{split}
\end{equation}
where the second line is due to $\int(\rho-\overline{\rho})\cdot\overline{u}\,dx=0.$ We comment that \eqref{C501} and \eqref{C502} play key roles in deriving \eqref{S111} and \eqref{S112} for periodic domain.

The rest of Appendix I is organized as follows: we first establish crucial Corollary \ref{L33} for periodic domain, then we extend Proposition \ref{L35}--\ref{L44} as well, finally we turn to the proof of Theorem \ref{T12}.

\subsection{Key assertions for periodic domain}
\quad In the same notations as Section \ref{SEC3}, we first modify calculations in Proposition \ref{P31}.
Note that $\Delta^{-1}$ is well defined on $\Omega$ for functions with zero average, thus by taking $\Delta^{-1}\mathrm{div}$ on \eqref{S101}$_2$, we deduce that
$$F-\bar{F}=\frac{d}{dt}F_1-\big[u_i,R_{ij}\big]\,(\rho u_j),$$
where $F\triangleq(2\mu+\lambda)\mathrm{div}u-P$ and $F_1\triangleq\Delta^{-1}\mathrm{div}(\rho u)$, while the commutator is given by
\begin{equation*}
\begin{split}
\big[u_i,R_{ij}\big]\,(\rho u_j)
&\triangleq u\cdot\nabla\Delta^{-1}\mathrm{div}(\rho u)-\Delta^{-1}\mathrm{div}\,\mathrm{div}(\rho u\otimes u).
\end{split}
\end{equation*}
Consequently, \eqref{S314} is transformed into
\begin{equation}\label{C503}
\frac{d}{dt}(\theta(\rho)+F_1)+P=\big[u_i,R_{ij}\big]\,(\rho u_j)-\bar{F}.
\end{equation}
Let us still set $\theta(\rho)\triangleq 2\mu\log\rho+\rho^\beta/\beta$ and $\Phi\triangleq (\theta(\rho)+F_1-M)_+$.

Multiplying \eqref{C503} by $\rho\Phi^{\alpha-1}$ and integrating over $\Omega$, the same calculations shown in Lemma \ref{L32} lead to
\begin{equation}\label{C504}
\frac{d}{dt}\int_\Omega\rho\Phi^\alpha dx+\int_\Omega\rho(\Phi+M)^{\frac{\gamma}{\beta}}\Phi^{\alpha-1}dx
\leq C\int_\Omega\rho \tilde{H}\cdot \Phi^{\alpha-1}dx,
\end{equation}
while this time, $\tilde{H}$ is simply given by
$$\tilde{H}\triangleq |F_1|^{\frac{\gamma}{\beta}}+\big|[u_i,R_{ij}]\,(\rho u_j)\big|+|\bar{F}|.$$

Now we declare that the right side of \eqref{C504} can be handled by
\begin{proposition}
For the periodic domain $\Omega$, under the setting of Lemma \ref{L32}, for $\alpha\geq 3$ and $\varepsilon>0$, there exist constants $M$ and $C$ depending on $\varepsilon$, $\alpha,$ $\beta,$ $\gamma,$ and $\mu$, such that
\begin{equation}\label{C505}
\int_\Omega\rho \tilde{H}\cdot \Phi^{\alpha-1}dx\leq
CA_1^2 \left(\int_\Omega\rho\Phi^\alpha dx+1\right)+\varepsilon\int_\Omega\rho(\Phi+M)^{\frac{\gamma}{\beta}}\Phi^{\alpha-1}dx.
\end{equation}
\end{proposition}
\begin{proof}
Note that
\begin{equation*}
\begin{split}
&\int_\Omega\rho\tilde{H}\cdot \Phi^{\alpha-1} dx
=\int_\Omega\rho\bigg(|F_1|^{\frac{\gamma}{\beta}}+\big|[u_i,R_{ij}]\,(\rho u_j)\big|+|\bar{F}|\bigg)\Phi^{\alpha-1} dx\triangleq \sum_{k=1}^3J_k.
\end{split}
\end{equation*}
We set $\alpha\geq 3$ and check each terms in details.

\textit{Estimates of $J_1$.}
Such term can be split into two parts by
\begin{equation}\label{C506}
\begin{split}
J_1=\int_\Omega\rho\big|\Delta^{-1}\mathrm{div}\big(\rho(u-\overline{\rho u})\big)\big|\Phi^{\alpha-1}dx+\int_\Omega\rho\big|\Delta^{-1}\mathrm{div}(\rho\,\overline{\rho u})\big|\Phi^{\alpha-1}dx.
\end{split}
\end{equation}
In view of \eqref{C502} and \eqref{CC315}--\eqref{CC313}, we argue that
\begin{equation}\label{C507}
\mbox{The first term of \eqref{C506}}\leq CA_1^2 \left(\int_\Omega\rho\Phi^\alpha dx+1\right)+\varepsilon\int_\Omega\rho(\Phi+M)^{\frac{\gamma}{\beta}}\Phi^{\alpha-1}dx.
\end{equation}

Next, let us turn to the second term of \eqref{C506}. Taking $p=\beta(\alpha-1+\gamma/\beta^2)>2$, by means of \eqref{CC306}, we have
\begin{equation}\label{C508}
\begin{split}
\int_\Omega\rho\big|\Delta^{-1}\mathrm{div}(\rho\,\overline{\rho u})\big|\Phi^{\alpha-1}dx
&\leq \overline{\rho u}\cdot\|\Delta^{-1}\nabla\rho\|_{L^\infty}^{\frac{\gamma}{\beta}}\int\rho\Phi^{\alpha-1}dx\\
&\leq C\,\bigg(\int_\Omega\rho^{p\beta+1}dx\bigg)^{\frac{\gamma}{p\beta^2}}\cdot\int\rho\Phi^{\alpha-1}dx\\
&\leq C\,\bigg(\int_\Omega\rho\Phi^{p}dx+M^p\bigg)^{\frac{\gamma}{p\beta^2}}\int\rho\Phi^{\alpha-1}dx\\
&\leq \varepsilon\int\rho(\Phi+M)^{\frac{\gamma}{\beta}}\Phi^{\alpha-1}dx,
\end{split}
\end{equation}
provided $M>(C/\varepsilon)^{{\beta^2}/{(\gamma\beta-\gamma)}}$. Combining \eqref{C507} and \eqref{C508}, we argue that
$$J_1\leq CA_1^2 \left(\int_\Omega\rho\Phi^\alpha dx+1\right)+\varepsilon\int_\Omega\rho(\Phi+M)^{\frac{\gamma}{\beta}}\Phi^{\alpha-1}dx.$$

\textit{Estimates of $J_2$.} According to \cite[Lemma 2.5]{HL}, it holds that
\begin{equation}\label{C509}
\big\|[u_i,R_{ij}](f_j)\big\|_{L^p}\leq C(p)\,\|\nabla u\|_{L^2}\|f\|_{L^p}.
\end{equation}

Let us rewrite $J_2$ as
\begin{equation}\label{C510}
\begin{split}
J_2=\int_\Omega\rho\big|[u_i,R_{ij}](\rho\,\overline{\rho u})\big|\Phi^{\alpha-1}dx+
\int_\Omega\rho\big|[u_i,R_{ij}]\big(\rho(u-\overline{\rho u})\big)\big|\Phi^{\alpha-1}dx.
\end{split}
\end{equation}

Setting $q=\beta(\alpha-1+1/\beta)>2$, the first term of \eqref{C510} is bounded by
\begin{equation*}
\begin{split}
&\int_\Omega\rho\big|[u_i,R_{ij}](\rho\,\overline{\rho u})\big|\Phi^{\alpha-1}dx\\
&\leq \bigg(\int_\Omega\rho|[u_i,R_{ij}](\rho\,\overline{\rho u})\big|^qdx\bigg)^{1/q}\bigg(\int_\Omega\rho\Phi^{\alpha-1+1/\beta}dx\bigg)^{1-1/q}\\
&\leq\bigg(\int_\Omega\rho^{q+1}dx\bigg)^{1/(q^2+q)}\bigg(\int_\Omega[u_i,R_{ij}](\rho\,\overline{\rho u})\big|^{q+1}dx\bigg)^{1/(q+1)} \bigg(\int_\Omega\rho\Phi^{\alpha-1+1/\beta}dx\bigg)^{1-1/q}\\
&\leq C\|\nabla u\|_{L^2}\bigg(\int_\Omega\rho^{q+1}dx\bigg)^{1/q}\bigg(\int_\Omega\rho\Phi^{\alpha-1+1/\beta}dx\bigg)^{1-1/q},
\end{split}
\end{equation*}
where the last line is due to \eqref{C509}. By virtue of \eqref{CC306}, we declare that
\begin{equation}\label{C511}
\begin{split}
&\int_\Omega\rho\big|[u_i,R_{ij}](\rho\,\overline{\rho u})\big|\Phi^{\alpha-1}dx\\
&\leq C\|\nabla u\|_{L^2}\bigg(\int\rho\Phi^{\alpha-1+1/\beta}dx+C\bigg)^{1/q}\bigg(\int_\Omega\rho\Phi^{\alpha-1+1/\beta}dx\bigg)^{1-1/q}\\
&\leq C\|\nabla u\|_{L^2}\bigg(\int\rho\Phi^{\alpha-1+1/\beta}dx+C\bigg)^{1/2}\bigg(\int_\Omega\rho\Phi^{\alpha-1+1/\beta}dx\bigg)^{1/2}\\
&\leq C\|\nabla u\|_{L^2}^2\bigg(\int_\Omega\rho\Phi^{\alpha}dx+1\bigg)+\varepsilon\int_\Omega\rho(\Phi+M)^{\frac{\gamma}{\beta}}\Phi^{\alpha-1}dx,
\end{split}
\end{equation}
provided $M>(C/\varepsilon)^{{\beta}/(\gamma-1)}$.

Similarly, the second term of \eqref{C510} is controlled by
\begin{equation}\label{C512}
\begin{split}
&\int_\Omega\rho\big|[u_i,R_{ij}]\big(\rho(u-\overline{\rho u})\big)\big|\Phi^{\alpha-1}dx\\
&\leq\bigg(\int_\Omega\rho|[u_i,R_{ij}](\rho(u-\overline{\rho u}))\big|^qdx\bigg)^{1/q}\bigg(\int_\Omega\rho\Phi^{\alpha-1+1/\beta}dx\bigg)^{1-1/q}\\
&\leq C\bigg(\int_\Omega\rho^{q+1}dx\bigg)^{1/(q^2+q)}\bigg(\|\nabla u\|_{L^2}\|\rho(u-\overline{\rho u})\|_{L^{q+1}}\bigg) \bigg(\int_\Omega\rho\Phi^{\alpha-1+1/\beta}dx\bigg)^{1-1/q}\\
&\leq C\,\|\nabla u\|_{L^2}^2\bigg(\int_\Omega\rho\Phi^{\alpha}dx+1\bigg),
\end{split}
\end{equation}
where the last line is due to \eqref{C502} and the fact $p/\beta=\alpha-1+1/\beta<\alpha$.
Combining \eqref{C511} and \eqref{C512}, we arrive at
$$J_2\leq CA_1^2 \left(\int_\Omega\rho\Phi^\alpha dx+1\right)+\varepsilon\int_\Omega\rho(\Phi+M)^{\frac{\gamma}{\beta}}\Phi^{\alpha-1}dx.$$

\textit{Estimates of $J_3$}. According to \eqref{S326}, we directly deduce that
$$J_3\leq CA_1^2 \left(\int_\Omega\rho\Phi^\alpha dx+1\right)+\varepsilon\int_\Omega\rho(\Phi+M)^{\frac{\gamma}{\beta}}\Phi^{\alpha-1}dx.$$

Collecting all estimates on $J_i$, we argue that
$$\int_\Omega\rho \tilde{H}\cdot \Phi^{\alpha-1}dx\leq
CA_1^2 \left(\int_\Omega\rho\Phi^\alpha dx+1\right)+\varepsilon\int_\Omega\rho(\Phi+M)^{\frac{\gamma}{\beta}}\Phi^{\alpha-1}dx,$$
which is \eqref{C505}. The proof is therefore completed.
\end{proof}

Substituting \eqref{C505} into \eqref{C504}, we apply the same method in Corollary \ref{L33} to declare
\begin{proposition}\label{P52}
For any $3\leq\alpha<\infty$, there are constants $C$ and $M_1$ depending only on $\mu$, $\alpha$, $\beta$, $\gamma$, $\rho_0$, $u_0$, and $\Omega$, such that
\begin{equation}\label{C513}
\sup_{0\leq t\leq T}\|\rho\|_{L^{\alpha\beta+1}}+\int_0^T\int_\Omega(\rho-M_1)_+^{\alpha-1} dxdt\leq C.
\end{equation}
\end{proposition}
In view of Proposition \ref{L34}, by making use of Proposition \ref{P52}, we also have
\begin{proposition}\label{P53}
Under the hypothesis of Theorem \ref{T12}, suppose \eqref{CC101} is valid and $(\rho, u)$ is the   strong solution given by Lemma \ref{L21}, then for any $1\leq p<\infty$, we have
\begin{equation}\label{C514}
\lim_{t\rightarrow\infty}\|\rho-\bar{\rho}\|_{L^{p}}=0.
\end{equation}
\end{proposition}
In light of \eqref{C514}, for some $\delta>0$ small enough, we may assume
\begin{equation}\label{C515}
\|\rho-\bar{\rho}\|_{L^4}\leq\delta,
\end{equation}
in the rest of this appendix. With the help of \eqref{C513} and \eqref{C515}, we can obtain crucial $L^2$ global integrability of $\|\rho-\bar{\rho}\|_{L^2}$.

\begin{proposition}\label{P54}
In the same setting as Proposition \ref{P53}, suppose that \eqref{C515} is valid, then there is a constant $C$ determined by $\mu$, $\beta$, $\gamma$, $\rho_0$, $u_0$, and $\Omega$, such that
\begin{equation}\label{C516}
\int_0^T\int_\Omega(\rho+1)^{\gamma-1}(\rho-\bar{\rho})^2\,dxdt\leq C.
\end{equation}
\end{proposition}
\begin{proof}
The outline is similar with Proposition \ref{L35}. Let us rewrite \eqref{S101}$_2$ as
\begin{equation}\label{C517}
\nabla(\rho^\gamma-\bar{\rho}^\gamma)=-\rho(u-\overline{\rho u})_t-\rho u\cdot\nabla u+\mu\Delta u+\nabla\big((\mu+\lambda)\mathrm{div}u\big).
\end{equation}
Multiplying \eqref{C517} by $\nabla\Delta^{-1}(\rho-\bar{\rho})$ and integrating over $\Omega$, we deduce that
\begin{equation}\label{C518}
\begin{split}
\int_\Omega(\rho^\gamma-\bar{\rho}^\gamma)(\rho-\bar{\rho})\,dx=\int_\Omega\rho(u-\overline{\rho u})_t\cdot\nabla\Delta^{-1}(\rho-\bar{\rho})\,dx+\mathcal{R}(x,t).
\end{split}
\end{equation}

In view of \eqref{S301}, \eqref{CC340}--\eqref{CC341} and \eqref{C513}, the remaining term $\mathcal{R}(x,t)$ can be bounded by
\begin{equation}\label{C519}
\mathcal{R}(x,t)\leq C\int_\Omega(\rho-M)_{+}^qdx+C\|\nabla u\|_{L^2}^2+\varepsilon\|\rho-\bar{\rho}\|_{L^2}^2,
\end{equation}
where $q=2\beta(\gamma+1)/(\gamma-1)>2$ and $\varepsilon>0$ will be determined later.

We focus on the first term of right side of \eqref{C518}. Observe that
\begin{equation*}
\begin{split}
&\int_\Omega\rho(u-\overline{\rho u})_t\cdot\nabla\Delta^{-1}(\rho-\bar{\rho})\,dx=\frac{d}{dt}\int_\Omega\rho(u-\overline{\rho u})\cdot\nabla\Delta^{-1}(\rho-\bar{\rho})\,dx+P_1+P_2,\\
\end{split}
\end{equation*}
where the remaining terms are given by
\begin{align*}
P_1&\triangleq\int_\Omega\mathrm{div}(\rho u)(u-\overline{\rho u})\cdot\nabla\Delta^{-1}(\rho-\bar{\rho})\,dx,\\
P_2&\triangleq\int_\Omega\rho(u-\overline{\rho u})\nabla\Delta^{-1}\mathrm{div}(\rho u)\,dx.
\end{align*}

Observe that by means of \eqref{S301} and \eqref{C513}, we also check that
\begin{equation}\label{C520}
\sup_{0\leq t\leq T}\int_\Omega\rho(u-\overline{\rho u})\cdot\nabla\Delta^{-1}(\rho-\bar{\rho})\,dx\leq C.
\end{equation}

Moreover, in view of \eqref{S202}, \eqref{C502}, and \eqref{C513}, we deduce that
\begin{equation}\label{C521}
\begin{split}
P_1&=-\int_\Omega\rho u\cdot\nabla\big((u-\overline{\rho u})\cdot\nabla\Delta^{-1}(\rho-\bar{\rho})\big) dx,\\
&\leq C\,\|\rho u\|_{L^4}\|\nabla u\|_{L^2}\|\rho-\bar{\rho}\|_{L^2}\\
&\leq C\|\nabla u\|_{L^2}^2+\varepsilon\|\rho-\bar{\rho}\|_{L^2}^2.
\end{split}
\end{equation}

Finally, to handle $P_2$, we check that
\begin{equation}\label{C522}
\begin{split}
P_2&=\int_\Omega(\rho-\bar{\rho})\,(u-\overline{\rho u})\cdot\nabla\Delta^{-1}\mathrm{div}(\rho u)\,dx-\bar{\rho}\int_\Omega\mathrm{div}u\cdot\Delta^{-1}\mathrm{div}(\rho u)\,dx\\
&\leq C\,\|\nabla u\|_{L^2}^2+\varepsilon\|\rho-\bar{\rho}\|_{L^2}^2, 
\end{split}
\end{equation}
where we have taking advantage of \eqref{S301}, \eqref{C502}, \eqref{C513}, and the fact
$$\bar{\rho}\int_\Omega\mathrm{div}u\,\Delta^{-1}\mathrm{div}(\rho u)\,dx=\bar{\rho}\int_\Omega\mathrm{div}u\,\Delta^{-1}(\mathrm{div}\big(\bar{\rho}u+(\rho-\bar{\rho})u)\big)\,dx.$$

Now, let us integrate \eqref{C518} with respect to $t$ and apply \eqref{S301}, \eqref{C513}, \eqref{C519}--\eqref{C522} to deduce that
$$\int_0^T\int_\Omega(\rho+1)^{\gamma-1}(\rho-\bar{\rho})^2\,dxdt\leq C\bigg(1+\varepsilon\,\|\rho-\bar{\rho}\|_{L^2(\Omega\times(0,T))}^2\bigg).$$
Setting $\varepsilon<1/2C$ gives \eqref{C516} and finishes the proof.
\end{proof}

We mention that Proposition \ref{P53} and \ref{P54} provide a crude description of the large time behavior of $\rho$. Then we establish further energy estimates in analogy with Proposition \ref{L41}.

\begin{proposition}
Under the same assumption of Theorem \ref{T12}, there exists a constant $C$ depending on $\Omega$, $\mu$, $\beta$, $\gamma$, $\|\rho_0\|_{L^{\infty}}$, and $\|u_0\|_{H^1}$ such that
\begin{equation}\label{C523}
\sup_{0\leq t\leq T}\int_\Omega\rho |u-\overline{\rho u}| ^{2+\nu}dx\leq C ,
\end{equation}
where we set
\begin{equation*}
\nu\triangleq R_T^{-\frac{\beta}{2}}\nu_{0},
\end{equation*}
for some suitably small generic constant $\nu_0\in (0,1)$ depending only on $\mu$ and $\Omega.$
\end{proposition}
\begin{proof}
The proof is the same as Proposition \ref{L41}, except that we replace $u$ by $(u-\overline{\rho u})$.

Note that $\overline{\rho u}_t=0$ due to \eqref{C501}, thus we can write \eqref{S101}$_2$ as
\begin{equation}\label{C524}
\rho\,\frac{d}{dt}(u-\overline{\rho u})-\nabla(\mu+\lambda)\mathrm{div}u-\mu\Delta u+\nabla P=0.
\end{equation}

Multiplying \eqref{C524} by $(2+\nu)\,|u-\overline{\rho u}|^\nu\,(u-\overline{\rho u})$, we carry out the similar calculations given by \eqref{CC345}--\eqref{CC377} to argue that
$$\sup_{0\leq t\leq T}\int_\Omega\rho |u-\overline{\rho u}| ^{2+\nu}dx\leq C,$$
which is \eqref{C523} and finishes the proof.
\end{proof}

We end this section by the proposition below for the periodic domain, whose proof can be found in \cite[Lemma 3.1]{HL}.
\begin{proposition}\label{P56}
Under the same assumptions of Theorem \ref{T11}, for any $\varepsilon\in(0,1)$, there is a constant $C $ depending only on $\varepsilon$, $\mu$, $\gamma$,  $\rho_0$, $u_0$, and $\Omega$ such that
\begin{equation*}
\sup_{0\leq t\leq T}\mathrm{log}(C+ A^2) +\int^T_0 \frac{B^2}{1+A^2}dt
\leq C R_T^{1+\varepsilon }.
\end{equation*}
Recalling that $A, B,$ and $R_T$ are defined in \eqref{CC344}.
\end{proposition}

\subsection{Proof of Theorem \ref{T12}}
\quad Now we follow the same way as
Section \ref{S5} to finish Theorem \ref{T12}. The situation is even simpler than Navier-slip boundary condition, hence we only sketch the proof.

\textit{Proof of Theorem \ref{T12}.}
Note that \eqref{C503} provides that
$$\frac{d}{dt}\theta(\rho)+P=-\frac{d}{dt}F_1+[u_i,R_{ij}]\,(\rho u_j)-\bar{F}.$$
According to \cite[Lemma 2.5]{HL}, we can write the commutator as the singular integral
\begin{align}\label{C525}
[u_i,R_{ij}]\,(\rho u_j)(x)=p.v.\int_\Omega K_{ij}(x-y)\big(u_i(x)-u_i(y)\big)\cdot\rho u_j(y)\,dy,
\end{align}
where the kernel $K_{ij}(x)$ has a singularity of second order at 0 and
$$K_{ij}(x)\leq C|x|^{-2}~~x\in\Omega.$$

The representation \eqref{C525} coincides with the principal term $J(x,t)$ in \eqref{CC401}, see \cite[Proposition 4]{FLL}. Moreover, observe that \eqref{CC435}, \eqref{C513}, and \eqref{C523} ensure that
$$\sup_{0\leq t\leq T}|\bar{F}|\leq C,~~\sup_{0\leq t\leq T}|F_1| \leq CR_T^{1+\frac{\beta}{4}+\varepsilon}.$$

Consequently, with the help of Proposition \ref{P52}--\ref{P56}, we apply the same arguments as  Lemma \ref{L48}, Proposition \ref{L47}--\ref{L52} to derive \eqref{S111} and \eqref{S112}. The proof of Theorem \ref{T12} is therefore completed.
\thatsall

\section{Appendix II: Elliptic estimates}
\quad We provide a proof for Assertions 3) given by Lemma \ref{L23}. For complete discussions about the elliptic systems, see \cite{adn}.

\textbf{Lemma\  2.2.}
\textit{Suppose that $(F,\omega)$ satisfies}
\begin{equation}\label{CC508}
\begin{cases}
\nabla F+\nabla^\bot\omega=\mathrm{div}(u\otimes v)\ &\textit{in $\Omega$},\\
\int_\Omega Fdx=0,\ \omega=0\
&\textit{on $\partial\Omega$}.
\end{cases}
\end{equation}

3) \textit{If $u\cdot v\in L^p(\Omega)$, $\mathrm{div}(u\otimes v)=\partial_i(u_i\cdot v_j)$ with $u\cdot n=0$ on $\partial\Omega$, then we can find some constant $C(p)$ depending only on $\Omega$, such that}
\begin{equation}\label{CC503}
\| F\|_{L^p}+\| \omega\|_{L^p}
\leq C(p)\|u\cdot v\|_{L^p}.
\end{equation}

\begin{proof}
We will apply the duality arguments due to V. A. Solonikov to finish the proof. Let us operate $\nabla^\bot$ on \eqref{CC508} and find that
\begin{equation}\label{CC505}
\Delta\omega=\nabla^\bot \mathrm{div}(u\otimes v).
\end{equation}

First, for any $g\in C_0^\infty(\Omega)$, let us consider a Dirichlet problem,
\begin{equation*}
\begin{cases}
\Delta G_{1}=g\ &\mbox{in}\ \Omega,\\
G_{1}=0 \ &\mbox{on}\ \partial\Omega.\\
\end{cases}
\end{equation*}
The standard $L^p$ elliptic estimates (see \cite{gilbarg2015elliptic}) ensure that, for any $1<p<\infty$, we have
\begin{equation}\label{CC504}
\|\nabla^2G_{1}\|_{L^p}\leq C\|g\|_{L^p}.
\end{equation}

Now, we check that
\begin{equation*}
\begin{split}
\int_\Omega\omega\cdot g\, dx=\int_\Omega\omega\cdot\Delta G_{1}\, dx=\int_\Omega\Delta\omega\cdot G_{1}\, dx.
\end{split}
\end{equation*}
We mention that no boundary term occurs since $\omega=G_{1}=0$ on $\partial\Omega$. Thus, in view of \eqref{CC505}, we deduce
\begin{equation}\label{CC506}
\begin{split}
\int_\Omega\Delta\omega\cdot G_{1}\, dx&=\int_\Omega\nabla^\bot\mathrm{div}(u\otimes v)\cdot G_{1}\, dx\\
&=-\int_\Omega\mathrm{div}(u\otimes v)\cdot\nabla^\bot G_{1}\, dx\\
&=\int_\Omega(u\otimes v)\cdot \nabla^2G_{1}\, dx,\\
\end{split}
\end{equation}
Once again, no boundary term occurs since $u\cdot n=G_1=0$ on $\partial\Omega$.

Consequently, by virtue of \eqref{CC504} and \eqref{CC506}, we declare
$$\bigg|\int_\Omega\omega\cdot g\, dx\bigg|=\bigg|\int_\Omega(u\otimes v)\cdot\nabla^2 G_1\, dx\bigg|
\leq C \|u\cdot v\|_{L^p}\|g\|_{L^{p/(p-1)}},$$
which in particular implies
\begin{equation}\label{CC507}
\|\omega\|_{L^p}\leq C\|u\cdot v\|_{L^p}.
\end{equation}

Next, for $g\in C_0^\infty(\Omega)$, let us consider a Neumann problem
\begin{equation}\label{CC509}
\begin{cases}
\Delta G_{2}=g-\bar{g}\ &\mbox{in}\ \Omega,\\
n\cdot\nabla G_{2}=0 \ &\mbox{on}\ \partial\Omega.\\
\end{cases}
\end{equation}
The system \eqref{CC509} is solvable due to $\int_\Omega(g-\bar{g})\, dx=0$, and the standard $L^p$ elliptic estimates (see \cite{gilbarg2015elliptic}) guarantee that, for any $1<p<\infty$, we have
\begin{equation}\label{CC510}
\|\nabla^2G_{2}\|_{L^p}\leq C\|g\|_{L^p}.
\end{equation}

Observe that $\int_\Omega F\, dx=0$ implies
\begin{equation*}
\begin{split}
\int_\Omega F\cdot g\, dx=\int_\Omega F\cdot(g-\bar{g})\, dx=\int_\Omega F\cdot\Delta G_2dx=\int_\Omega\nabla F\cdot\nabla G_2\, dx.
\end{split}
\end{equation*}
We mention that $n\cdot\nabla G_2=0$ on $\partial\Omega$ cancels the boundary term. Thus in view of \eqref{CC508}, we also deduce that
\begin{equation}\label{CC511}
\begin{split}
\int_\Omega\nabla F\cdot\nabla G_2\, dx
&=-\int_\Omega\nabla^\bot\omega\cdot\nabla G_2\, dx+\int_\Omega\mathrm{div}(u\otimes v)\cdot\nabla G_2\, dx\\
&=\int_\Omega (u\otimes v)\cdot\nabla^2 G_2\, dx,
\end{split}
\end{equation}
where we have made use of the fact $\nabla^\bot\cdot\nabla=0$. Moreover, there is no boundary term since $u\cdot n=\omega=0$ on $\partial\Omega$.

Combining \eqref{CC510} with \eqref{CC511} leads to
\begin{equation*}
\bigg|\int_\Omega F\cdot g\, dx\bigg|=\bigg|\int_\Omega (u\otimes v)\cdot\nabla^2 G\, dx\bigg|
\leq C\|u\cdot v\|_{L^p}\|g\|_{L^{p/(p-1)}},
\end{equation*}
which is equivalent to
\begin{equation}\label{CC513}
\|F\|_{L^p}\leq C\|u\cdot v\|_{L^p}.
\end{equation}

Collecting \eqref{CC507} and \eqref{CC513} gives \eqref{CC503} and finishes the proof.
\end{proof}


\begin {thebibliography} {99}


\bibitem{adn}
S. Agmon, A. Douglis, L. Nirenberg,
Estimates near the boundary for solutions of elliptic partial differential equations satisfying general boundary conditions II.
Commun. Pure Appl. Math.,
\textbf{17}(1) (1964),
35-92.

\bibitem{Aramaki2014Lp}
  J.  Aramaki,
  $L^p$ theory for the div-curl system.
  Int.  J.  Math.  Anal.,
  {\bf 8}(6) (2014),
  259-271.

\bibitem{BD1}
D. Bresch, B. Desjardins,
Existence of global weak solutions for a 2D viscous
shallow water equations and convergence to the quasi-geostrophic model.
Commun. Math. Phys.,
{\bf 238} (2003), 211-223.

\bibitem{BD2}
D. Bresch, B. Desjardins,
 Some diffusive capillary models for Korteweg type. C. R. Mecanique,
{\bf 331} (2003).

\bibitem{BD3}
D. Bresch, B. Desjardins, C. K. Lin,
On some compressible fluid models: Korteweg, lubrication, and shallow water systems.
Commun. Partial Differential Equations,
{\bf 28} (2003), 843-868.

\bibitem{BD4}
D. Bresch, P. Noble,
Mathematical justification of a shallow water model. Methods Appl. Anal.,
{\bf 14} (2007), no. 2, 87-117.

\bibitem{BD5}
D. Bresch, P. Noble,
Mathematical derivation of viscous shallow-water equations with zero surface tension.
Indiana Univ. Math. J.,
{\bf 60} (2011), 1137-1169.

\bibitem{caili01}G. C. Cai,  J. Li, Existence and exponential growth of  global classical solutions to the  compressible Navier-Stokes equations with slip boundary conditions in 3D bounded domains. arXiv: 2102.06348.





\bibitem{2010Partial}L.  C.  Evans,
  Partial differential equations: Second edition,
  2010.


\bibitem{FLL}
X. Y. Fan, J. X. Li, J. Li,
Global existence of strong and weak solutions to 2D compressible Navier-Stokes system in bounded domains with large data and vacuum. arXiv: 2102.09229.

\bibitem{FP}
E.  Feireisl, H.  Petzeltov\'{a},
Large-time behaviour of solutions to the Navier-Stokes equations of compressible flow.
Arch. Ration. Mech. Anal.,
{\bf 150}(1) (1999),
77-96.


\bibitem{feireisl2004dynamics}
E.  Feireisl,   A. Novotny, H.  Petzeltov\'{a},
On the existence of globally defined weak solutions to
the Navier-Stokes equations.
J. Math. Fluid Mech.,
{\bf 3} (2001),
358-392.

\bibitem{Fri}
H. Frid, D. Marroquin, J. F. C. Nariyoshi,
Global smooth solutions with large data for a system modeling aurora type phenomena in the 2-torus.
SIAM J. Math. Anal.,
{\bf 53} (2021),
1122-1167.



\bibitem{gilbarg2015elliptic}D.  Gilbarg, N.  S.  Trudinger,
  Elliptic partial differential equations of second order,
  {Springer},
  2015.

\bibitem{LXY}
T. P. Liu, Z. P. Xin, T. Yang,
Vacuum states for compressible flow.
Discrete Contin. Dynam. Systems,
{\bf 4} (1998), 1-32.

\bibitem{1995Global} D.  Hoff,
  Global solutions of the Navier-Stokes equations for multidimensional compressible flow with discontinuous initial data.
  {J. Differ. Eqs.},
  {\bf 120}(1) (1995),
  215-254.

\bibitem{hoff2005} D.  Hoff,
  Compressible flow in a half-space with Navier boundary conditions.
  J.  Math.  Fluid Mech.,
  {\bf 7}(3) (2005),
  315-338.

\bibitem{HLX} F. M. Huang, J. Li, Z. P. Xin,
Convergence to equilibria and blowup behavior of global strong solutions to the Stokes approximation equations for two-dimensional compressible flows with large data.
J. Math. Pures Appl., {\bf 86} (2006),  471-491.


\bibitem{HL} X. D.  Huang, J.  Li,
  Existence and blowup behavior of global strong solutions to the two-dimensional barotropic compressible Navier-Stokes system with vacuum and large initial data.
  J.  Math.  Pures Appl.,   {\bf 106}(1) (2016),  123-154.


\bibitem{hl21} X. D. Huang,   J.  Li, Global well-posedness of classical solutions to the Cauchy problem of two-dimensional baratropic compressible Navier-Stokes system with vacuum and large initial data.
 SIAM J. Math. Anal., {\bf 54}(3) (2022),  3192-3214.


\bibitem{hlx21}X. Huang, J. Li, Z. P. Xin, Global well-posedness of classical solutions with large oscillations and vacuum to the three-dimensional isentropic compressible Navier-Stokes equations. Commun. Pure Appl. Math., {\bf 65} (2012), 549-585.


\bibitem{jwx}Q. S. Jiu,  Y. Wang,  Z. P. Xin,   Global well-posedness of 2D compressible Navier-Stokes equations with large data and vacuum. J. Math. Fluid Mech., {\bf 16} (2014),  483-521.

 \bibitem{jwx1} Q. S. Jiu,  Y. Wang,  Z. P. Xin,   Global classical solution to two-dimensional compressible Navier-Stokes equations with large data in $\mathbb{R}^2$. Phys. D, {\bf 376/377} (2018), 180-194.

\bibitem{LXZ}
H. L. Li, Z. P. Xin, X. W. Zhang,
Existence and uniqueness of global strong solutions to a free boundary value problem for 2D spherically symmetric compressible Navier-Stokes equations.
Phys. D,
{\bf 377} (2018), 185-199.






\bibitem{lx01} J. Li, Z. P. Xin, Global well-posedness and large time asymptotic
behavior of classical solutions to the compressible
Navier-Stokes equations with vacuum. Annals of PDE, {\bf 5}  (2019), 7.


\bibitem{1998Mathematical} P.  L.  Lions,
  Mathematical topics in fluid mechanics, vol.  2, Compressible models, Oxford University Press, New York, 1998.



\bibitem{1980The}A.  Matsumura, T.  Nishida,
  The initial value problem for the equations of motion of viscous and heat-conductive gases.
  J. Math. Kyoto Univ.,
  {\bf 20}(1) (1980),
  67-104.

\bibitem{1992Stability}
A. Matsumura, M. Padula,
Stability of stationary flow of compressible fluids subject to large external potential forces.
Stab. Appl. Anal. Cont. Media,
{\bf 2} (1992),
182-202


\bibitem{Mitrea2005Integral}
  D.  Mitrea,
  Integral equation methods for div-curl problems for planar vector fields in nonsmooth domains.
  Differ.  Int.  Equ.,
 {\bf 18}(9) (2005),
  1039-1054.

\bibitem{Mu}
F. Murat,
Compacité par compensation.
Ann. Scuola Norm. Sup. Pisa Cl. Sci.,
\textbf{5}(3) (1978),
489-507.

\bibitem{Nash}
J. Nash,
Le problme de Cauchy pour les ´equations diff´erentielles d’un fluide g´en´eral.
Bull. Soc. Math. France,
{\bf 90} (1962),
487-497.


\bibitem{NS}
A. Novotny, I. Strakraba,
Stabilization of weak solutions to compressible Navier-Stokes equations.
Commun. Math. Phys.,
{\bf 40}(2) (2000),
169-204.

\bibitem{PP}
M. Perepelitsa,
 On the global existence of weak solutions for the Navier-Stokes equations of compressible
fluid flows.
SIAM J. Math. Anal.,
{\bf 38} (2006), 1126-1153.


\bibitem{2016Representation}  M.  A.  Sadybekovand, B.  T.  Torebek, B.  Kh Turmetov,
  Representation of Green's function of the Neumann problem for a multi-dimensional ball.
  {Complex Variables and Elliptic Equations},
  {\bf 61}(1) (2016),
  104-123.




\bibitem{ser1}J. Serrin, On the uniqueness of compressible fluid motion. Arch. Ration. Mech. Anal., {\bf 3} (1959), 271-288.



\bibitem{2007Complex}E. M.  Stein,  R.  Shakarchi, Complex analysis,  Princeton University Press,  2003.

\bibitem{tal1}G. Talenti,
Best constant in Sobolev inequality.  Ann. Mat. Pura Appl., {\bf 110} (1976), 353-372.

\bibitem{Tar}
L. Tartar,
Compensated compactness and applications to partial differential equations.
Nonlinear Anal. and  Mech., Heriot-Watt Symposiumx,
(1979), 136-212.



\bibitem{Wahl1992Estimating}
W. Von Wahl,
Estimating $u$ by div$u$ and curl$u$.
Math.  Methods  Appl.  Sci.,
  {\bf 15}(2) (1992),
  123-143.

\bibitem{vaigant1995}
V. A. Vaigant, A. V. Kazhikhov,
On existence of global solutions to the two-dimensional Navier-Stokes equations for a compressible viscous fluid.
  Sib.  Math.  J.,
  {\bf 36}(6) (1995),
  1108-1141.


\bibitem{Zho}
Y. Zhou,
An $L^{p}$ theorem for compensated compactness.
Proc. Roy. Soc. Edinburgh Sect.,
\textbf{122}(1-2) (1992),
177-189.

\bibitem{zlt}
A. A. Zlotnik,
Uniform estimates and stabilization of symmetric solutions of a system of quasilinear equations.
Differ. Equ.,
\textbf{36}(5) (2000),
701-716.



\end{thebibliography}
\end{document}